\documentclass[]{article}

\usepackage[style=alphabetic, backend=biber, url=false, doi=false, isbn=false, eprint=false, maxbibnames=99, maxcitenames=4]{biblatex}
\usepackage{graphicx}
\usepackage{amsmath}
\usepackage{amssymb}
\usepackage{amsthm} 
\usepackage{bm}
\usepackage{enumerate}
\usepackage{color}
\usepackage{mathdots}
\usepackage{sectsty}
\usepackage{hyperref}
\hypersetup{hidelinks}
\usepackage{tikz}
\usepackage{caption}
\usepackage{adjustbox}
\usepackage{yhmath}
\usepackage{tikz-cd}
\usepackage{mathrsfs}
\usepackage{quiver}
\sectionfont{\scshape\centering\fontsize{12}{14}\selectfont}
\subsectionfont{\scshape\fontsize{12}{14}\selectfont}
\usepackage{fancyhdr}

\newcommand\shorttitle{Ind-Banach approach to Grothendieck duality}
\newcommand\authors{Arun Soor}

\newtheorem{thm}{Theorem}[section]
\newtheorem{cor}[thm]{Corollary}
\newtheorem{lem}[thm]{Lemma}
\newtheorem{defn}[thm]{Definition}
\newtheorem{rmk}[thm]{Remark}
\newtheorem{example}[thm]{Example}

\newtheorem{prop}[thm]{Proposition}

\newtheorem{notation}[thm]{Notation}
\newtheorem{construction}[thm]{Construction}

\newtheorem*{thm*}{Theorem}

\newtheorem{warning}[thm]{Warning}
\newtheorem*{claim*}{Claim}
\newcounter{egapar}[section]

\renewcommand{\theegapar}{\thesection.\arabic{egapar}}

\newcommand{\egapar}{%
  \par\vspace{0.5\baselineskip}\noindent 
  \refstepcounter{egapar}
  \textbf{\theegapar.} 
  \ignorespaces 
}

\newcommand*{\sheafhom}{\mathcal{H}\kern -.5pt om}

\newcommand{\Addresses}{{
  \bigskip
  \footnotesize

  Arun Soor, \textsc{Department of Mathematics, University of Michigan, 4831 East Hall, 530 Church Street, Ann Arbor, MI 48109}\par\nopagebreak
  \textit{E-mail address}, Arun Soor, \texttt{\href{mailto:soor@umich.edu}{soor@umich.edu}}}}

\title{\large \bf Ind-Banach approach to Grothendieck duality in Rigid-analytic geometry}
\author{Arun Soor}
\date{\today}
\begin{document}
\maketitle
\begin{abstract}
We prove a duality theorem for quasi-compactly supported cohomology of quasi-coherent sheaves on rigid-analytic spaces, with respect to a smooth and Kiehl partially-proper morphism. This includes an identification of the dualizing object with volume forms. The functional analysis underlying our theory does not use condensed mathematics, but rather Ind-Banach spaces, following Ben-Bassat--Kelly--Kremnizer.
\end{abstract}
\tableofcontents
\section{Introduction}
\egapar I will start by providing some historical background. For the purposes of this introduction only, $F$ denotes an abstract field, $\mathbf{C}$ denotes the field of complex numbers and $K$ denotes a non-Archimedean Banach field with non-trivial valuation. \emph{No assumptions are made on the characteristic of $F$ or $K$} and additionally \emph{we do not assume that $K$ is discretely valued}.  If $X$ is a smooth proper complex analytic manifold, rigid analytic space over $K$, or algebraic variety over $F$, of dimension $d$ then there is a canonical isomorphism\footnote{Here the $\operatorname{Ext}$-group is computed in the category of coherent sheaves on $X$.} $\operatorname{Ext}^i_X(M, \Omega_X^d) \cong H^{d-i}(M)^\lor$ for every $i \geqslant 0$ and every coherent sheaf $M$ on $X$. In the complex and algebraic settings this is due to Serre \cite{SerreComplex,SerreICM}, and consequently, this isomorphism is known as Serre duality; in the rigid-analytic setting, this is due to Chiarelotto, van der Put, and Beyer \cite{ChiarellottoDuality,vanderPutSerre,BeyerSerre}. This is the generalization to higher dimensions of much older results on the duality of linear series, on Riemann surfaces due to Roch in 1865 \cite{Roch}, and on algebraic curves due to Hasse \cite{HasseDuality} and Chevalley \cite{ChevalleyIntroduction}. Serre duality can be naturally phrased in terms of the derived category. Let $f: X \to *$ be the structure morphism and let $f_*$ be the derived pushforward functor. Then there is a canonical equivalence $\operatorname{Hom}_X(M,\Omega^d_X[d]) \simeq (f_*M)^\lor$. (Our convention is that all functors appearing here are implicitly derived). In the case when $X$ is an algebraic variety, Grothendieck proved the following, the details of which were provided in Hartshorne's book \emph{Residues and Duality} \cite{hartshorne_residues_1966}. Our convention is that $\operatorname{QCoh}(X)$ stands for the derived category of quasi-coherent sheaves.
\begin{thm}[Grothendieck]\label{thm:Intro1}
Let $f: X \to Y$ be a smooth proper morphism of algebraic varieties over $F$, of dimension $d$. Then for any $M \in \operatorname{QCoh}(X)$, $N \in \operatorname{QCoh}(Y)$, there is a canonical equivalence $\underline{\operatorname{Hom}}_Y(f_*M,N) \simeq f_*\underline{\operatorname{Hom}}_X(M, f^*N \otimes_X \Omega_{X/Y}^d[d])$ in $\operatorname{QCoh}(Y)$. (In other words, there is an adjunction $f_* \dashv f^* \otimes_X \Omega^d_{X/Y}[d]$). 
\end{thm}
\noindent The above equivalence is known as Grothendieck duality. We highlight that Grothendieck duality generalizes Serre duality in two important directions. Firstly, from the absolute to the relative setting: the base $Y$ is arbitrary. Secondly, no finiteness conditions are imposed on the coefficients: $M$ and $N$ are allowed to be any complexes of quasi-coherent sheaves.

\egapar Since Serre duality exists in both the algebraic and analytic settings, it is natural to ask if there is an analytic counterpart of Grothendieck duality. There are two main obstructions to achieving this goal. 

The first problem is that one must find the correct definition of $\operatorname{QCoh}$ in the analytic setting. Fortunately, this has been resolved by new theories of analytic geometry \cite{CondensedMathematics, CondensedComplexGeometry,DAnG}. In short, the correct definition $\operatorname{QCoh}$ must account for the analytic structure on the rings we work with. At the same time, it must also be derived: not only because this is the most natural way to formulate Grothendieck duality, but also because of the pathology that the Grothendieck topologies appearing in analytic geometry are not flat, and so descent for $\operatorname{QCoh}$ cannot hold without passing to the stable setting. The crux is therefore the question of how to synthesise functional analysis and homological algebra. Our definition of $\operatorname{QCoh}$ follows Ben-Bassat--Kelly--Kremnizer \cite{DAnG} and uses the theory of $\operatorname{Ind}$-Banach spaces, rather than the condensed mathematics of \cite{CondensedMathematics,CondensedComplexGeometry}.

The second problem is that, for applications in rigid geometry, duality for proper morphisms is not enough. Intuitively, the analytic topology is much finer than the Zariski topology and so being quasi-compact is a much stronger condition. For instance, the affine line $\mathbf{A}^1$, the open unit disk $\mathring{\mathbf{D}}^1$, and the upper half-plane are not quasi-compact, but we should still expect to have a duality theory for these. For a morphism $f: X \to Y$ of rigid or complex-analytic spaces subject to some mild hypotheses, we need to define the functor $f_!: \operatorname{QCoh}(X) \to \operatorname{QCoh}(Y)$ of pushforward with \emph{relative} quasi-compact supports, in such a way that $f_!$ is compatible with base-change and satisfies the projection formula. This is a highly non-trivial problem which has been resolved by the recent theory of six-functor formalisms \cite{liu_enhanced_2017, GaitsgoryStudy1,  heyer_6-functor_2024}. As has been demonstrated in \cite{ScholzeSixFunctors}, this theory also vastly simplifies the proof of many duality theorems.
\egapar In this article we make use of the six-functor formalism for quasi-coherent sheaves in rigid geometry as developed in our previous work \cite{SoorThesis}. Before stating our Main Theorem, we explain how this implemented. Because our category $\operatorname{QCoh}$ of coefficients is derived, it stands to reason that our algebras should be too, otherwise we will not be able to obtain basic properties such as base-change. The category $\operatorname{Afnd}$ of \emph{derived affinoid spaces} is opposite to the category $\operatorname{AfndAlg}$ of \emph{derived affinoid algebras}, defined as in Definition \ref{defn:DerivedAffinoid}. The category $\operatorname{Afnd}$ is endowed with the \emph{rational Grothendieck topology} $\tau_\mathrm{rat}$. There is a six-functor formalism in the sense of \cite{heyer_6-functor_2024}, denoted $\operatorname{QCoh}$, defined on $\operatorname{Shv}_{\tau_{\mathrm{rat}}}(\operatorname{Afnd})$, where $\operatorname{Shv}$ stands for the category of sheaves valued in $\infty$-groupoids. When $X = \operatorname{Spec}(A)$ is the object of $\operatorname{Afnd}$ corresponding to $A \in \operatorname{AfndAlg}$, one has
\begin{equation}
    \operatorname{QCoh}(X) = \operatorname{Mod}_A D(\operatorname{IndBan}_K),
\end{equation}
where $D(\operatorname{IndBan}_K)$ is the derived category of the $\operatorname{Ind}$-category of non-Archimedean $K$-Banach spaces. We denote by $E$ the class of edges for which $!$-functors are defined in this six-functor formalism. The category of derived rigid spaces is a full subcategory of $\operatorname{Shv}_{\tau_{\mathrm{rat}}}(\operatorname{Afnd})$, defined in Definition \ref{defn:derivedrigidspace}. For a morphism $f: X \to Y$ in $\operatorname{Shv}_{\tau_{\mathrm{rat}}}(\operatorname{Afnd})$ we denote the six operations by $(f^*,f_*,f_!,f^!,\hat{\otimes}_X, \underline{\operatorname{Hom}}_X)$, where as usual every second functor is right adjoint to the previous, and the $!$-functors are only defined if $f \in E$.
\begin{thm}\label{thm:intro2}
With notations as above. Let $f: X \to Y$ be a morphism of derived rigid spaces. Any undefined notions will be explained subsequently. 
\begin{enumerate}[(i)]
    \item If $f$ is \emph{taut} then $f$ belongs to the class $E$.
    \item If $f$ is \emph{qcqs} then there is a canonical equivalence $f_! \simeq f_*$. 
    \item If $f$ is a \emph{quasi-smooth Zariski-closed immersion} then $f$ is \emph{suave}.
    \item If $f$ is \emph{partially proper} and  \emph{\'etale} then there is a canonical equivalence $f^! \simeq f^*$. 
    \item If $f$ is \emph{Kiehl partially-proper} and \emph{smooth of dimension $d$} then $f$ is \emph{suave} and there is a canonical equivalence $f^! \simeq f^* \hat{\otimes}_X \operatorname{det}(\mathbf{L}_{X/Y})[d]$.
\end{enumerate}
\end{thm}
\egapar Let us explain the notations appearing in the above Theorem. 

In (i), a morphism is \emph{taut} if the underlying morphism of classical rigid spaces is \emph{taut} in the sense of Huber \cite[Definition 0.4.7, Definition 5.1.2]{HuberEtale}. This is a very mild condition. In \emph{ibid., p. 18.,} Huber writes: \emph{``I hope that the class of taut morphisms is enough for all applications of rigid geometry."}

In (ii), \emph{qcqs} stands for \emph{quasi-compact and quasi-separated}. In fact, we already proved this part in \cite{SoorThesis}. 

In (iii), \emph{quasi-smooth Zariski-closed immersion} is defined as in Definition \ref{defn:Zariskiclosedquasismooth}. This is the derived analogue of a regular closed immersion. 

In (iii) and (v), the notion of \emph{suave} morphism is defined as in \cite[\S4.5]{heyer_6-functor_2024}. This already implies a version of Grothendieck duality up to the identification of the dualizing object $f^!1$. 

In (iv), we say that a morphism is \emph{partially proper} if the underlying morphism of classical rigid spaces is partially proper in the sense of \cite[Definition 1.3.3]{HuberEtale}. We think of such morphisms as being ``without boundary". Boundaries are a feature of analytic geometry which does not appear in algebraic geometry. 

In (iv), the notion of an \emph{\'etale} morphism of derived rigid spaces is defined in Definition \ref{defn:Derivedsmooth}. 

In (v), we say that a morphism is \emph{Kiehl-partially proper} if the underlying morphism of classical rigid spaces is as in \cite[Remark 1.3.19(i)]{HuberEtale}. This notion is almost the same as that of a partially-proper morphism. Indeed, Kiehl-partially proper implies partially-proper, and the two notions coincide whenever $K$ is discretely valued \cite[Remark 1.3.19]{HuberEtale}. 

In (v), the notion of a \emph{smooth} (of dimension $d$) morphism of derived rigid spaces is defined in Definition \ref{defn:Derivedsmooth}. The object $\mathbf{L}_{X/Y}$ appearing is the \emph{cotangent complex} of the morphism $f$. In fact, by Theorem \ref{thm:CoherentCotangent}, this exists for any morphism of derived rigid spaces. In the case when $f$ is smooth it is a vector bundle. The operation $\operatorname{det}$ is defined in \S\ref{sec:duality}, it is the \emph{determinant} of this vector bundle. 

\egapar By \cite[Proposition 3.2.2(iii)]{heyer_6-functor_2024}, part (v) of Theorem \ref{thm:intro2} yields the following: for every smooth and Kiehl partially-proper morphism $f:X \to Y$ of derived rigid spaces, of dimension $d$, and every $M \in \operatorname{QCoh}(X), N \in \operatorname{QCoh}(Y)$, there is a canonical equivalence 
\begin{equation}
    \underline{\operatorname{Hom}}_Y(f_!M, N) \simeq f_*\underline{\operatorname{Hom}}_X(M, f^*N \hat{\otimes}_X (\operatorname{det}_X \mathbf{L}_{X/Y})[d]). 
\end{equation}
This is the analytic counterpart of Theorem \ref{thm:Intro1}. It is a generalization of the traditional statement of Serre duality for rigid-analytic varieties \cite{ChiarellottoDuality,vanderPutSerre,BeyerSerre} in the following three ways. Firstly, from the absolute to the relative setting: $f$ is allowed to be an arbitrary smooth and Kiehl partially-proper morphism, in particular, there are no restrictions on $Y$. Secondly, no finiteness conditions such as coherence are imposed on $M$ and $N$, and there are no ``$\operatorname{Ext}$-independence" conditions as in \cite[Main Theorem 5.1]{vanderPutSerre}. Thirdly, $X$ and $Y$ are allowed to be derived, rather than just classical rigid spaces. 
\egapar I will outline some potential applications of the main result. 

For the first application, we let $p$ be a prime and specialize to the case when the base field is $\mathbf{C}_p$. We consider a compact open subgroup $K = K^pK_p \subseteq \operatorname{GL}_2(\mathbf{A}^\infty)$ and denote Emerton's completed cohomology of tame level $K^p$ \cite{emerton_interpolation_2006} by $\tilde{H}^i(K^p, \mathbf{C}_p)$. After passing to locally-analytic vectors we obtain an admissible locally-analytic representation $\tilde{H}^i(K^p, \mathbf{C}_p)^\mathrm{la}$ of $\mathrm{GL}_2(\mathbf{Q}_p)$. In \cite{pan_locally_2022-1}, Pan constructs a certain equivariant $\mathscr{D}_X$-module $\mathscr{O}^{\mathrm{la}}$ on $\mathbf{P}^{1, \mathrm{an}}$ such that $H^i(\mathbf{P}^{1, \mathrm{an}}, \mathscr{O}^\mathrm{la}) \simeq \tilde{H}^i(K^p, \mathbf{C}_p)^\mathrm{la}$. A related construction of an equivariant $\mathscr{D}_X$-module on $\mathbf{P}^{1, \mathrm{an}}$ proceeds\footnote{There is a technicality about infinitesimal characters which I am ignoring.} by applying Ardakov's equivariant Beilinson--Bernstein localization \cite{Equivariant} to the dual of $\tilde{H}^i(K^p, \mathbf{C}_p)^\mathrm{la}$. In \cite[Remark 4.2.10]{pan_locally_2022-1}, Pan posits that the two constructions are related by some kind of Serre duality. But $\mathscr{O}^\mathrm{la}$ is not coherent, so the existing Serre duality results do not apply: to establish this relation, we really need Grothendieck duality. 

Serre duality in rigid geometry has important applications in Iwasawa theory, for instance in the theory of explicit reciprocity laws \cite[\S4.2]{schneider2023reciprocitylaws}. For these applications it is important to know many functoriality properties of the duality isomorphism. This is precisely what the six-functor formalism gives.

Serre duality can be used to construct shifted symplectic structures on moduli spaces of perfect complexes on Calabi-Yau varieties \cite{ShiftedSymplectic}. Due to its relative nature, it is plausible that Grothendieck duality can be used to construct families of shifted symplectic structures. Therefore we expect that our result can be used to construct families of shifted symplectic structures in analytic geometry.

\egapar Now I will explain the organization of this paper, as well as highlight some ancillary results which may be of independent interest.

In \S\ref{sec:RestrictionExtension} we develop sheaf-theoretic preliminaries. Many of these results are ``folklore" in the sense that they are already proven in the classical setting in \cite{SGA4-1}, and we merely observe that the argument is formally the same in the higher-categorical setting. 

In \S\ref{sec:derivedalgebracontexts}, we recall Raksit's theory of \emph{derived algebraic contexts} and the theory of \emph{derived algebras} due to Mathew, Raksit and Brantner--Mathew. This is necessary because we do not make any assumption on the characteristic of the ground field $K$; for instance, $K = \mathbf{F}_q(\!(t)\!)$ is permitted. Therefore, we have to exercise caution in our definition of ``derived algebras": Raksit's theory gives the right thing in positive characteristic.  

In \S\ref{sec:Analytification}, we use the sheaf-theoretic tools developed to implement the functor of analytification, with very good functoriality properties: it is a geometric morphism of $\infty$-topoi. The analytification functor is written as a composite of two geometric morphisms, using an intermediate topos introduced in Definition \ref{defn:dri}. This intermediate topos plays a key role in the remainder of the paper as it is large enough to contain objects of both algebraic and analytic nature.

In \S\ref{sec:classical}, we recall some features of the functor of classical truncation introduced in \cite{SoorThesis}, and also recast this as a sheaf-theoretic operation. This section is important for backwards-compatibility. It also allows us to bootstrap results about classical rigid spaces, to prove results about derived rigid spaces.

In \S\ref{sec:relativespectrum} we define the relative spectrum, in analogy with the corresponding construction in algebraic geometry. This is used to define vector bundles and square-zero extensions of arbitrary stacks. 

In \S\ref{sec:cotangent} we define cotangent complexes, smooth and \'etale morphisms, and quasi-smooth Zariski-closed immersions and prove various properties of these. 

In \S\ref{sec:partiallyproper}, we define partially-proper morphisms and prove various properties of these. In particular we prove structure theorems for partially-proper \'etale morphisms (Proposition \ref{prop:ppetalestructure}) and Kiehl partially-proper smooth morphisms \ref{prop:SmooothPartiallyProper}. 

In \S\ref{sec:6FF1}, we prove the first four parts of Theorem \ref{thm:intro2}, and half of Theorem \ref{thm:intro2}(v): that is, we show that $f$ is suave, yielding the equivalence $f^! \simeq f^* \hat{\otimes}_X f^!1$, but do not yet identify the dualizing object $f^!1$. We remark that our proof of Theorem \ref{thm:intro2}(i) does not use compactifications. 

In \S\ref{sec:duality}, we construct the canonical equivalence $f^!1 \simeq \operatorname{det}(\mathbf{L}_{X/Y})[d]$, using a deformation to the normal cone. Our deformation to the normal cone is based on the construction from \cite{BenBassatHekkingBlowUp, hekking_khan_deformation_2025}, using derived Weil restriction. The idea of using a deformation to the normal cone to obtain this equivalence is not original to us. As far as we are aware, in algebraic geometry this was first used in \cite[Ch. 9, Corollary 7.3.2]{GaitsgoryStudy2}, and subsequently in the analytic setting in \cite{CondensedComplexGeometry, zavyalov_poincare_2023, camargo_notes_2026}. 

The final \S\ref{sec:bonus} contains results which are not used elsewhere in the paper. The motivation is the following. In \S\ref{sec:duality}, the deformation space which is constructed is algebro-analytic in nature: it interpolates between a Zariski-closed immersion and its algebraic normal bundle. In this section, we show that there is an analytic deformation space: this is a rigid-analytic space which interpolates between a quasi-smooth Zariski-closed immersion and its analytic normal bundle. 

\egapar  We note also that our method can be used with minimal modifications to prove a duality result for dagger-analytic spaces \cite{GrosseKlonneRigid}. By using Kedlaya's theory of reified adic spaces \cite{kedlaya2015reifiedvaluationsadicspectra}, it is likely that our method can be adapted to obtain a duality theorem for non-strict Berkovich analytic spaces. Further, we believe that our method is well-adapted to the Archimedean setting, and can be used to obtain a duality theorem in complex-analytic geometry --- this explains the title. 

\paragraph{Acknowledgements.} I would like to thank Jack Kelly for a number of helpful Zoom meetings. I also thank Kobi Kremnizer for insightful discussions during Trinity 2025 in Oxford. 

\section{Restriction and extension of 
sheaves}\label{sec:RestrictionExtension}
\egapar The purpose of this section is to record some results about sheaves and topoi which will be used later.

\egapar Let $f:\mathscr{C} \to \mathscr{D}$ be a functor between small $\infty$-categories. Precomposition with $f$ gives a functor
\begin{equation}
    f_P^*:  \operatorname{Psh}(\mathscr{D}) \to  \operatorname{Psh}(\mathscr{C}).
\end{equation}
By left and right Kan extension we obtain an adjoint triple
\begin{equation}
    f_{P,!} \dashv f_P^* \dashv f_{P,*}.
\end{equation}
\begin{lem}\label{lem:presheavesFF}
With notations as above. The following are equivalent:
\begin{enumerate}[(i)]
    \item The functor $f$ is fully faithful;
    \item The functor $f_{P,!}$ is fully faithful,
    \item The functor $f_{P,*}$ is fully faithful.
\end{enumerate}
\end{lem}
\begin{proof}
The proof is identical to \cite[Expos\'e I, Proposition 5.6]{SGA4-1}. 
\end{proof}
\begin{lem}\label{lem:adjointfunctorssites1}
Let $f: \mathscr{C} \leftrightarrows \mathscr{D}:g$ be an adjunction\footnote{I always write right adjoints on the right.} between small $\infty$-categories. Then there are canonical equivalences $f_{P}^* \simeq g_{P,!}$ and $f_{P,*} \simeq g_P^*$.
\end{lem}
\begin{proof}
The proof is identical to \cite[Expos\'e I, Proposition 5.5]{SGA4-1}.
\end{proof}
\begin{lem}\label{lem:Pshexact}
Suppose that $\mathscr{C}$ and $\mathscr{D}$ have finite limits, and the functor $f: \mathscr{C} \to \mathscr{D}$ is left exact. Then $f_{P,!}$ is left exact. 
\end{lem}
\begin{proof}
The proof is identical to \cite[Expos\'e I, Proposition 5.2]{SGA4-1}. 
\end{proof}
\begin{defn}
Let $\mathscr{C}$ be a small $\infty$-category and let $X \in \mathscr{C}$. A \emph{sieve} is a full subcategory $\mathscr{C}^0_{/X} \subseteq \mathscr{C}_{/X}$ such that, if $[Y \to X] \in \mathscr{C}^0_{/X}$ and $[Z \to Y]$ is any morphism, then $[Z \to X] \in \mathscr{C}^0_{/X}$. 
\end{defn}
\begin{defn}
A site is a small $\infty$-category $\mathscr{C}$ with a collection $\tau$ of sieves satisfying the following properties:
\begin{enumerate}[(i)]
    \item[(T1)] For every $X \in \mathscr{C}$, $\mathscr{C}_{/X} \in \tau$;
    \item[(T2)] For every morphism $f: Y \to X$ of $\mathscr{C}$ and every $\mathscr{C}_{/X}^0 \in \tau$, one has $\mathscr{C}_{/Y} \times_{\mathscr{C}_{/X}}\mathscr{C}_{/X}^0 \in \tau$; 
    \item[(T3)] Let $\mathscr{C}^0_{/X}$ and $\mathscr{C}^1_{/X}$ be two sieves on $X$ such that $\mathscr{C}^0_{/X} \in \tau$. Suppose that for every $[Y \to X] \in \mathscr{C}^0_{/X}$ one has $\mathscr{C}_{/Y} \times_{\mathscr{C}_{/X}}\mathscr{C}_{/X}^1 \in \tau$. Then $\mathscr{C}_{/X}^1 \in \tau$.  
\end{enumerate}
A collection $\tau$ of sieves satisfying (T1)-(T3) is called a \emph{Grothendieck topology} on $\mathscr{C}$. A sieve belonging to $\tau$ is then called a \emph{$\tau$-covering sieve}.
\end{defn}
\begin{defn}
Let $(\mathscr{C}, \tau)$ be an $\infty$-site and let $\mathscr{D}$ be an $\infty$-category admitting small limits. A functor $F: \mathscr{C}^\mathrm{op} \to \mathscr{D}$ is said to be a \emph{$\mathscr{D}$-valued sheaf} on $(\mathscr{C},\tau)$, if for every $X \in \mathscr{C}$ and every $\tau$-covering sieve $\mathscr{C}^0_{/X}$, the natural morphism
\begin{equation}
    F(X) \to \underset{U \in (\mathscr{C}^0_{/X})^\mathrm{op}}{\operatorname{lim}} F(U)
\end{equation}
is an equivalence. We denote the full subcategory of $\operatorname{Fun}(\mathscr{C}^\mathrm{op},\mathscr{D})$ on $\mathscr{D}$-valued sheaves by $\operatorname{Shv}_\tau(\mathscr{C}, \mathscr{D})$. When $\mathscr{D} = \infty\mathrm{Grpd}$ we abbreviate to $\operatorname{Shv}_\tau(\mathscr{C})$. 
\end{defn}
\begin{defn}
Let $(\mathscr{C}, \tau)$ and $(\mathscr{D}, \tau^\prime)$ be $\infty$-sites. 
\begin{enumerate}[(i)]
    \item Let $f: \mathscr{C} \to \mathscr{D}$ be any functor. We denote by $f^*$ the composite functor 
    \begin{equation}
        \operatorname{Shv}_{\tau^\prime}(\mathscr{D}) \xrightarrow[]{j} \operatorname{Psh}(\mathscr{D}) \xrightarrow[]{f_P^*}  \operatorname{Psh}(\mathscr{C}) \xrightarrow[]{L}  \operatorname{Shv}_{\tau}(\mathscr{C}),
    \end{equation}
    and by $f_!$ the composite functor 
    \begin{equation}
        \operatorname{Shv}_{\tau}(\mathscr{C}) \xrightarrow[]{j} \operatorname{Psh}(\mathscr{C}) \xrightarrow[]{f_{P,!}} \operatorname{Psh}(\mathscr{D}) \xrightarrow[]{L} \operatorname{Shv}_{\tau^\prime}(\mathscr{D}), 
    \end{equation} 
    and by $f_*$ the composite functor
    \begin{equation}
        \operatorname{Shv}_{\tau}(\mathscr{C}) \xrightarrow[]{j} \operatorname{Psh}(\mathscr{C}) \xrightarrow[]{f_{P,*}} \operatorname{Psh}(\mathscr{D}) \xrightarrow[]{L} \operatorname{Shv}_{\tau^\prime}(\mathscr{D}), 
    \end{equation}
    here $j$ is the inclusion and $L$ stands for sheafification. 
    \item A functor $f: \mathscr{C} \to \mathscr{D}$ is called \emph{continuous} if the composite 
    \begin{equation}
        \operatorname{Shv}_{\tau^\prime}(\mathscr{D}) \xrightarrow[]{j} \operatorname{Psh}(\mathscr{D}) \xrightarrow[]{f_P^*} \operatorname{Psh}(\mathscr{C})
    \end{equation}
    factors through the full subcategory $\operatorname{Shv}_\tau(\mathscr{C}) \subseteq \operatorname{Psh}(\mathscr{C})$.
    \item A functor $f: \mathscr{C} \to \mathscr{D}$ is called \emph{cocontinuous} if the composite 
    \begin{equation}
    \operatorname{Shv}_{\tau}(\mathscr{C}) \xrightarrow[]{j} \operatorname{Psh}(\mathscr{C}) \xrightarrow[]{f_{P,*}} \operatorname{Psh}(\mathscr{D})
    \end{equation}
    factors through the full subcategory $\operatorname{Shv}_{\tau^\prime}(\mathscr{D}) \subseteq \operatorname{Psh}(\mathscr{D})$. 
\end{enumerate}
\end{defn}
\begin{lem}\label{lem:continuousfunctor}
Let $f: \mathscr{C} \to \mathscr{D}$ be a continuous functor between $\infty$-sites $(\mathscr{C}, \tau)$ and $(\mathscr{D}, \tau^\prime)$. Then there is an adjunction
\begin{equation}
    f_!: \operatorname{Shv}_\tau({\mathscr{C}}) \leftrightarrows \operatorname{Shv}_{\tau^\prime}({\mathscr{D}}) :f^*.
\end{equation}
If $\mathscr{C}$ and $\mathscr{D}$ both admit finite limits and $f: \mathscr{C} \to \mathscr{D}$ is left exact, then $f_!$ is left exact.
\end{lem}
\begin{proof}
There is an adjunction
\begin{equation}
L \circ f_{P,!} : \operatorname{Psh}(\mathscr{C}) \leftrightarrows \operatorname{Shv}_{\tau^\prime}(\mathscr{D}) : f^*_P \circ j. 
\end{equation}
By assumption, $f^*_P \circ j$ factors through $ \operatorname{Shv}_\tau(\mathscr{C}) \subseteq \operatorname{Psh}(\mathscr{C})$, so this induces the desired adjunction 
\begin{equation}
f_! : \operatorname{Shv}_\tau(\mathscr{C}) \leftrightarrows \operatorname{Shv}_{\tau^\prime}(\mathscr{D}) : f^*.
\end{equation}
The last part follows from Lemma \ref{lem:Pshexact} together with the fact that sheafification is left exact \cite[Proposition 6.2.2.7]{HigherToposTheory}.
\end{proof}
\begin{cor}\label{cor:lowershriekrepresentable}
Let $(\mathscr{C}, \tau)$ and $(\mathscr{D}, \tau^\prime)$ be $\infty$-sites. Let $X \in \mathscr{C}$. Let us also denote its image under the Yoneda embedding $\mathscr{C} \hookrightarrow \operatorname{Psh}(\mathscr{C})$ by $X$. If $f: \mathscr{C} \to \mathscr{D}$ is a continuous functor, then there is a canonical equivalence $f_!(L X) \simeq Lf(X)$. 
\begin{proof}
Let $Y \in \operatorname{Shv}_{\tau^\prime}(\mathscr{D})$. By adjunction and using that $f_P^*$ preserves sheaves, we obtain a chain of equivalences 
\begin{equation}
\begin{aligned}
    \operatorname{Map}_{\operatorname{Shv}(\mathscr{D})}(f_!LX, Y) &\simeq  \operatorname{Map}_{\operatorname{Shv}(\mathscr{C})}(LX, f^*Y) \\
    &\simeq \operatorname{Map}_{\operatorname{Psh}(\mathscr{C})}(X, f_P^*Y)\\
    &\simeq (f_P^*Y)(X) = Y(f(X)) \\
    &\simeq \operatorname{Map}_{\operatorname{Psh}(\mathscr{D})}(f(X), Y)\\
    &\simeq  \operatorname{Map}_{\operatorname{Shv}(\mathscr{D})}(Lf(X), Y),
\end{aligned}
\end{equation}
now the claim follows from Yoneda's lemma. 
\end{proof}
\end{cor}
\begin{lem}\label{lem:cocontinuousadjunction}
Let $f: \mathscr{C} \to \mathscr{D}$ be a cocontinuous functor between $\infty$-sites $(\mathscr{C}, \tau)$ and $(\mathscr{D}, \tau^\prime)$. Then there is an adjunction
\begin{equation}
        f^* : \operatorname{Shv}_{\tau^\prime}(\mathscr{D}) \leftrightarrows \operatorname{Shv}_\tau(\mathscr{C}): f_*.
\end{equation}
If $\mathscr{C}$ and $\mathscr{D}$ admit finite limits and $f: \mathscr{C} \to \mathscr{D}$ is left exact, then $f^*$ is left exact. 
\end{lem}
\begin{proof}
The proof is formally similar to that of Lemma \ref{lem:continuousfunctor}, and so we omit it.  
\end{proof}
\begin{lem}\label{lem:adjointfunctorssites2}
Let $(\mathscr{C}, \tau)$ and $(\mathscr{D}, \tau^\prime)$ be $\infty$-sites, and let $f: \mathscr{C} \leftrightarrows \mathscr{D}: g$ be an adjunction between the underlying $\infty$-categories. Then $g$ is continuous if and only if $f$ is cocontinuous. In this situation there are canonical equivalences $f_* \simeq g^*$ and $f^* \simeq g_!$. 
\end{lem}
\begin{proof}
By Lemma \ref{lem:adjointfunctorssites1} there is an equivalence $f_{P, *} \simeq g_P^*$, hence $f_{P, *}$ preserves sheaves if and only if $g_P^*$ does. In this situation we obtain a canonical equivalence $f_* \simeq g^*$. Passing to left adjoints using Lemma \ref{lem:continuousfunctor} and Lemma \ref{lem:cocontinuousadjunction}, we obtain the equivalence $f^* \simeq g_!$.  
\end{proof}
\begin{lem}\label{lem:continuouscriterion}
Let $f: \mathscr{C} \to \mathscr{D}$ be a functor between $\infty$-sites $(\mathscr{C}, \tau)$ and $(\mathscr{D}, \tau^\prime)$. Assume that $\mathscr{C}$ and $\mathscr{D}$ admit finite limits and that $f$ is left-exact. Suppose that, for every $X \in \mathscr{C}$ and every $\tau$-covering sieve $\mathscr{C}^0_{/X}$, the sieve $\mathscr{D}^1_{/f(X)}$ generated by the image of $\mathscr{C}^0_{/X} \to \mathscr{D}_{/f(X)}$ is a $\tau^\prime$-covering sieve. Then $f$ is continuous.
\end{lem}
\begin{proof}
Let $X \in \mathscr{C}$ and let $\mathscr{C}^0_{/X}$ be a $\tau$-covering sieve. We need to show that for every $F \in \operatorname{Shv}_{\tau^\prime}(\mathscr{D})$, the canonical morphism $F(f(X)) \to \operatorname{lim}_{U \in (\mathscr{C}^0_{/X})^{\mathrm{op}}} F(f(U))$ is an equivalence. It suffices to show that the functor $\mathscr{C}^0_{/X} \to \mathscr{D}^1_{/f(X)}$ is right cofinal. We use Quillen's Theorem A \cite[\href{https://kerodon.net/tag/02NX}{Tag 02NX}]{kerodon}: it suffices to show that for each $[Z \to f(X)] \in \mathscr{D}^1_{/f(X)}$ the $\infty$-category
\begin{equation}
    \mathscr{E} := \mathscr{C}^0_{/X} \times_{\mathscr{D}^1_{/f(X)}} \big(\mathscr{D}^1_{/f(X)}\big)_{Z/}
\end{equation}
is weakly contractible. Concretely, this is the full subcategory of $\mathscr{D}_{Z/\!/f(X)}$ on objects of the form $[Z \to f(X^\prime) \to f(X)]$ where the second morphism is induced by some $[X^\prime \to X] \in \mathscr{C}^0_{/X}$. It is clearly nonempty. By the assumption that $f$ preserves finite limits, we see that $\mathscr{E}$ admits finite products, and hence $\mathscr{E}^\mathrm{op}$ admits finite coproducts; in particular $\mathscr{E}^\mathrm{op}$ is sifted and so by \cite[Proposition 5.5.8.7]{HigherToposTheory} $\mathscr{E}$ is weakly contractible.
\end{proof}
\begin{rmk}
It should also be possible to show that the condition in Lemma \ref{lem:continuouscriterion} is necessary for $f$ to be continuous.
\end{rmk}
\begin{lem}\label{lem:cocontinuousmap}
Let $f: \mathscr{C} \to \mathscr{D}$ be a functor between $\infty$-sites $(\mathscr{C}, \tau)$ and $(\mathscr{D}, \tau^\prime)$. Then the following are equivalent:
\begin{enumerate}[(i)]
    \item $f$ is cocontinuous;
    \item For every $X \in \mathscr{C}$ and every $\tau^\prime$-covering sieve $\mathscr{D}^0_{/f(X)} \subseteq \mathscr{D}_{/f(X)}$, the subcategory $\mathscr{C}_{/X} \times_{\mathscr{D}_{/f(X)}} \mathscr{D}^0_{/f(X)} \subseteq \mathscr{C}_{/X}$ is a $\tau$-covering sieve.
\end{enumerate}
\end{lem}
\begin{proof}
The proof is identical to \cite[Expos\'e III, Proposition 2.2]{SGA4-1}.
\end{proof}
\egapar The following construction is important for the construction of the analytification functor in the next section.
\begin{defn}\label{defn:pushforwardtopology}
Let $(\mathscr{C}, \tau)$ be an $\infty$-site, and let $\mathscr{D}$ be a small $\infty$-category. Let $f: \mathscr{C} \to \mathscr{D}$ be any functor. We define a collection of sieves $f_*\tau$ on $\mathscr{D}$ in the following way:
\begin{itemize}
    \item[$\star$] For each $Y \in \mathscr{D}$, a full subcategory $\mathscr{D}^0_{/Y} \subseteq \mathscr{D}_{/Y}$ belongs to $f_*\tau$ if and only if, for every $X \in \mathscr{C}$ and every morphism $f(X) \to Y$ in $\mathscr{D}$, the subcategory $\mathscr{C}_{/X} \times_{\mathscr{D}_{/Y}}\mathscr{D}^0_{/Y} \subseteq \mathscr{C}_{/X}$ is a $\tau$-covering sieve\footnote{Here the functor $\mathscr{C}_{/X} \to \mathscr{D}_{/Y}$ sends $[Z \to X]$ to the composite $[f(Z) \to f(X) \xrightarrow[]{g} Y]$.}.
\end{itemize}
\end{defn}
\begin{lem}\label{lem:cocontinuousinitial}
With notations as in Definition \ref{defn:pushforwardtopology}. The collection $f_*\tau$ of sieves forms a Grothendieck topology on $\mathscr{D}$. With respect to this topology the functor $f: \mathscr{C} \to \mathscr{D}$ is cocontinuous. 
\end{lem}
\begin{proof}
Let us first prove that $f_*\tau$ is a Grothendieck topology. The axiom (T1) is clear. The axiom (T2) is clear from the associativity of fiber products. For axiom (T3), let $\mathscr{D}^0_{/Y} \in f_*\tau$ and let $\mathscr{D}^1_{/Y}$ be a sieve on $\tau$ such that, for every $[Y^\prime \to Y] \in \mathscr{D}^0_{/Y}$, the subcategory
\begin{equation}
    \mathscr{D}_{/Y^\prime} \times_{\mathscr{D}_{/Y}} \mathscr{D}^1_{/Y} \subseteq \mathscr{D}_{/Y^\prime}
\end{equation}
is a $f_*\tau$-covering sieve on $Y^\prime$. Let us fix $f(X) \to Y$; we need to show that
\begin{equation}\label{eq:coveringsievelemma}
    \mathscr{C}_{/X} \times_{\mathscr{D}_{/Y}} \mathscr{D}^1_{/Y}  \subseteq \mathscr{C}_{/X}
\end{equation}
is a $\tau$-covering sieve. Let $[X^\prime \to X] \in \mathscr{C}_{/X} \times_{\mathscr{D}_{/Y}} \mathscr{D}^0_{/Y} \subseteq \mathscr{C}_{/X}$; by assumption, this is a $\tau$-covering sieve. Hence by (T3) it is sufficient to show that 
\begin{equation}\label{eq:coveringsievelemma2}
    \mathscr{C}_{/X^\prime} \times_{\mathscr{C}_{/X}}\big( \mathscr{C}_{/X} \times_{\mathscr{D}_{/Y}} \mathscr{D}^1_{/Y}\big)
\end{equation}
is a $\tau$-covering sieve. Now by the associativity of fiber products, 
\begin{equation}
\begin{aligned}
   \mathscr{C}_{X^\prime} \times_{\mathscr{C}_{/X}}\big( \mathscr{C}_{/X} \times_{\mathscr{D}_{/Y}} \mathscr{D}^1_{/Y}\big) &\simeq \mathscr{C}_{/X^\prime} \times_{\mathscr{D}_{/Y}} \mathscr{D}^1_{/Y}\\
   &\simeq \mathscr{C}_{/X^\prime} \times_{\mathscr{D}_{/f(X^\prime)}}\big( \mathscr{D}_{/f(X^\prime)} \times _{\mathscr{D}_{/Y}} \mathscr{D}^1_{/Y}\big). 
\end{aligned}
\end{equation}
Now $[f(X^\prime) \to Y] \in \mathscr{D}^0_{/Y}$, so $\mathscr{D}_{/f(X^\prime)} \times _{\mathscr{D}_{/Y}} \mathscr{D}^1_{/Y}$ is a $f_*\tau$-covering sieve. This shows that \eqref{eq:coveringsievelemma2} is a $\tau$-covering sieve. By taking $g = \operatorname{id}: f(X) \to Y = f(X)$ in Definition \ref{defn:pushforwardtopology}, it is clear from Lemma \ref{lem:cocontinuousmap} that $f: \mathscr{C} \to \mathscr{D}$ is a cocontinuous functor with respect to this topology. 
\end{proof}
\begin{rmk}
It should also be possible to show that $f_*\tau$ is the finest topology on $\mathscr{D}$ such that the functor $f: \mathscr{C} \to \mathscr{D}$ is cocontinuous. 
\end{rmk}
\begin{thm}\label{thm:pushforwardtopology}
Let $(\mathscr{C},\tau)$ be an $\infty$-site and let $\mathscr{D}$ be a small $\infty$-category. Assume that $\mathscr{C}$ and $\mathscr{D}$ admit finite limits, and let $f: \mathscr{C} \to \mathscr{D}$ be a fully-faithful left-exact functor. Then the functor $f$ is both continuous and cocontinuous with respect to the topology $f_*\tau$ on $\mathscr{D}$. 
\end{thm}
\begin{proof}
Because of Lemma \ref{lem:cocontinuousinitial}, we just need to show that $f$ is continuous. Let $X \in \mathscr{C}$ and let $\mathscr{C}^0_{/X} \subseteq \mathscr{C}_{/X}$ be a $\tau$-covering sieve. Let $\mathscr{D}^1_{/f(X)} \subseteq \mathscr{D}_{/f(X)}$ be the sieve generated by the image of $\mathscr{C}^0_{/X} \to \mathscr{D}_{/f(X)}$. By Lemma \ref{lem:continuouscriterion}, it is sufficient to show that the sieve $\mathscr{D}^1_{/f(X)}$ is a $f_*\tau$-covering sieve. Let us fix a morphism $f(X^\prime) \to f(X)$; by fully-faithfulness this is induced by a morphism $X^\prime \to X$. We need to check that $\mathscr{C}_{/X^\prime} \times_{\mathscr{D}_{/f(X)}} \mathscr{D}^1_{/f(X)}$ is a $\tau$-covering sieve. By construction, this sieve consists of those morphisms $Y \to X^\prime$ such that $f(Y) \to f(X)$ factors over some $f(Y^\prime) \to f(X)$, with $[Y^\prime \to X] \in \mathscr{C}^0_{/X}$. By fully-faithfulness this is equivalent to saying that $Y \to X$ factors over $Y^\prime \to X$. Hence  
\begin{equation}
    \mathscr{C}_{/X^\prime} \times_{\mathscr{C}_{/X}} \mathscr{C}_{/X}^0 = \mathscr{C}_{/X^\prime} \times_{\mathscr{D}_{/f(X)}} \mathscr{D}^1_{/f(X)}
\end{equation}
as full subcategories of $\mathscr{C}_{/X^\prime}$, and so $\mathscr{D}^1_{/f(X)}$ is a $f_*\tau$-covering sieve, according to Definition \ref{defn:pushforwardtopology}. 
\end{proof}
\begin{lem}\label{lem:fullyfaithful}
Let $(\mathscr{C}, \tau)$ and $(\mathscr{D},\tau^\prime)$ be $\infty$-sites. Assume that $\mathscr{C}$ and $\mathscr{D}$ admit finite limits, and let $f: \mathscr{C} \to \mathscr{D}$ be a fully-faithful functor which is both continuous and cocontinuous. Then the functors $f_!$ and $f_*$ are fully-faithful. 
\end{lem}
\begin{proof}
It is sufficent to prove that the counit morphism $
f^*f_* \to \operatorname{id}$ is an equivalence; then the equivalence $\operatorname{id} \xrightarrow[]{\sim} f^*f_!$ then follows by passing to left adjoints. Because $f$ is both continuous and cocontinuous, the functors $f_P^*$ and $f_{P,*}$ preserve sheaves; hence the natural transformation $
f^*f_* \to \operatorname{id}$ identifies with the restriction of $f_{P}^*f_{P,*} \to \operatorname{id}$ to sheaves. The latter is an equivalence by Lemma \ref{lem:presheavesFF}. 
\end{proof}
\begin{lem}\label{lem:tautologicallimitlemma}
Let $\mathscr{C}$ be an $\infty$-category equipped with a subcanonical Grothendieck topology $\tau$. Then for every $X \in \mathscr{C}$ and every $\tau$-covering sieve $\mathscr{C}^0_{/X}$ on $X$, the limit $\operatorname{lim}_{Y \in (\mathscr{C}^0_{/X})^\mathrm{op}} \mathscr{O}(Y)$ exists in $\mathscr{C}^\mathrm{op}$ and the natural morphism
\begin{equation}\label{eq:tautologicallimitlemma}
    \mathscr{O}(X) \to \underset{Y \in (\mathscr{C}^0_{/X})^\mathrm{op}}{\operatorname{lim}} \mathscr{O}(Y)
\end{equation} 
is an equivalence\footnote{If we pretend that $\mathscr{C}$ admits all small colimits (this is in fact not possible since $\mathscr{C}$ is small), this is the same as saying that the identity functor $\mathscr{C}^\mathrm{op} \to \mathscr{C}^\mathrm{op}$ is a $\mathscr{C}^\mathrm{op}$-valued sheaf on $\mathscr{C}$.} in $\mathscr{C}^\mathrm{op}$. Here $\mathscr{O}(X)$ denotes the object of $\mathscr{C}^\mathrm{op}$ corresponding to $X \in \mathscr{C}$. 
\end{lem}
\begin{proof}
The Lemma is a tautology. Let $Z \in \mathscr{C}$. Then as $\tau$ is subcanonical, the natural morphism $\operatorname{Map}(X,Z) \to  \operatorname{lim}_{Y \in(\mathscr{C}^0_{/X})^\mathrm{op}}\operatorname{Map}(Y,Z)$
is an equivalence, which is nothing but the map 
\begin{equation}
    \operatorname{Map}(\mathscr{O}(Z),\mathscr{O}(X)) \xrightarrow[]{\sim}  \underset{Y \in(\mathscr{C}^0_{/X})^\mathrm{op}}{\operatorname{lim}}\operatorname{Map}(\mathscr{O}(Z),\mathscr{O}(Y)).
\end{equation}
By the universal property of limits, this implies that $\operatorname{lim}_{Y \in(\mathscr{C}^0_{/X})^\mathrm{op}}\mathscr{O}(Y)$ exists and the natural morphism \eqref{eq:tautologicallimitlemma} is an equivalence. 
\end{proof}
\egapar Let $\mathscr{X}$ be an $\infty$-topos and let $f: X \to Y$ be a morphism in $\mathscr{X}$. This induces an adjunction
\begin{equation}
    f_\sharp : \mathscr{X}_{/X} \leftrightarrows \mathscr{X}_{/Y}: f^\sharp
\end{equation}
where $f_\sharp$ is given by postcomposition with $f$, and $f^\sharp$ is given by pullback along $f$.
\begin{lem}\label{lem:weilresdef}
The functor $f^\sharp: \mathscr{X}_{/Y} \to \mathscr{X}_{/X}$ admits a right adjoint.
\end{lem}
\begin{proof}
By the ``fundamental theorem of topos theory" \cite[Proposition 6.3.5.1]{HigherToposTheory} the slice over any object is again a topos. In particular, by the first Giraud axiom \cite[Theorem 6.1.0.6(3)(i)]{HigherToposTheory} $\mathscr{X}_{/X}$ and $\mathscr{X}_{/Y}$ are both presentable. That $f^\sharp$ preserves colimits is equivalent to the second Giraud axiom \cite[Theorem 6.1.0.6(3)(ii)]{HigherToposTheory} which says that colimits in $\mathscr{X}$ are universal. Hence the adjoint functor theorem \cite[Corollary 5.5.2.9]{HigherToposTheory} implies that $f^\sharp$ admits a right adjoint. 
\end{proof}
\egapar The right adjoint to $f^\sharp$, denoted $f_\flat$, is called the \emph{Weil restriction} along $f$. The Weil restriction is compatible with base-change:
\begin{lem}\label{lem:weilresbasechange}
Let $f: X \to Y$ and $g:Y^\prime \to Y$ be morphisms in $\mathscr{X}$. Denote by $f^\prime: X^\prime \to Y^\prime $ (resp. $g^\prime: X^\prime \to X$) the base-change of $f$ (resp. $g$) along $g$ (resp. $f$). Then one has $g^\sharp f_\flat \xrightarrow[]{\sim} f^{\prime}_\flat g^{\prime, \sharp}$ via the canonical (i.e., Beck-Chevalley) morphism. 
\end{lem}
\begin{proof}
It is clear that $g^\prime_\sharp f^{\prime,\sharp} \xrightarrow[]{\sim} f^\sharp g_\sharp$: indeed, if $Z \in \mathscr{X}_{/Y^\prime}$ there is an equivalence
\begin{equation}
X^\prime \times_{Y^\prime} Z = X \times_Y Y^\prime \times_{Y^\prime} Z \simeq X \times_YZ,
\end{equation}
in $\mathscr{X}_{/X}$. The claim then follows by passing to right adjoints. 
\end{proof}
\begin{lem}\label{lem:weilresfunctorial}
Let $\mathscr{X}, \mathscr{Y}$ be $\infty$-topoi and let $F: \mathscr{X} \leftrightarrows \mathscr{Y}:G$ be an adjunction in which the left adjoint $F$ is left-exact. Let $f: X \to Y$ be a morphism in $\mathscr{X}$. Then for every $Y^\prime \in \mathscr{Y}_{/F(X)}$ there is a canonical equivalence
\begin{equation}
    G(F(f)_\flat Y^\prime) \times_{GF(Y)} Y \simeq f_\flat(G(Y^\prime)\times_{GF(X)} X).
\end{equation}
In particular if $F$ is fully-faithful then there is a canonical equivalence 
\begin{equation}
G(F(f)_\flat Y^\prime) \simeq f_\flat G(Y^\prime).
\end{equation} 
\end{lem}
\begin{proof}
Since $F$ is left-exact, there is a commutative square 
\begin{equation}
\begin{tikzcd}
	{\mathscr{X}_{/Y}} & {\mathscr{Y}_{/F(Y)}} \\
	{\mathscr{X}_{/X}} & {\mathscr{Y}_{/F(X)}}
	\arrow["F", from=1-1, to=1-2]
	\arrow["{f^\sharp}"', from=1-1, to=2-1]
	\arrow["{F(f)^\sharp}", from=1-2, to=2-2]
	\arrow["F", from=2-1, to=2-2]
\end{tikzcd}
\end{equation}
The claim follows by passing to right adjoints, using the slicing of adjunctions from \cite[Proposition 5.2.5.1]{HigherToposTheory}.
\end{proof}
\begin{prop}\label{prop:Slicetopoifunctoriality}
Let $(\mathscr{C}, \tau)$ be an $\infty$-site and let $X \in \operatorname{Shv}(\mathscr{C})$. We endow $\mathscr{C}_{/X}$ with the topology induced by the projection $p_X:\mathscr{C}_{/X} \to \mathscr{C}$ (with respect to this topology, the functor $p_X$ is both continuous and cocontinuous). Then, the functor $p_{X,!}: \operatorname{Shv}(\mathscr{C}_{/X}) \to \operatorname{Shv}(\mathscr{C})$ factors over an equivalence $\operatorname{Shv}(\mathscr{C}_{/X}) \xrightarrow[]{\sim} \operatorname{Shv}(\mathscr{C})_{/X}$. This is natural in the sense that if $f: X \to Y$ is any morphism in $\operatorname{Shv}(\mathscr{C})$ and $p_{XY}:\mathscr{C}_{/X} \to \mathscr{C}_{/Y}$ is the projection then there is a commutative square 
\begin{equation}
\begin{tikzcd}
	{\operatorname{Shv}(\mathscr{C}_{/X})} & {\operatorname{Shv}(\mathscr{C})_{/X}} \\
	{\operatorname{Shv}(\mathscr{C}_{/Y})} & {\operatorname{Shv}(\mathscr{C})_{/Y}}
	\arrow["\sim", from=1-1, to=1-2]
	\arrow["{p_{XY,!}}"', from=1-1, to=2-1]
	\arrow["{f_{\sharp}}", from=1-2, to=2-2]
	\arrow["\sim", from=2-1, to=2-2]
\end{tikzcd}
\end{equation}
and further, the squares
\begin{equation}
\begin{aligned}
\begin{tikzcd}
	{\operatorname{Shv}(\mathscr{C}_{/Y})} & {\operatorname{Shv}(\mathscr{C})_{/Y}} \\
	{\operatorname{Shv}(\mathscr{C}_{/X})} & {\operatorname{Shv}(\mathscr{C})_{/X}}
	\arrow["\sim", no head, from=1-1, to=1-2]
	\arrow["{p_{XY}^*}"', from=1-1, to=2-1]
	\arrow["{f^{\sharp}}", from=1-2, to=2-2]
	\arrow["\sim", no head, from=2-1, to=2-2]
\end{tikzcd} && \text{ and } && 
\begin{tikzcd}
	{\operatorname{Shv}(\mathscr{C}_{/X})} & {\operatorname{Shv}(\mathscr{C})_{/X}} \\
	{\operatorname{Shv}(\mathscr{C}_{/Y})} & {\operatorname{Shv}(\mathscr{C})_{/Y}}
	\arrow["\sim", no head, from=1-1, to=1-2]
	\arrow["{p_{XY,*}}"', from=1-1, to=2-1]
	\arrow["{f_\flat}", from=1-2, to=2-2]
	\arrow["\sim", no head, from=2-1, to=2-2]
\end{tikzcd}
\end{aligned}
\end{equation}
are naturally commutative.
\end{prop}
\begin{proof}
The argument is essentially the same as in \cite[Exposé III, Proposition 5.4, Proposition 5.5]{SGA4-1}. By definition, the functor $p_{X,!}$ is the composite 
\begin{equation}
\operatorname{Shv}(\mathscr{C}_{/X}) \xrightarrow[]{j} \operatorname{Psh}(\mathscr{C}_{/X}) \xrightarrow[]{p_{X,P,!}} \operatorname{Psh}(\mathscr{C}) \xrightarrow[]{L} \operatorname{Shv}(\mathscr{C}).
\end{equation}
By \cite[Corollary 5.1.6.12]{HigherToposTheory} the functor $p_{X,P,!}$ factors over an equivalence 
\begin{equation}
   e_{X,P}: \operatorname{Psh}(\mathscr{C}_{/X}) \xrightarrow[]{\sim} \operatorname{Psh}(\mathscr{C})_{/X}
\end{equation}
in particular the image of the terminal object under $p_{X,P,!}$ is $X$. Hence $p_{X,!}$ carries the terminal object to $X$ and so factors over a morphism $e_X: \operatorname{Shv}(\mathscr{C}_{/X}) \to \operatorname{Shv}(\mathscr{C})_{/X}$. There is a commutative diagram 
\begin{equation}
\begin{tikzcd}
	{\operatorname{Psh}(\mathscr{C}_{/X})} & {\operatorname{Shv}(\mathscr{C}_{/X})} & {\operatorname{Psh}(\mathscr{C}_{/X})} \\
	{\operatorname{Psh}(\mathscr{C})_{/X}} & {\operatorname{Shv}(\mathscr{C})_{/X}} & {\operatorname{Psh}(\mathscr{C})_{/X}}
	\arrow["{L_X}", from=1-1, to=1-2]
	\arrow["{e_{X,P}}"', from=1-1, to=2-1]
	\arrow["{j_X}", hook, from=1-2, to=1-3]
	\arrow["{e_X}"', from=1-2, to=2-2]
	\arrow["{e_{X,P}}", from=1-3, to=2-3]
	\arrow["{L_{/X}}", from=2-1, to=2-2]
	\arrow["{j_{/X}}", hook, from=2-2, to=2-3]
\end{tikzcd}
\end{equation}
hence, $e_X$ is the retract of an equivalence, and therefore an equivalence. For the last part, we can reduce to the case $Y = *$ by replacing $\mathscr{C}$ with the slice $\mathscr{C}_{/Y}$. In this case the composite $ \operatorname{Shv}(\mathscr{C}_{/X}) \xrightarrow{\sim} \operatorname{Shv}(\mathscr{C})_{/X}  \xrightarrow[]{f_\sharp} \operatorname{Shv}(\mathscr{C})$ is equivalent to $p_{X,!}$ by definition. 
\end{proof}
\begin{defn}\cite[Definition 1.3.1.4]{SpectralAlgebraicGeometry}.
Let $\mathscr{X}$ be an $\infty$-topos and $\mathscr{E}$ be an $\infty$-category admitting all small limits. An \emph{$\mathscr{E}$-valued sheaf on $\mathscr{X}$} is a limit-preserving functor $\mathscr{X}^\mathrm{op} \to \mathscr{E}$. We denote the $\infty$-category of such by $\operatorname{Shv}(\mathscr{X}, \mathscr{E})$.
\end{defn}
\begin{prop}\label{prop:sheavesontopos}
Let $(\mathscr{C},\tau)$ be an $\infty$-site and let $\mathscr{X} := \operatorname{Shv}_\tau(\mathscr{C})$ be the associated $\infty$-topos. Then restriction and right Kan extension along\footnote{This morphism is (the opposite of) the composite of the Yoneda embedding with sheafification.} $ \mathscr{C}^\mathrm{op} \to \mathscr{X}^\mathrm{op}$ yields an equivalence of $\infty$-categories $\operatorname{Shv}_\tau(\mathscr{C}, \mathscr{E}) \simeq \operatorname{Shv}(\mathscr{X}, \mathscr{E})$. 
\end{prop}
\begin{proof}
This is \cite[Proposition 1.3.1.7]{SpectralAlgebraicGeometry}.
\end{proof}
\begin{lem}\label{lem:rightKanextensionLemma}
Let $(\mathscr{C},\tau)$ and $(\mathscr{D},\tau^\prime)$ be $\infty$-sites, let $\mathscr{E}$ be an $\infty$-category admitting all small limits, and let $f: \mathscr{C} \to \mathscr{D}$ be a continuous functor. Then for every $G \in \operatorname{Shv}_{\tau^\prime}(\mathscr{D}, \mathscr{E})$ there is a canonical equivalence 
\begin{equation}
    \operatorname{Ran}_G \circ f_! \xrightarrow[]{\sim} \operatorname{Ran}_{G \circ f},
\end{equation}
in $\operatorname{Shv}(\operatorname{Shv}_\tau(\mathscr{C}), \mathscr{E})$. Here $\operatorname{Ran}_G$ (resp. $\operatorname{Ran}_{G \circ f}$) denotes the right Kan extension of $G: \mathscr{D}^\mathrm{op} \to \mathscr{E}$  along $\mathscr{D}^{\mathrm{op}} \to \operatorname{Shv}_{\tau^\prime}(\mathscr{D})^\mathrm{op}$ (resp. the right Kan extension of $G\circ f: \mathscr{C}^\mathrm{op} \to \mathscr{E}$ along $\mathscr{C}^{\mathrm{op}} \to \operatorname{Shv}_{\tau}(\mathscr{C})^\mathrm{op}$). 
\end{lem}
\begin{proof}
Put $G^\prime := \operatorname{Ran}_G$, by Proposition \ref{prop:sheavesontopos} this is limit-preserving and $G^\prime|_\mathscr{D} \simeq G$. The functor $f_!$ is colimit-preserving, therefore $\operatorname{Ran}_G \circ f_!$ is limit-preserving. By Corollary \ref{cor:lowershriekrepresentable} one has $(G^\prime \circ f_!)|_\mathscr{C} \simeq G^\prime \circ L \circ h \circ f$, where $h$ denotes the Yoneda embedding. But Proposition \ref{prop:sheavesontopos} implies that $G^\prime \circ L \circ h \simeq G$, hence $(G^\prime \circ f_!)|_\mathscr{C} \simeq G\circ f$ and by Proposition \ref{prop:sheavesontopos} again we deduce the claim. 
\end{proof}
\section{Derived algebraic contexts}\label{sec:derivedalgebracontexts}
\egapar The purpose of this section is the following. Our Main Theorem \ref{thm:intro2} does not impose any hypothesis on the characteristic of the ground non-Archimedean field $K$. For instance, $K = \mathbf{F}_q(\!(t)\!)$ is permitted. On the other hand, an essential component of our proof is the the use of derived geometry. Because of this, we need to exercise caution in positive characteristic, to ensure that we are working with the correct notion of ``derived algebras". This is achieved via the machinery developed in \cite{raksit_hochschild_2026} and \cite{BrantnerMathew}. Roughly speaking, our ``derived algebras" are an Ind-Banach version of simplicial commutative rings. The reason why we have to use this particular model, and not, for instance, $\mathbf{E}_\infty$-algebra objects, is because it is essential to define the derived symmetric algebra in a way which avoids introducing higher coinvariants for the action of the symmetric group (which do not vanish in positive characteristic). 
\egapar The definitions in this section are the same as those of Raksit \cite{raksit_hochschild_2026}, except that we prefer cohomological indexing notation. 
\begin{defn}\cite[Definition 3.3.1]{raksit_hochschild_2026}
Let $(\mathscr{V}, \otimes)$ be a presentably symmetric monoidal stable $\infty$-category. A $t$-structure $(\mathscr{V}^{\leqslant 0}, \mathscr{V}^{\geqslant 0})$ on the underlying stable $\infty$-category $\mathscr{V}$ is called \emph{compatible} if:
\begin{enumerate}[(i)]
    \item $\mathscr{V}^{\leqslant 0}$ is stable under filtered colimits in $\mathscr{V}$,
    \item the unit $1_{\mathscr{V}} \in \mathscr{V}^{\leqslant 0}$,
    \item if $M, N \in \mathscr{V}^{\leqslant 0}$ then $M \otimes N \in \mathscr{V}^{\leqslant 0}$.
\end{enumerate} 
\end{defn}
\begin{defn}\cite[Definition 4.2.1]{raksit_hochschild_2026}
A \emph{derived algebraic context} is a tuple \begin{equation}
(\mathscr{V},\mathscr{V}^{\leqslant 0}, \mathscr{V}^{\geqslant 0} ,\mathscr{V}^0)  
\end{equation}consisting of a presentably symmetric monoidal stable $\infty$-category $\mathscr{V}$, a compatible $t$-structure $(\mathscr{V}^{\leqslant 0}, \mathscr{V}^{\geqslant 0})$, and a small full subcategory $\mathscr{V}^0 \subseteq \mathscr{V}^\heartsuit$ such that:
\begin{enumerate}[(i)]
    \item The $t$-structure is right complete;
    \item $\mathscr{V}^0 \subseteq \mathscr{V}$ is a symmetric monoidal subcategory and $\mathscr{V}^0 \subseteq \mathscr{V}^\heartsuit$ is closed under the formation of $\mathscr{V}^\heartsuit$-symmetric powers;
    \item The subcategory $\mathscr{V}^0$ is closed under the formation of finite coproducts in $\mathscr{V}$ and its objects form a set of compact projective generators for $\mathscr{V}^{\leqslant 0}$.
\end{enumerate}
A \emph{morphism of derived algebraic contexts} $(\mathscr{V},\mathscr{V}^{\leqslant 0}, \mathscr{V}^{\geqslant 0} ,\mathscr{V}^0) \to (\mathscr{W},\mathscr{W}^{\leqslant 0}, \mathscr{W}^{\geqslant 0} ,\mathscr{W}^0)$ is a colimit-preserving, symmetric monoidal, right $t$-exact functor $\mathscr{V} \to \mathscr{W}$ which carries $\mathscr{V}^0$ into $\mathscr{W}^0$. 
\end{defn}
\begin{example}[The initial derived algebraic context]
Let $\mathscr{V} := D(\mathbf{Z})$ be the derived category of $\mathbf{Z}$-modules, let $(\mathscr{V}^{\leqslant 0}, \mathscr{V}^{\geqslant 0}) = (D^{\leqslant 0}(\mathbf{Z}), D^{\geqslant 0}(\mathbf{Z}))$ be the standard $t$-structure, and let $\mathscr{V}^0 \subseteq \mathscr{V}^\heartsuit$ be the category of finite free $\mathbf{Z}$-modules. Then $(\mathscr{V},\mathscr{V}^{\leqslant 0}, \mathscr{V}^{\geqslant 0} ,\mathscr{V}^0)$ is the initial object in the category of derived algebraic contexts. 
\end{example}
\egapar Here is the most important example for this purposes of this article.
\begin{example}\label{example:IndBancontext}
Let $K$ be a non-Archimedean field, and let $\operatorname{IndBan}_K$ denote the (ordinary) $\operatorname{Ind}$-category of the category $\operatorname{Ban}_K$ of non-Archimedean $K$-Banach spaces. We define $\mathscr{V} := (D(\operatorname{IndBan}_K), \hat{\otimes}^\mathbf{L}_K)$ where $\hat{\otimes}^\mathbf{L}_K$ is the derived completed tensor product. We define 
\begin{equation}
(\mathscr{V}^{\leqslant 0}, \mathscr{V}^{\geqslant 0}) = (D^{\leqslant 0}(\operatorname{IndBan}_K), D^{\geqslant 0}(\operatorname{IndBan}_K)) 
\end{equation}
to be the \emph{left $t$-structure} on $\mathscr{V}$, c.f. \cite{schneiders_quasi-abelian_1999}. For every small set $S$ we define the $K$-Banach space of \emph{zero sequences on $S$} as 
\begin{equation}
    c_0(S) := \{ \varphi:S \to K: \forall \varepsilon > 0 \ \exists \text{ at most finitely many }s \in S : |\varphi(s) |> \varepsilon\},
\end{equation}
equipped with the norm $\|\varphi\|:= \sup_{s \in S} |\varphi(s)|$. We define $\mathscr{V}_0$ to be the full subcategory of $\operatorname{Ban}_K$ on objects of the form $c_0(S)$. Then $(\mathscr{V},\mathscr{V}^{\leqslant 0}, \mathscr{V}^{\geqslant 0} ,\mathscr{V}^0)$ is a derived algebraic context. For instance, one has 
\begin{equation*}
\begin{aligned}
c_0(S) \oplus c_0(S^\prime) \cong c_0(S \amalg S^\prime ), && c_0(S) \hat{\otimes}_Kc_0(S^\prime) \cong c_0(S \times S^\prime), && \operatorname{Sym}^{n}_{\heartsuit,K} c_0(S) = c_0(S^n/\Sigma_n),
\end{aligned}
\end{equation*}
for small sets $S, S^\prime$. By \cite[Proposition 2.1.44]{SoorThesis} one has $\mathscr{V}^{\leqslant 0} \simeq \operatorname{sInd}(\mathscr{V}^0)$ and by \cite[Lemma 2.1.48]{SoorThesis} the $t$-structure is right complete. 
\end{example}
\begin{construction}\label{construction:LSymMonad}
Let $(\mathscr{V},\mathscr{V}^{\leqslant 0}, \mathscr{V}^{\geqslant 0} ,\mathscr{V}^0)$ be a derived algebraic context. We recall that Raksit \cite[Construction 4.2.19]{raksit_hochschild_2026}, following Brantner-Mathew \cite{BrantnerMathew}, has constructed a \emph{filtered monad}\footnote{A \emph{filtered monad} on $\mathscr{V}$ is lax monoidal functor $(\mathbf{Z}_{\geqslant 0} ,\times) \to (\operatorname{End}(\mathscr{V}), \circ)$.}  $\operatorname{LSym}^{\leqslant \bullet}_\mathscr{V}$ on $\mathscr{V}$ with the following properties:
\begin{enumerate}[(i)]
    \item For every $i \geqslant 0$, the underlying endofunctor of $\operatorname{LSym}^{\leqslant i}_\mathscr{V}$ is colimit preserving and \emph{excisively polynomial}\footnote{See \cite[Definition 4.2.7]{raksit_hochschild_2026}.}.
    \item For every $i \geqslant 0$ there are natural equivalences $\operatorname{LSym}^{\leqslant i}_\mathscr{V}|_{\mathscr{V}^0} \simeq \operatorname{Sym}^{ \leqslant i}_{\mathscr{V}^\heartsuit}$.
    \item There is a canonical morphism of filtered monads $\theta^{\leqslant \bullet}: \operatorname{Sym}^{\leqslant \bullet}_\mathscr{V} \to \operatorname{LSym}^{\leqslant \bullet}_\mathscr{V}$. Here $\operatorname{Sym}^{\leqslant \bullet}_\mathscr{V}$ is the symmetric powers filtered monad from \cite[Construction 4.1.8]{raksit_hochschild_2026}. 
\end{enumerate}
Taking colimits, we obtain a natural transformation of monads $\theta: \operatorname{Sym}_{\mathscr{V}} \to \operatorname{LSym}_{\mathscr{V}}$. We define the category of \emph{derived (commutative) algebra objects} of $\mathscr{V}$ as $\operatorname{DAlg}(\mathscr{V}) := \operatorname{Mod}_{\operatorname{LSym}_\mathscr{V}}\mathscr{V}$. The morphism $\theta$ induces a functor $\Theta: \operatorname{DAlg}(\mathscr{V}) \to \operatorname{CAlg}(\mathscr{V})$, where $\operatorname{CAlg}(\mathscr{V})$ denotes the $\infty$-category of $\mathbf{E}_\infty$-algebra objects in $\mathscr{V}$. By \cite[Proposition 4.2.27]{raksit_hochschild_2026} the functor $\Theta$ preserves all small limits and colimits. For $A \in \operatorname{DAlg}(\mathscr{V})$ the object $\Theta(A) \in \operatorname{CAlg}(\mathscr{V})$ is called the \emph{underlying $\mathbf{E}_\infty$-algebra of $A$}.
\end{construction}
\begin{rmk}
With notations as in Construction \ref{construction:LSymMonad}. Raksit also constructs for each $i \geqslant 0$ a sifted-colimit preserving, excisively polynomial functor $\operatorname{LSym}^i_\mathscr{V}: \mathscr{V} \to \mathscr{V}$ such that $\operatorname{LSym}^i_\mathscr{V}|_{\mathscr{V}^0} \simeq \operatorname{Sym}^i_{\mathscr{V}^\heartsuit}$. One then has canonical equivalences $\operatorname{LSym}^{\leqslant i}_{\mathscr{V}} \simeq \bigoplus_{j \leqslant i} \operatorname{LSym}^j$ for each $i \geqslant 0$. We may also define for each $i \geqslant 0$ the endofunctor $\operatorname{LSym}^{\geqslant i}_{\mathscr{V}} := \bigoplus_{j \geqslant i} \operatorname{LSym}^j$.
\end{rmk}
\begin{rmk}\label{rmk:connectivediscrete}
We define the category of connective (resp. discrete) derived commutative algebra objects as 
\begin{equation}
\begin{aligned}
    \operatorname{DAlg}^{\leqslant 0}(\mathscr{V}) := \operatorname{DAlg}(\mathscr{V}) \times_\mathscr{V} \mathscr{V}^{\leqslant 0}, && \operatorname{DAlg}^\heartsuit(\mathscr{V}) := \operatorname{DAlg}(\mathscr{V}) \times_\mathscr{V} \mathscr{V}^{\heartsuit}.
\end{aligned}
\end{equation}
As in \cite[Remark 4.2.24]{raksit_hochschild_2026} one has the following: 
\begin{enumerate}[(i)]
    \item There is a canonical equivalence $\operatorname{DAlg}^{\heartsuit}(\mathscr{V}) \simeq \operatorname{CAlg}(\mathscr{V}^\heartsuit)$, where the latter is the (ordinary) category of commutative algebra objects in $\mathscr{V}$;
    \item Let $\operatorname{Poly}^\mathscr{V}$ be the full subcategory of $\operatorname{CAlg}(\mathscr{V}^\heartsuit)$ on objects of the form $\operatorname{Sym}_{\mathscr{V}^\heartsuit} P$ for $P \in \mathscr{V}^0$. Then there is a canonical equivalence $\operatorname{sInd}(\operatorname{Poly}^\mathscr{V}) \simeq \operatorname{DAlg}^{\leqslant 0}(\mathscr{V})$.
\end{enumerate}
\end{rmk}
\begin{construction}\label{construction:DAlgmodules}
With notations as in Construction \ref{construction:LSymMonad}. All of the below is taken from \cite[Notation 4.2.28]{raksit_hochschild_2026}.
\begin{enumerate}[(i)]
    \item For $A \in \operatorname{DAlg}(\mathscr{V})$ we define $\operatorname{Mod}_A \mathscr{V} := \operatorname{Mod}_{\Theta(A)}\mathscr{V}$, we define $\operatorname{DAlg}_A\mathscr{V} := \operatorname{DAlg}(\mathscr{V})_{A/}$, we define the forgetful functor to be the composite $\operatorname{oblv}_A:\operatorname{DAlg}_A\mathscr{V} \to \operatorname{CAlg}(\mathscr{V})_{\Theta(A)/} \to \operatorname{Mod}_A\mathscr{V}$ and we denote by $\operatorname{LSym}_A : \operatorname{Mod}_A\mathscr{V} \to \operatorname{DAlg}_A\mathscr{V}$ its left adjoint. 
    \item The adjunction $ \operatorname{LSym}_A \dashv \operatorname{oblv}_A$ is monadic. The functor $\operatorname{oblv}_A$ is canonically symmetric-monoidal for the cocartesian monoidal structure on the source, therefore, we denote the coproduct on $\operatorname{DAlg}_A\mathscr{V}$ by $\otimes_A$. 
    \item If $A \to B$ is a morphism in $\operatorname{DAlg} (\mathscr{V})$ then there is an adjunction $B \otimes_A (-) :\operatorname{DAlg}_A\mathscr{V} \leftrightarrows \operatorname{DAlg}_B\mathscr{V} : \operatorname{oblv}$ and for every $M \in \operatorname{Mod}_A \mathscr{V}$ there is a canonical equivalence $B \otimes_A\operatorname{LSym}_A M\xrightarrow{\sim}\operatorname{LSym}_B B \otimes_A M$.
\end{enumerate}

\end{construction} 
\section{Analytification as a geometric morphism of $\infty$-topoi}\label{sec:Analytification}
\egapar In this section we use the sheaf-theoretic operations developed in \S\ref{sec:RestrictionExtension} to construct an analytification functor with the best possible functoriality properties. To be precise, the analytification functor is implemented as a \emph{geometric morphism of $\infty$-topoi}. Once we know that the analytification of the affine line is the right thing, it then becomes formal\footnote{Under appropriate finite-type hypotheses.} to see that the analytification functor sends (derived) schemes to (derived) analytic spaces and, more generally, sends (derived) algebraic stacks to (derived) analytic stacks.
To construct the analytification functor, the idea is to write it as a composite of two geometric morphisms, using a carefully chosen topos which is intermediate between the analytic and algebraic topoi. This is the purpose of the topology introduced in Definition \ref{defn:dri}. In fact, this intermediate topos has some nice properties: for instance, it admits a good definition of ``affine morphism", see Proposition \ref{prop:NCrepresentable}.

\egapar Let $K$ be a non-trivially valued non-Archimedean field with ring of integers $o$ and pseudo-uniformizer $\varpi$. \emph{From now until the end of this article we fix $\mathscr{V} = D(\operatorname{IndBan}_K)$ and $(\mathscr{V}, \mathscr{V}^{\leqslant 0}, \mathscr{V}^{\geqslant 0}, \mathscr{V}^0)$ denotes the derived algebraic context from Example \ref{example:IndBancontext}}. 
\begin{notation}\label{notation:Aff}
We define $\operatorname{Aff} := \operatorname{DAlg}(\mathscr{V})^\mathrm{op}$. For $A \in \operatorname{DAlg}(\mathscr{V})$ we denote the corresponding object of $\operatorname{Aff}$ by the formal expression $\operatorname{Spec}(A)$.
\end{notation}
\egapar Note that Notation \ref{notation:Aff} allows for ``nonconnective affines".
\begin{defn}\cite[Definition 3.1.1]{SoorThesis}\label{defn:DerivedAffinoid}
\begin{enumerate}[(i)]
    \item We let $\operatorname{AfndAlg}$ denote the full subcategory of $\operatorname{DAlg}(\mathscr{V})$ on objects such that $A \in \operatorname{DAlg}^{\leqslant 0}(\mathscr{V})$ is connective, $H^0A$ is a classical $K$-affinoid algebra and for every $n \leqslant 0$, $H^{n}A$ is finitely-generated as a $H^0A$-module. 
    \item We let $\mathrm{Afnd} \subseteq \mathrm{Aff}$ denote the corresponding full subcategory of $\mathrm{Aff}$, which we call the category of \emph{derived affinoid spaces}.
\end{enumerate}
\end{defn} 
\begin{lem}
The subcategory $\mathrm{Afnd} \subseteq \operatorname{Aff}$ is stable under fiber products, finite coproducts and retracts in $\mathrm{Aff}$.
\end{lem}
\begin{proof}
This is \cite[Lemma 3.1.2]{SoorThesis}. 
\end{proof}
\begin{defn}
\begin{enumerate}[(i)]
    \item Let $A \in \operatorname{DAlg}^{\leqslant 0}(\mathscr{V})$ and let $M \in \operatorname{Mod}_A \mathscr{V}^{\leqslant 0}$. We say that $M$ is \emph{derived strong} if, for every $n \leqslant 0$ the canonical morphism $H^n (A) \hat{\otimes}^\mathbf{L}_{H^0(A)}H^0(M) \to H^n(M)$ is an equivalence. 
    \item We say that a morphism $A \to B$ in $\operatorname{DAlg}^{\leqslant 0}(\mathscr{V})$ is \emph{derived strong} if $B$ is derived strong when regarded as an object of $\operatorname{Mod}_A \mathscr{V}^{\leqslant 0}$.
\end{enumerate}
\end{defn}
\begin{defn}
We say that a morphism $A \to B$ in $\operatorname{AfndAlg}$ is a \emph{derived rational localization} if it is derived strong, and $H^0(A) \to H^0(B)$ is a rational localization of classical affinoid algebras. The corresponding morphism $\operatorname{Spec}(B) \to \operatorname{Spec}(A)$ in $\operatorname{Afnd}$ is called a \emph{rational subdomain} or \emph{rational subspace}.
\end{defn}
\begin{lem}\label{lem:derivedrational}
Let $f: A \to B$ be morphism in $\operatorname{AfndAlg}$. Then the following are equivalent:
\begin{enumerate}[(i)]
    \item $f$ is a derived rational localization;
    \item There exists morphisms $a_0, \dots, a_n: K \to A$ whose images in $H^0(A)$ generate the unit ideal, such that $A \to B$ is equivalent to the morphism
    \begin{equation}
        A \to (A \hat{\otimes}^\mathbf{L}_K K\langle T_1, \dots, T_n\rangle)/^\mathbf{L}(a_0T_1-a_1, \dots, a_0T_n-a_n). 
    \end{equation}
\end{enumerate}
\end{lem}
\begin{proof}
This is \cite[Lemma 3.1.9]{SoorThesis}.
\end{proof}
\begin{defn}
In the setting of Lemma \ref{lem:derivedrational}(ii). We may use the notation
\begin{equation}
    A \langle a_1\dots a_n /a_0\rangle := (A \hat{\otimes}^\mathbf{L}_K K\langle T_1, \dots, T_n\rangle)/^\mathbf{L}(a_0T_1-a_1, \dots, a_0T_n-a_n),
\end{equation}
and call this the \emph{derived rational localization} at $a_0, \dots, a_n$. 
\end{defn}
\begin{defn}\cite[Definition 3.1.12]{SoorThesis} We define the \emph{rational Grothendieck topology}, denoted $\tau_{\mathrm{rat}}$, on $\operatorname{Afnd}$ as follows. The covering sieves are generated by finite families $\{\operatorname{Spec}(A_i) \to \operatorname{Spec}(A)\}_i$ of rational subspaces such that $\{\operatorname{Spec}(H^0(A_i)) \to \operatorname{Spec}(H^0(A))\}_i$ is a rational cover of the classical rigid space $\operatorname{Spec}(H^0(A))$. 
\end{defn}
\begin{defn}
\begin{enumerate}[(i)]
    \item We say that a morphism $A \to B$ in $\operatorname{DAlg}(\mathscr{V})$ is \emph{descendable} if the induced morphism $\Theta(A) \to \Theta(B)$ is a descendable morphism of $\mathbf{E}_\infty$-algebra objects in the sense of \cite[\S 3.3]{MathewGalois}. 
    \item We define the descendable topology, denoted $\tau_{\mathrm{desc}}$, on $\operatorname{Aff}$ in the following way: the covering sieves are generated by finite families $\{\operatorname{Spec}(A_i) \to \operatorname{Spec}(A)\}_i$ such that the induced morphism $A \to \prod_i A_i$ is descendable in the sense of (i). 
\end{enumerate}
\end{defn}
\begin{lem}\label{lem:NCdescent}
The functor $\operatorname{Aff}^\mathrm{op} \to \operatorname{Pr}^{L, \otimes}: \operatorname{Spec}(A) \mapsto \operatorname{DAlg}_A\mathscr{V}$ is a sheaf in the topology $\tau_{\mathrm{desc}}$. In particular the topology $\tau_{\mathrm{desc}}$ is subcanonical. 
\end{lem} 
\begin{proof}
Let $A \to B$ be a descendable morphism in $\operatorname{DAlg}(\mathscr{V})$. For brevity we put $B^{(n)} := B^{\hat{\otimes}^\mathbf{L}_An}$. We need to show that the canonical adjunction 
\begin{equation}\label{eq:DAlgdescent}
\operatorname{DAlg}_{A}\mathscr{V} \leftrightarrows \underset{[n] \in \Delta}{\operatorname{lim}}\operatorname{DAlg}_{B^{(n+1)}} \mathscr{V}
\end{equation}
is an equivalence of categories. For fully-faithfulness of the left adjoint, let $C \in \operatorname{DAlg}_{A}\mathscr{V}$; we need to show that the canonical morphism $C \to \operatorname{lim}_{[n] \in \Delta} C \hat{\otimes}^\mathbf{L}_A B^{(n+1)}$ is an equivalence. The forgetful functor $\operatorname{DAlg}_{A}\mathscr{V} \to \operatorname{Mod}_A \mathscr{V}$ is conservative, preserves limits and is strongly monoidal for the cocartesian monoidal structure on the source, c.f. \cite[Notation 4.2.28]{raksit_hochschild_2026}, hence, this follows from \cite[Proposition 3.22]{MathewGalois}. For essential surjectivity of the left adjoint, let $(C_n)_{[n] \in \Delta}$ be an object of the right-hand side of \eqref{eq:DAlgdescent}; we need to show for each $m \geqslant 0$ that the canonical morphism $C_m \to B^{(m+1)} \hat{\otimes}_A^\mathbf{L}\operatorname{lim}_{[n] \in \Delta} C_n$ is an equivalence. Again, since the forgetful functor $\operatorname{DAlg}_{A}\mathscr{V} \to \operatorname{Mod}_A \mathscr{V}$ is conservative, preserves limits and is strongly monoidal, this follows from \cite[Proposition 3.22]{MathewGalois}. 
\end{proof}
\begin{rmk}\label{rmk:connectivedescent}
Lemma \ref{lem:NCdescent} does not refer to any notion of connectivity, and is the reason why we work with nonconnective affines. Indeed \cite[Example 9.1.2]{BerkeleyLectures} can be used to show that the moduli of connective derived algebras does not satisfy descent. Let us explain the reasoning. We put $B = K \langle T, S\rangle$ consider the cover of $X = \mathbf{D}^2 = \operatorname{Spec}(B)$ by affinoid subsets $U_0 = \{ |T| \leqslant |\varpi| \text{ and } |S| \leqslant |\varpi|\}$, $U_1 = \{ |T| \geqslant |\varpi| \text{ and } |S| \leqslant |T|\}$, $U_2 = \{ |S| \geqslant |\varpi| \text{ and } |T| \leqslant |S|\}$. We put $A_i := \mathscr{O}(U_i)$, $A := \prod_{i= 0}^2 A_i$ and $A^{(n)} := A^{\hat{\otimes}_B^\mathbf{L}n}$.

The open subset $V := \{ |T| = 1 \text{ or } |S| = 1\}$ satisfies $V \cap U_0 = \emptyset$, $V \cap U_1 = \{ |T| \leqslant 1\}$ and $V \cap U_2 = \{ |S| \leqslant 1\}$. These are all affinoid, hence $V$ determines an object of the category of descent datum, that is, $\operatorname{lim}_{[n] \in \Delta} \operatorname{DAlg}_{A^{(n+1)}}^{\leqslant 0}\mathscr{V} $. The image of this object under the functor to $\operatorname{DAlg}^{\leqslant 0}_{K\langle T, S\rangle}\mathscr{V}$, given by taking the limit, is the connective truncation of the \v{C}ech complex associated to this covering of $V$, which computes $\mathscr{O}(V) = K \langle T,S \rangle$. But the pullback of $K \langle T,S \rangle$ to $U_0$ is not $0$. This shows that the natural adjunction
\begin{equation}
\operatorname{DAlg}_{K \langle T, S \rangle }^{\leqslant 0}\mathscr{V} \leftrightarrows  \underset{{[n] \in \Delta}}{\operatorname{lim}}\operatorname{DAlg}_{A^{(n+1)}}^{\leqslant 0}\mathscr{V}    
\end{equation}
is not an equivalence of categories. However, Lemma \ref{lem:NCdescent} shows that the moduli of \emph{nonconnective} derived algebras does satisfy descent, in particular, the space $V$ is nonconnectively affine. This positive assertion is reminiscent of the example in \cite[\S 2.2]{ToenChampsAffines} and \cite[Remark 2.1.9]{BraveNew}, and in some sense gives a solution to Problem (2) described in the introduction to \cite{XiaLiuMorphism}. 
\end{rmk}
\egapar We denote by $i$ the tautological inclusion functor $i: \mathrm{Afnd} \hookrightarrow \mathrm{Aff}$.
\begin{prop}
The inclusion functor $i: \operatorname{Afnd} \to \operatorname{Aff}$ is continuous for the rational topology on the source and the descendable topology on the target. 
\end{prop}
\begin{proof}
It is sufficient to show that, if $\{\operatorname{Spec}(A_i) \to \operatorname{Spec}(A)\}_i$ is a finite collection of rational localizations generating a $\tau_{\mathrm{rat}}$-covering sieve, then the canonical morphism $A \to \prod_i A_i$ is descendable. This follows from (the proof of) \cite[Lemma 3.1.16]{SoorThesis}. 
\end{proof}
\begin{defn}\label{defn:dri}
We define the the \emph{descendable and rational-induced topology}, denoted $\tau_{\mathrm{dri}}$, on $\operatorname{Aff}$ in the following way. By Lemma \ref{lem:cocontinuousinitial} the rational topology induces a Grothendieck topology $i_*\tau_{\mathrm{rat}}$ on $\mathrm{Aff}$. We define $\tau_{\mathrm{dri}} := i_*\tau_{\mathrm{rat}} \cap \tau_{\mathrm{desc}}$ on $\mathrm{Aff}$. In other words, the covering sieves in $\tau_{\mathrm{dri}}$ are the ones which belong to both $i_*\tau_{\mathrm{rat}}$ and $\tau_{\mathrm{desc}}$.
\end{defn}
\begin{prop}\label{prop:analyticinclusion}
With notations as above. There is an adjoint triple
    \begin{equation}
        (i_! \dashv i^* \dashv i_*): \operatorname{Shv}_{\tau_\mathrm{rat}}(\mathrm{Afnd}) \leftrightarrows \operatorname{Shv}_{\tau_{\mathrm{dri}}}(\mathrm{Aff}),
    \end{equation}
    of functors $i_!, i_*: \operatorname{Shv}_{\tau_\mathrm{rat}}(\mathrm{Afnd}) \to \operatorname{Shv}_{\tau_\mathrm{dri}}(\mathrm{Aff})$ and $i^*: \operatorname{Shv}_{\tau_\mathrm{dri}}(\mathrm{Aff}) \to \operatorname{Shv}_{\tau_\mathrm{rat}}(\mathrm{Afnd})$. Further, $i_!$ is left-exact, the functors $i_!$ and $i_*$ are fully-faithful, and $i_!$ preserves representable objects.
\end{prop}
\begin{proof}
This follows from Theorem \ref{thm:pushforwardtopology}, Lemma \ref{lem:fullyfaithful} and Corollary \ref{cor:lowershriekrepresentable}. 
\end{proof}
\begin{prop}\label{prop:NCrepresentable}
Let $f: X \to Y$ be a morphism in $\operatorname{Shv}_{\tau_{\mathrm{dri}}}(\operatorname{Aff})$. Then the following are equivalent:
\begin{enumerate}[(i)]
\item $f$ is representable: that is, for every object $Y^\prime = \operatorname{Spec}(A)$ in $\operatorname{Aff}$ with a map $Y^\prime \to Y$, the pullback $X^\prime := X\times_Y Y^\prime$ belongs to $\operatorname{Aff}$.
\item $f$ is representable locally on $Y$: that is, there exists a cover $\{U_i = \operatorname{Spec}(A_i) \to Y\}_{i}$ of $Y$ by objects of $\operatorname{Aff}$, such that each pullback $X_i := Y \times_X U_i$ belongs to $\operatorname{Aff}$. 
\end{enumerate}
\end{prop}
\begin{proof}
That (i) implies (ii) is trivial. To see that (ii) implies (i) we argue as in \cite[Lemma 9.3.1.1]{SpectralAlgebraicGeometry}. Put $U = \coprod_i U_i$ and let $ Y^\prime=\operatorname{Spec}(B) \to Y$ be a morphism from an object of $\operatorname{Aff}$. Then there exists a descendable morphism $V:= \operatorname{Spec}(C) \to \operatorname{Spec}(B)$ such that the composite $V \to Y^\prime \to Y$ factors over $U \to Y$. Consider the simplicial object $[n] \mapsto Z_n :=  V^{n+1/Y^\prime} \times_{Y} X$. By assumption, each morphism $V^{{n+1}/Y^\prime} \to Y$ factors through $U$, so we can write $Z_n = \operatorname{Spec}(D^{(n)})$ for some object $(D^{(n)})_{[n] \in \Delta} \in \operatorname{lim}_{[n] \in \Delta} \operatorname{DAlg}_{C^{(n+1)}} \mathscr{V}$; here we have abbreviated $C^{(n)} := C^{\hat{\otimes}^\mathbf{L}_B n}$. By Lemma \ref{lem:NCdescent} we deduce that there exists an object $D \in \operatorname{DAlg}_B \mathscr{V}$ such that $C^{(n+1)} \hat{\otimes}^\mathbf{L}_B D \simeq D^{(n)}$ for every $n \geqslant 0$. Then 
\begin{equation*}
    Z \simeq \underset{[n] \in \Delta^\mathrm{op}}{\operatorname{colim}} Z_n \simeq  \underset{[n] \in \Delta^\mathrm{op}}{\operatorname{colim}} \operatorname{Spec}(D^{(n)}) \simeq \operatorname{Spec}(D) \times_{Y^\prime} \underset{[n] \in \Delta^\mathrm{op}}{\operatorname{colim}} \operatorname{Spec}(C^{(n+1)}) \simeq \operatorname{Spec}(D),
\end{equation*}
hence, $f$ is representable. 
\end{proof}
\egapar We note that the same counterexample as in Remark \ref{rmk:connectivedescent} can be used to show that there is no connective analogue of Proposition \ref{prop:NCrepresentable} either.  
\begin{cor}
Let $f: X \to Y$ be a representable morphism in $\operatorname{Shv}_{\tau_{\mathrm{rat}}}(\operatorname{Afnd})$. Then the image $i_!(f) : i_!X \to i_!Y$ of $f$ under the functor $i_!: \operatorname{Shv}_{\tau_{\mathrm{rat}}}(\operatorname{Afnd}) \to \operatorname{Shv}_{\tau_{\mathrm{dri}}}(\operatorname{Aff})$ is representable.
\end{cor}
\begin{proof}
Since $i_!$ is colimit preserving, it takes covers to covers. Further, $i_!$ preserves representable objects. So $i_!(f)$ is representable locally on $Y$ in the sense of Proposition \ref{prop:NCrepresentable}. But by Proposition \ref{prop:NCrepresentable}, this implies that $i_!(f)$ is representable.
\end{proof}
\egapar We denote by $\operatorname{AniAlg}$ be the category of (abstract) animated commutative $K$-algebras. We equip $\mathrm{AniAlg}^\mathrm{op}$ with the Zariski topology and denote the corresponding site by $(\mathrm{AniAlg}^\mathrm{op},\tau_{\mathrm{Zar}})$. We let $\operatorname{Poly}^a$ denote the category of abstract polynomial algebras over $K$. There is a fully-faithful coproduct-preserving functor $\operatorname{Poly}^a \to \operatorname{Poly}^\mathscr{V} \subseteq \operatorname{DAlg}^{\leqslant 0}(\mathscr{V})$, which endows a polynomial algebra with the fine bornology (here we use notations as in Remark \ref{rmk:connectivediscrete}). Extending by sifted colimits, we get a fully-faithful colimit-preserving functor $\operatorname{AniAlg} \to \operatorname{DAlg}^{\leqslant 0}(\mathscr{V})$, which we may post-compose with the (colimit-preserving) inclusion $\operatorname{DAlg}^{\leqslant 0}(\mathscr{V}) \hookrightarrow \operatorname{DAlg}(\mathscr{V})$ to obtain a fully-faithful colimit-preserving functor $\alpha: \operatorname{AniAlg} \to \operatorname{DAlg}(\mathscr{V})$. Passing to opposite categories we obtain a fully-faithful limit-preserving functor 
\begin{equation}
a: \operatorname{AniAlg}^\mathrm{op} \to \operatorname{Aff}. 
\end{equation}
\begin{lem}
The functor $a: \operatorname{AniAlg}^\mathrm{op} \to \operatorname{Aff}$ is continuous for the Zariski topology on the source and the descendable topology on the target. 
\end{lem}
\begin{proof}
This is well-known: It follows from the boundedness of the alternating \v{C}ech complex associated to a Zariski covering. 
\end{proof}
\egapar The following Lemma is the most important technical ingredient in the construction of the analytification functor.
\begin{lem}\label{lem:alginclusioncontinuity}
The functor $a: \operatorname{AniAlg}^\mathrm{op} \to \operatorname{Aff}$ is continuous for the Zariski topology on the source and the topology $\tau_{\mathrm{dri}}$ on the target. Consequently there is an induced adjunction
\begin{equation}
    a_!: \operatorname{Shv}_{\tau_{\mathrm{Zar}}}(\mathrm{AniAlg}^\mathrm{op}) \leftrightarrows  \operatorname{Shv}_{\tau_{\mathrm{dri}}}(\operatorname{Aff}) :a^*,
\end{equation}
in which the left adjoint $a_!$ is left-exact.
\end{lem}
\begin{proof}
It is sufficient to show that the functor $a$ is continuous for the topology $\tau_{\mathrm{Zar}}$ on $\operatorname{AniAlg}^\mathrm{op} $ and $i_*\tau_{\mathrm{rat}}$ on $\operatorname{Aff}$. For the sake of brevity let us put $\mathscr{C} := \operatorname{AniAlg}^\mathrm{op}$, $\mathscr{D} := \operatorname{Aff}$ and $\mathscr{E} := \mathrm{Afnd}$. Let $X = \operatorname{Spec}(A) \in \mathscr{C}$ and let $\mathscr{C}^0_{/X}$ be a Zariski covering sieve on $X$. To be explicit this means the following: there exists a finite collection $\{A \to A_j\}_j$ of Zariski localizations which induces a jointly surjective morphism on Zariski spectra, such that $\{X_j \to X\}_j$ generates $\mathscr{C}^0_{/X}$ where $X_j = \operatorname{Spec}(A_j)$. By passing to a refinement if necessary, we may and will assume that the $A_j$ are of the form $A_j = A[f_j^{-1}]$ where the $f_j: K \to A$ are morphisms\footnote{In the derived category of $K$-vector spaces.} which generate the unit ideal in $H^0(A)$. Denote by $\mathscr{D}^1_{/a(X)}$ the sieve generated by the image of $\mathscr{C}^0_{/X} \to \mathscr{D}_{/a(X)}$. By Lemma \ref{lem:continuouscriterion}, it is sufficient to show that $\mathscr{D}^1_{/a(X)}$ is a $i_*\tau_{\mathrm{rat}}$-covering sieve.  Let $Y = \operatorname{Spec}(B) \in \operatorname{Afnd}$ and let $i(Y) \to a(X)$ be any morphism. We need to show that $\mathscr{E}^1_{/Y}:= \mathscr{E}_{/Y} \times_{\mathscr{D}_{/i(Y)}} \mathscr{D}^1_{/a(X)}$ is a $\tau$-covering sieve. Put $Z_j := a(X_j) \times_{a(X)} i(Y)$; then each $Z_j \to i(Y)$ is induced by\footnote{Here we abuse notation and write $f_j$ for the composite morphism $K \to A \to B$ in $\mathscr{V}$.} some $B \to B[f_j^{-1}]$, where the $f_j$ generate the unit ideal in $H^0(B)$. We define $B_j := B\langle f_1 \dots \hat{f}_j \dots f_n/f_j\rangle$ to be the rational localisation and $Y_j := \operatorname{Spec}(B_j)$. Then for each $j$ there is a factorization $i(Y_j) \to  Z_j \to i(Y)$, so $Y_j \to Y$ belongs to $\mathscr{E}^1_{/Y}$. But since the $f_j$ generate the unit ideal in $H^0(B)$, the collection $\{Y_j \to Y\}_j$ generates a $\tau_{\mathrm{rat}}$-covering sieve.
\end{proof}
\begin{defn}
The \emph{analytification functor} is defined to be 
\begin{equation}
   (-)^\mathrm{an} :=i^*a_! :  \operatorname{Shv}_{\tau_{\mathrm{Zar}}}(\operatorname{AniAlg}^\mathrm{op}) \to \operatorname{Shv}_{\tau_{\mathrm{rat}}}(\operatorname{Afnd}). 
\end{equation}
\end{defn}
The counit morphism $i_!i^* \to \operatorname{id}$ gives a morphism $i_!i^*a_! \to a_!$. So, the analytification and the original space can be compared in $\operatorname{Shv}_{\tau_{\mathrm{dri}}}(\operatorname{Aff})$. As a consequence of Lemma \ref{lem:alginclusioncontinuity} and Proposition \ref{prop:analyticinclusion} we obtain the following, which justifies the name of this section:
\begin{cor}
The analytification functor $(-)^\mathrm{an}$ is a left exact left adjoint. 
\end{cor}
\egapar Aside from the excellent functorial properties of the analytification functor, we need to check that it ``does the right thing". This can be reduced to the case of the algebraic affine line, which we now compute using the universal property deduced from adjunction. 
\begin{prop}\label{prop:affineline}
One has $\operatorname{Spec}(K[T])^\mathrm{an} = \underset{n \in \mathbf{N}}{\operatorname{colim}} \operatorname{Spec}(K \langle \varpi^n T \rangle ) =: \mathbf{A}^{1,\mathrm{an}.}$ is\footnote{Here the colimit is taken in $\operatorname{Shv}_{\tau_{\mathrm{rat}}}(\mathrm{Afnd})$.} the analytic affine line. 
\end{prop}
\begin{proof}
Let $\operatorname{Spec}(A) \in \operatorname{Afnd}$; then by adjunction one has 
\begin{equation}
\begin{aligned}
    \operatorname{Map}(\operatorname{Spec}(A), \operatorname{Spec}(K[T])^\mathrm{an}) &= \operatorname{Map}(i_!\operatorname{Spec}(A), a_!\operatorname{Spec}(K[T])) \\
    &\simeq \operatorname{Map}(K[T], A)
\end{aligned}
\end{equation}
where we used Corollary \ref{cor:lowershriekrepresentable} together with the fact that the descendable topology is subcanonical. Because the rational topology is finitary one has
\begin{equation}
    \underset{n \in \mathbf{N}}{\operatorname{colim}} \operatorname{Map}(\operatorname{Spec}(A),\operatorname{Spec}(K \langle \varpi^nT \rangle )) \xrightarrow[]{\sim} \operatorname{Map}(\operatorname{Spec}(A),\underset{n \in \mathbf{N}}{\operatorname{colim}} \operatorname{Spec}(K \langle \varpi^nT \rangle )).
\end{equation}
We claim that the canonical map
\begin{equation}\label{eq:affinelineiso}
    \underset{n \in \mathbf{N}}{\operatorname{colim}}\operatorname{Map}(K \langle \varpi^nT \rangle , A) \to \operatorname{Map}(K[T],A)
\end{equation}
is an equivalence. Since each $K[T] \to K \langle \varpi^nT \rangle $ is a homotopy epimorphism, it is certainly true that \eqref{eq:affinelineiso} is a monomorphism. Hence, it suffices to show that \eqref{eq:affinelineiso} is surjective on connected components. This is clear if $A$ is a classical affinoid, as it is an expression of the equality of underlying sets:
\begin{equation}
\bigcup_{n  \in \mathbf{N}} \{ a \in A : \varpi^na \text{ is power-bounded}\} = A. 
\end{equation}
In general, take a morphism $[K \langle \varpi^n T\rangle \to H^0(A)]$. By restriction we obtain $[K[T] \to H^0(A)]$, which lifts to $[K[T] \to A]$, since $K[T]$ is projective. But $[K[T] \to K\langle \varpi^n T \rangle]$ is a homotopy epimorphism, hence formally \'etale, so \cite[Corollary 2.1.36]{DAnG} guarantees that $[K\langle \varpi^n T \rangle \to H^0(A)]$ lifts to a morphism $[K \langle \varpi^n T  \rangle \to A]$ under $K[T]$. This shows that the left arrow in the commutative square 
\begin{equation}
\begin{tikzcd}
	{\underset{n \in \mathbf{N}}{\operatorname{colim}}\operatorname{Map}(K\langle \varpi^n T\rangle, A)} & {\operatorname{Map}(K[T],A)} \\
	{\underset{n \in \mathbf{N}}{\operatorname{colim}}\operatorname{Map}(K\langle  \varpi^n T \rangle, H^0(A))} & {\operatorname{Map}(K[T],H^0(A))}
	\arrow[from=1-1, to=1-2]
	\arrow[from=1-1, to=2-1]
	\arrow[from=1-2, to=2-2]
	\arrow[from=2-1, to=2-2]
\end{tikzcd}
\end{equation}
is surjective on connected components. The bottom arrow is an isomorphism, and the right arrow is an isomorphism on connected components, since $K[T]$ is projective. Hence the top arrow is surjective on connected components.   
\end{proof}
\egapar We recall the definition of derived rigid spaces from \cite[Definition 3.1.19]{SoorThesis}, which we reproduce here for the reader's convenience. 
\begin{defn}\label{defn:derivedrigidspace}
\begin{enumerate}[(i)]
    \item An \emph{affinoid rigid space} is an object $X = \operatorname{Spec}(A)$ in the image of the Yoneda embedding $\operatorname{Afnd} \hookrightarrow \operatorname{Shv}_{\tau_\mathrm{rat}}(\operatorname{Afnd})$. 
    \item Let $X = \operatorname{Spec}(A)$ be an affinoid rigid space. An \emph{analytic subspace} of $X$ is a subsheaf $U \hookrightarrow X$ such that there exists a small collection $\{U_i \hookrightarrow U\}_i$ of rational subdomains of $X$ such that $\coprod_i U_i \twoheadrightarrow U$ is an effective epimorphism.
    \item Let $X \in \operatorname{Shv}_{\tau_\mathrm{rat}}(\operatorname{Afnd})$. An \emph{analytic subspace} of $X$ is a subsheaf $U \hookrightarrow X$ such that, for every affinoid rigid space $X^\prime \to X$ mapping to $X$, the pullback $U^\prime \hookrightarrow X^\prime$ is an analytic subspace in the sense of (ii) above; 
    \item A \emph{derived rigid space} is an object $X \in \operatorname{Shv}_{\tau_\mathrm{rat}}(\operatorname{Afnd})$ such that there exists a small family $\{U_i \hookrightarrow X\}$ of affinoid analytic subspaces such that $\coprod_i U_i \twoheadrightarrow X$ is an effective epimorphism. We denote the category of derived rigid spaces by $\operatorname{dRig}$.
    \item The \emph{strong} or \emph{analytic topology} on $\mathrm{dRig}$ is defined so that covering sieves are generated by small families $\{U_i \hookrightarrow X\}_i$ of analytic subspaces such that $\coprod_i U_i \twoheadrightarrow X$ is an effective epimorphism. 
\end{enumerate}
\end{defn}
\begin{rmk}\label{rmk:rigidfinitelimits}
We remark that the category $\operatorname{dRig}$ is stable under finite limits and retracts in $\operatorname{Shv}_{\tau_\mathrm{rat}}(\operatorname{Afnd})$. 
\end{rmk}
\begin{defn}\label{defn:HomotopicallyFinite}
An object $A \in \operatorname{AniAlg}$ is called \emph{homotopically of finite presentation} if $A$ is compact as an object of the $\infty$-category $\operatorname{AniAlg}$. 
\end{defn}
\egapar We denote by $K[-]$ the left adjoint to the forgetful functor $\operatorname{AniAlg} \to \infty\mathrm{Grpd}_*$ to pointed $\infty$-groupoids. For the purposes of the next Definition, it is better to think of $\infty\mathrm{Grpd}_*$ as the $\infty$-category of pointed spaces. For $n \geqslant 0$ we denote the pointed $n$-sphere by $S^n$ and the pointed $n$-disk by $D^n$. Of course, one has $D^n \simeq *$. 
\begin{defn}\label{defn:cellular}\cite[Appendix B]{hekking2021gradedalgebrasprojectivespectra}.
\begin{enumerate}[(i)]
    \item Let $A, B \in \operatorname{AniAlg}$ and let $\sigma: S^{n-1} \to A$ be a morphism of pointed spaces. If $n \geqslant 1$ we say that \emph{$B$ is obtained from $A$ by attaching an $n$-cell along $\sigma$} if there is a pushout square in $\operatorname{AniAlg}$:
    \begin{equation}
\begin{tikzcd}
	{K[S^{n-1}]} & {K[D^n] \simeq K} \\
	A & B
	\arrow[from=1-1, to=1-2]
	\arrow["\sigma"', from=1-1, to=2-1]
	\arrow[from=1-2, to=2-2]
	\arrow[from=2-1, to=2-2]
	\arrow["\lrcorner"{anchor=center, pos=0.125, rotate=180}, draw=none, from=2-2, to=1-1]
\end{tikzcd}
    \end{equation}
    If $n = 0$ we say that $A[T]$ is obtained from $A$ by attaching a $0$-cell.
    \item We say that a morphism $A \to B$ in $\operatorname{AniAlg}$ is \emph{finite cellular} if $B$ is obtained from $A$ by finitely many cell attachments, including $0$-cells. We say that $A \in \operatorname{AniAlg}$ is \emph{finite cellular} if the morphism $K \to A$ is finite cellular.
\end{enumerate}
\end{defn}
\begin{prop}\label{prop:AniAlgcompact}
Let $A \in \operatorname{AniAlg}$. Then the following are equivalent.
\begin{enumerate}[(i)]
    \item $A$ is homotopically of finite presentation.
    \item There exists a finite diagram $p:I \to \operatorname{Poly}_K$ such that $A$ is a retract of $\operatorname{colim}p$. Here the colimit is taken in $\operatorname{AniAlg}$.  
    \item $A$ is a retract of a finite cellular $K$-algebra.
\end{enumerate}
\end{prop}
\begin{proof}
This is \cite[Construction B.4.1]{hekking2021gradedalgebrasprojectivespectra}, see also \cite[Remark B.8.4]{hekking2021gradedalgebrasprojectivespectra}. Note that what we call \emph{homotopically of finite presentation} is called \emph{locally of finite presentation} in \emph{loc. cit.}. 
\end{proof}
\egapar In an entirely analogous way to Definition \ref{defn:derivedrigidspace}, we regard the category of derived schemes over $K$ as a full subcategory of $\operatorname{Shv}_{\tau_{\mathrm{Zar}}}(\operatorname{AniAlg}^\mathrm{op})$. We do not give the details as this construction is well-known.
\begin{defn}
A derived scheme is called \emph{locally homotopically of finite presentation} if there exists a Zariski cover $\coprod_i U_i \twoheadrightarrow X$ by affine open subschemes $U_i = \operatorname{Spec}(A_i) \subseteq X$ such that each $A_i \in \operatorname{AniAlg}$ is homotopically of finite presentation in the sense of Definition \ref{defn:HomotopicallyFinite}. 
\end{defn}
\begin{prop}
Let $X$ be a derived scheme (over $K$) which is locally homotopically of finite presentation. Then $X^\mathrm{an}$ is a derived rigid space. 
\end{prop}
\begin{proof}
We recall that the category $\operatorname{dRig}$ is stable under finite limits and retracts in $\operatorname{Shv}_{\tau_{\mathrm{rat}}}(\operatorname{Afnd})$, c.f. Remark \ref{rmk:rigidfinitelimits}. By Proposition \ref{prop:AniAlgcompact}, if $X = \operatorname{Spec}(A) \in \operatorname{AniAlg}^\mathrm{op}$ is homotopically of finite presentation, then there is a finite diagram $p: I \to \operatorname{Poly}_K$ such that $A$ is a retract of $\operatorname{colim}p$. Using left exactness of $(-)^\mathrm{an}$, to show that $X^\mathrm{an}$ is representable by a derived rigid space, we reduce to the cases when $A = K[T]$ and $A = K$. This follows from Proposition \ref{prop:affineline}. 

Further, if $A \to A[f^{-1}]$ is a the Zariski localization at some morphism $f:K \to A$, then $Y = \operatorname{Spec}(A[f^{-1}])$ is also homotopically of finite presentation so $Y^{\mathrm{an}}$ is also a derived rigid space. The morphism $Y^{\mathrm{an}} \to X^\mathrm{an}$ is the base-change of $\mathbf{G}_m^\mathrm{an} \to \mathbf{A}^{1, \mathrm{an}}$, which, by direct calculation, is an analytic subspace. Hence $Y^\mathrm{an}$ is an analytic subspace of $X^\mathrm{an}$. Now if $Y \subseteq X$ is a general Zariski open subspace, one can cover $Y$ by affine Zariski opens of $X$ of the form considered above, and we deduce that $Y^\mathrm{an} \subseteq X^\mathrm{an}$ is an analytic subspace; note that $(-)^\mathrm{an}$ preserves effective epimorphisms as it is a left-exact left adjoint. 

In the general case when $X$ is a derived scheme which is locally homotopically of finite presentation, one chooses a cover of $X$ by open affine subspaces $Y_i \subseteq X$ where each $Y_i$ is homotopically of finite presentation. By the above discussion each $Y_i^\mathrm{an} \subseteq X^\mathrm{an}$ is an analytic subspace. Further, as $(-)^\mathrm{an}$ is a left-exact left adjoint, $\coprod_i Y_i^\mathrm{an} \to X^\mathrm{an}$ is an effective epimorphism, so $X^\mathrm{an}$ is a derived rigid space. 
\end{proof}
\begin{example}
One has $(\mathbf{G}^{\mathrm{alg}}_m)^\mathrm{an} = \mathbf{G}^{\mathrm{an}}_m$ and $(\mathbf{P}^{1, \mathrm{alg}})^\mathrm{an} = \mathbf{P}^{1,\mathrm{an}}$.
\end{example}
\section{On classical truncation}\label{sec:classical}
\egapar We recall from \cite{SoorThesis} that every classical rigid space may be regarded as a derived rigid space, and further that the inclusion of classical rigid spaces admits a right adjoint, which is the functor of \emph{classical truncation}. The purpose of this section is to use the sheaf theory of \S\ref{sec:RestrictionExtension} to better understand this functor. This functor is an indispensable tool, because of the relatively inexplicit nature of derived geometry, and also because of the ansatz that ``all the geometry happens on the classical truncation". As an example of the latter we recall the invariance of the underlying topological space in Theorem \ref{thm:topologicalinvariance}.

\egapar Let $\operatorname{AfndAlg}_{\mathrm{cl}}$ denote the category of classical affinoid algebras and define $\operatorname{Afnd}_{\mathrm{cl}} := \operatorname{AfndAlg}_{\mathrm{cl}}^\mathrm{op}$. Then there is an adjunction
 \begin{equation}
    e: \operatorname{Afnd}_{\mathrm{cl}} \leftrightarrows \operatorname{Afnd} : t
 \end{equation}
 in which $e$ is fully-faithful, and $e$ and $t$ are both continuous for the respective rational topologies. Explicitly, $e$ is opposite to the inclusion $\operatorname{AfndAlg}_{\mathrm{cl}} \hookrightarrow \operatorname{AfndAlg}$ and $t$ is opposite to the functor $H^0: \operatorname{AfndAlg} \to \operatorname{AfndAlg}_{\mathrm{cl}}$. This adjunction has a surprising amount of structure.
 \begin{lem}\label{lem:classicaltruncationgeneralnonsense}
There is an adjoint triple $e_! \dashv e^* \dashv e_*$ and an adjoint pair $t_! \dashv t^*$ of functors $e_!, e_*, t^*: \operatorname{Shv}_{\tau_\mathrm{rat}}(\operatorname{Afnd}_\mathrm{cl}) \to \operatorname{Shv}_{\tau_\mathrm{rat}}(\operatorname{Afnd})$ and $t_!, e^*: \operatorname{Shv}_{\tau_\mathrm{rat}}(\operatorname{Afnd}) \to \operatorname{Shv}_{\tau_\mathrm{rat}}(\operatorname{Afnd}_{\mathrm{cl}})$ equipped with canonical equivalences $e_* \simeq t^*$ and $e^* \simeq t_!$. The functors $e_!$  and $e_* \simeq t^*$ are fully-faithful. The functor $t_!$ (resp. $e_!$) restricts to the functor $t$ (resp. $e$) on $\operatorname{Afnd}$ (resp. $\operatorname{Afnd}_{\mathrm{cl}}$).
 \end{lem}
\begin{proof}
By Lemma \ref{lem:adjointfunctorssites2} the functor $e$ is also cocontinuous. Using also Lemma \ref{lem:continuousfunctor} and Lemma \ref{lem:cocontinuousadjunction} we obtain the adjoint triple $e_! \dashv e^* \dashv e_*$ and the adjoint pair $t_! \dashv t^*$ together with the canonical equivalences $e_* \simeq t^*$ and $e^* \simeq t_!$. Then Lemma \ref{lem:fullyfaithful} implies that $e_!$ and $e_* \simeq t^*$ are fully-faithful. The rational topologies on $\operatorname{Afnd}_{\mathrm{cl}}$ and $\operatorname{Afnd}$ are both subcanonical, hence Corollary \ref{cor:lowershriekrepresentable} implies that $t_!$ (resp. $e_!$) restricts to $t$ (resp. $e$) on $\operatorname{Afnd}$ (resp. $\operatorname{Afnd}_{\mathrm{cl}}$). 
\end{proof}
\egapar Let $\operatorname{cRig}$ denote the category of classical rigid spaces. An easy consequence of Lemma \ref{lem:classicaltruncationgeneralnonsense} is that the adjunction $e_! \dashv t_!$ restricts to an adjunction $e_!: \operatorname{cRig} \leftrightarrows \operatorname{dRig}: t_!$, and we may view $\operatorname{cRig}$ as a full subcategory of $\operatorname{dRig}$. Further, we observe that the counit of the adjunction gives a canonical morphism $e_!t_!X \to X$ in $\operatorname{dRig}$. We may choose to suppress the functor $e_!$ in our notation.
\begin{defn}
Let $X \in \operatorname{dRig}$. We define its \emph{classical truncation} to be $t_!X \in \operatorname{cRig}$.
\end{defn}
\egapar We recall from \cite[Lemma 3.1.23]{SoorThesis} that for each $X \in \operatorname{dRig}$ the poset of analytic subspaces of $X$ forms a locale, denoted $\operatorname{An}(X)$. The locale $\operatorname{An}(X)$ is locally compact in the sense of \cite[VII.4]{JohnstoneStone} and therefore spatial. In this way one obtains\footnote{Here $\operatorname{pt}$ denotes the functor from locales to topological spaces, which sends the locale to the collection of completely prime filters on it.} a locally spectral topological space $|X| := \operatorname{pt}(\operatorname{An}(X))$ whose open subsets biject with analytic subspaces of $X$. This determines a functor $|\cdot|$ from $\operatorname{dRig}$ to the category of locally spectral topological spaces, with locally spectral morphisms. The following was proved in \cite[\S3.1.3]{SoorThesis}.  
\begin{thm}\label{thm:topologicalinvariance}
\begin{enumerate}[(i)]
    \item Let $X \in \operatorname{Afnd}$. Then the functor $t_!$ identifies the poset of derived rational subdomains of $X$ with the poset of rational subdomains of $t_!X$. 
    \item Let $X \in \operatorname{dRig}$. Then the functor $t_!$ identifies the locale of analytic subspaces of $X$ with the locale of analytic subspaces of $t_!X$. The morphism $e_!t_!X \to X$ induces a homeomorphism $|t_!X| \xrightarrow[]{\sim} |X|$.
\end{enumerate}
\end{thm}
\egapar Let us explicate a number of useful consequences of Theorem \ref{thm:topologicalinvariance}. Part (i) implies that for every $X \in \operatorname{Afnd}$ and every rational subspace $U \subseteq t_!X$, there is a unique rational subspace $\tilde{U} \subseteq X$ such that $t_! \tilde{U} \simeq U$. Part (ii) implies that for every $X \in \operatorname{dRig}$ and every analytic subspace $U \subseteq t_!X$, there is a unique analytic subspace $\tilde{U} \subseteq X$ such that $t_! \tilde{U} \simeq U$.

Let $f: X \to Y$ be a morphism in $\operatorname{dRig}$ and let $V \subseteq Y$ be an analytic subspace. Then Theorem \ref{thm:topologicalinvariance}(ii) implies that $f$ factors over $V$ if and only if $t_!(f): t_!X \to t_!Y$ factors over $t_!V$. In particular, we obtain the following functor-of-points description of $V$:
\begin{equation}\label{eq:liftfunctorofpoints}
    \operatorname{Map}(X,V) \xrightarrow[]{\sim}   \operatorname{Map}(t_!X, t_!V) \times_{\operatorname{Map}(t_!X, t_!Y)}\operatorname{Map}(X,Y).
\end{equation}
\egapar\label{par:HuberSpace} Let $X \in \operatorname{dRig}$. As in \cite[Remark 3.1.28]{SoorThesis}, using \cite[Corollary 4.5]{huber_continuous_1993} we see that the space $|t_!X|$ is canonically identified with the underlying topological space of the analytic adic space associated to $t_!X$. We may use this fact in combination with Theorem \ref{thm:topologicalinvariance}(ii) implicitly. In particular, this gives us a way to explicitly write down analytic subspaces of $X$. 
\section{The relative spectrum}\label{sec:relativespectrum}
\egapar We recall that we have fixed $\mathscr{V} := D(\operatorname{IndBan}_K)$. We denote by $\operatorname{Pr}^{L, \otimes}_{\mathrm{st}, \mathscr{V}}$ the $\infty$-category of $\mathscr{V}$-linear stable presentable symmetric monoidal $\infty$-categories, with left adjoint symmetric monoidal functors. By \cite[Proposition 3.22]{MathewGalois}, the functor 
\begin{equation}
    \operatorname{QCoh}: \operatorname{Aff}^\mathrm{op} \to \operatorname{Pr}^{L, \otimes}_{\mathrm{st}, \mathscr{V}}: \operatorname{Spec}(A) \mapsto \operatorname{Mod}_A\mathscr{V}
\end{equation}
is a sheaf in the descendable topology. By Proposition \ref{prop:sheavesontopos}, right Kan extension along $\operatorname{Aff}^\mathrm{op} \to \operatorname{Shv}_{\tau_{\mathrm{dri}}}(\operatorname{Aff})^\mathrm{op}$ extends this to a limit-preserving functor 
\begin{equation}\label{eq:QCohdefn1}
    \operatorname{QCoh}: \operatorname{Shv}_{\tau_{\mathrm{dri}}}(\operatorname{Aff})^\mathrm{op} \to \operatorname{Pr}^{L, \otimes}_{\mathrm{st}, \mathscr{V}}. 
\end{equation}
By Lemma \ref{lem:NCdescent}, the functor
\begin{equation}
    \operatorname{QAlg}: \operatorname{Aff}^\mathrm{op} \to \operatorname{Pr}^{L, \otimes}: \operatorname{Spec}(A) \mapsto \operatorname{DAlg}_A\mathscr{V}
\end{equation}
is a sheaf in the descendable topology. By Proposition \ref{prop:sheavesontopos}, right Kan extension along $\operatorname{Aff}^\mathrm{op} \to \operatorname{Shv}_{\tau_{\mathrm{dri}}}(\operatorname{Aff})^\mathrm{op}$ extends this to a limit-preserving functor
\begin{equation}
     \operatorname{QAlg}: \operatorname{Shv}_{\tau_{\mathrm{dri}}}(\operatorname{Aff})^\mathrm{op} \to \operatorname{Pr}^{L, \otimes}.
\end{equation}
\egapar\label{par:PsiEquivalencePar} For the sake of this paragraph only, let us denote the functor of \eqref{eq:QCohdefn1} by $\operatorname{QCoh}^\mathscr{V}$. There are two (a priori different) ways to define quasi-coherent sheaves on $\operatorname{Shv}_{\tau_{\mathrm{rat}}}(\operatorname{Afnd})$. Firstly, we may pre-compose the functor $\operatorname{QCoh}^\mathscr{V}$ with  $i_!: \operatorname{Shv}_{\tau_\mathrm{rat}}(\mathrm{Afnd}) \to \operatorname{Shv}_{\tau_\mathrm{dri}}(\mathrm{Aff})$. Secondly, we can consider the continuous functor $i: \operatorname{Afnd} \to \operatorname{Aff}$ and then right Kan extend $\operatorname{QCoh} \circ i$ along $\operatorname{Afnd}^\mathrm{op} \to \operatorname{Shv}_{\tau_\mathrm{rat}}(\mathrm{Afnd})^\mathrm{op}$ to obtain a limit-preserving functor 
\begin{equation}\label{eq:Qcohanalytic}
    \operatorname{QCoh}: \operatorname{Shv}_{\tau_\mathrm{rat}}(\operatorname{Afnd})^\mathrm{op} \to \operatorname{Pr}^{L, \otimes}_{\mathrm{st}, \mathscr{V}}.
\end{equation}
Thankfully, Lemma \ref{lem:rightKanextensionLemma} gives a canonical equivalence 
\begin{equation}\label{eq:Psiequivalence}
   \Psi : \operatorname{QCoh}^\mathscr{V} \circ i_! \xrightarrow[]{\sim} \operatorname{QCoh} 
\end{equation}
of functors $\operatorname{Shv}_{\tau_\mathrm{rat}}(\operatorname{Afnd})^\mathrm{op} \to \operatorname{Pr}^{L, \otimes}_{\mathrm{st}, \mathscr{V}}$, where on the right $\operatorname{QCoh}$ is defined as in \eqref{eq:Qcohanalytic}, so that the two definitions are in agreement. Because of this we will usually suppress the functor $\Psi$ in our notations, except when it is necessary to avoid confusion. An entirely similar discussion holds with $\operatorname{QAlg}$ in place of $\operatorname{QCoh}$.
\egapar For a morphism $f: X \to Y$ in $\operatorname{Shv}_{\tau_{\mathrm{dri}}}(\operatorname{Aff})$ we write $f^*$ for the pullback functor $f^* : \operatorname{QCoh}(Y) \to \operatorname{QCoh}(X)$ and $f_*$ for its right adjoint. We write $1_X \in \operatorname{QCoh}(X)$ for the unit object, the ``structure sheaf". 

For a morphism $f: X \to Y$ in $\operatorname{Shv}_{\tau_{\mathrm{dri}}}(\operatorname{Aff})$ we write $f^*$ for the pullback functor $f^* : \operatorname{QAlg}(Y) \to \operatorname{QAlg}(X)$ and $f_*$ for its right adjoint. By functoriality of right Kan extension we obtain natural transformations $\operatorname{LSym} : \operatorname{QCoh} \to \operatorname{QAlg}$ and $\operatorname{oblv}: \operatorname{QAlg} \to \operatorname{QCoh}$ of functors $\operatorname{Shv}_{\tau_{\mathrm{dri}}}(\operatorname{Aff}) \to \operatorname{Pr}^{L}$. Concretely, this means that for every such $f$ there are commutative squares
\begin{equation}
\begin{aligned}
\begin{tikzcd}
	{\operatorname{QCoh}(Y)} & {\operatorname{QAlg}(Y)} \\
	{\operatorname{QCoh}(X)} & {\operatorname{QAlg}(X)}
	\arrow["{\operatorname{LSym}_Y}", from=1-1, to=1-2]
	\arrow["{f^*}"', from=1-1, to=2-1]
	\arrow["{f^*}", from=1-2, to=2-2]
	\arrow["{\operatorname{LSym}_X}"', from=2-1, to=2-2]
\end{tikzcd} && \text{and} &&
\begin{tikzcd}
	{\operatorname{QAlg}(Y)} & {\operatorname{QCoh}(Y)} \\
	{\operatorname{QAlg}(X)} & {\operatorname{QCoh}(X)}
	\arrow["{\operatorname{oblv}_Y}", from=1-1, to=1-2]
	\arrow["{f^*}"', from=1-1, to=2-1]
	\arrow["{f^*}", from=1-2, to=2-2]
	\arrow["{\mathrm{oblv}_X}"', from=2-1, to=2-2]
\end{tikzcd}
\end{aligned}
\end{equation}
Furthermore, for every $X$ the functor $\operatorname{oblv}_X$ is canonically symmetric monoidal, for the coCartesian symmetric monoidal structure on $\operatorname{QAlg}(X)$ and the symmetric monoidal structure $\hat{\otimes}_X$ on $\operatorname{QCoh}(X)$. Therefore, we denote the coproduct on $\operatorname{QAlg}(X)$ by $\hat{\otimes}_X$. By passing to right adjoints we obtain also a commutative square 
\begin{equation}
\begin{tikzcd}
	{\operatorname{QAlg}(X)} & {\operatorname{QCoh}(X)} \\
	{\operatorname{QAlg}(Y)} & {\operatorname{QCoh}(Y)}
	\arrow["{\operatorname{oblv}_X}", from=1-1, to=1-2]
	\arrow["{f_*}"', from=1-1, to=2-1]
	\arrow["{f_*}", from=1-2, to=2-2]
	\arrow["{\operatorname{oblv}_Y}"', from=2-1, to=2-2]
\end{tikzcd}
\end{equation}
In an entirely similar manner, functoriality of right Kan extension gives a natural transformation $\Theta: \operatorname{QAlg} \to \operatorname{CAlg}(\operatorname{QCoh}(-))$ of functors $\operatorname{Shv}_{\tau_{\mathrm{dri}}}(\operatorname{Aff}) \to \operatorname{Pr}^{L, \otimes}$, so that we obtain a commutative square
\begin{equation}
\begin{tikzcd}
	{\operatorname{QAlg}(Y)} & {\operatorname{CAlg}(\operatorname{QCoh}(Y))} \\
	{\operatorname{QAlg}(X)} & {\operatorname{CAlg}(\operatorname{QCoh}(X))}
	\arrow["{\Theta_Y}", from=1-1, to=1-2]
	\arrow["{f^*}"', from=1-1, to=2-1]
	\arrow["{f^*}", from=1-2, to=2-2]
	\arrow["{\Theta_X}"', from=2-1, to=2-2]
\end{tikzcd}
\end{equation}
where the right vertical arrow is induced by the canonical symmetric-monoidal structure on $f^*: \operatorname{QCoh}(Y) \to \operatorname{QCoh}(X)$. For $A \in \operatorname{QAlg}(X)$ we define $\operatorname{Mod}_A \operatorname{QCoh}(X) := \operatorname{Mod}_{\Theta_X(A)} \operatorname{QCoh}(X)$. 
\begin{defn}\label{defn:relativespec}
\begin{enumerate}[(i)]
    \item Let $X \in \operatorname{Shv}_{\tau_{\mathrm{dri}}}(\operatorname{Aff})$ and let $A \in \operatorname{QAlg}(X)$. We define the presheaf $\underline{\operatorname{Spec}}_XA$ on $\mathrm{Aff}_{/X}$ by 
\begin{equation}\label{eq:SpecVfunctorofpoints}
  (\underline{\operatorname{Spec}}_X^\mathscr{V}A)(Y) := \operatorname{Map}_{\operatorname{QAlg}(X)}(A, f_*1_Y)
\end{equation}
for every $[f:Y \to X] \in \mathrm{Aff}_{/X}$. 
\item Let $X \in \operatorname{Shv}_{\tau_\mathrm{rat}}(\operatorname{Afnd})$ and let $A \in \operatorname{QAlg}(X)$. We define the presheaf $\underline{\operatorname{Spec}}^\mathrm{an}_XA$ on $\operatorname{Afnd}_{/X}$ by
\begin{equation}\label{eq:Specanfunctorofpoints}
    (\underline{\operatorname{Spec}}_X^\mathrm{an}A)(Y) := \operatorname{Map}_{\operatorname{QAlg}(X)}(A, f_*1_Y)
\end{equation}
for every $[f:Y \to X] \in \mathrm{Afnd}_{/X}$.
\end{enumerate}
\end{defn}
\egapar For the purposes of the next Lemmas we recall that, by Proposition \ref{prop:Slicetopoifunctoriality}, for every $X \in \operatorname{Shv}_{\tau_{\mathrm{dri}}}(\operatorname{Aff})$ there is a canonical equivalence $\operatorname{Shv}_{\tau_{\mathrm{dri}}}(\operatorname{Aff}_{/X}) \xrightarrow[]{\sim}  \operatorname{Shv}_{\tau_{\mathrm{dri}}}(\operatorname{Aff})_{/X}$.
\begin{lem}\label{lem:RelativespecLemma1}
\begin{enumerate}[(i)]
    \item For any $X \in \operatorname{Shv}_{\tau_{\mathrm{dri}}}(\operatorname{Aff})$ and any $A \in \operatorname{QAlg}(X)$, the presheaf $ \underline{\operatorname{Spec}}_X^\mathscr{V}A$ is a sheaf in the topology $\tau_{\mathrm{dri}}$ on $\mathrm{Aff}_{/X}$.  
    \item For any $X \in \operatorname{Shv}_{\tau_{\mathrm{rat}}}(\operatorname{Afnd})$ and any $A \in \operatorname{QAlg}(X)$, the presheaf $ \underline{\operatorname{Spec}}_X^\mathrm{an}A$ is a sheaf in the topology $\tau_{\mathrm{rat}}$ on $\mathrm{Afnd}_{/X}$.  
    \item With notations as in (i). If $f: X^\prime \to X$ is any morphism in $\operatorname{Shv}_{\tau_\mathrm{dri}}(\operatorname{Aff})$ then there is a canonical equivalence $\big(\underline{\operatorname{Spec}}_X^\mathscr{V}A\big)\times_X X^\prime  \simeq \underline{\operatorname{Spec}}_{X^\prime}^\mathscr{V}f^*A$. 
    \item With notations as in (ii). If $f: X^\prime \to X$ is any morphism in $\operatorname{Shv}_{\tau_\mathrm{rat}}(\operatorname{Afnd})$ then there is a canonical equivalence $\big(\underline{\operatorname{Spec}}_X^\mathrm{an}A\big)\times_X X^\prime  \simeq \underline{\operatorname{Spec}}_{X^\prime}^\mathrm{an}f^*A$. 
    \item For any $X \in \operatorname{Shv}_{\tau_{\mathrm{rat}}}(\operatorname{Afnd})$ and any $A \in \operatorname{QAlg}(X)$ there is a canonical equivalence $\underline{\operatorname{Spec}}_X^\mathrm{an}A \simeq i^*\underline{\operatorname{Spec}}_{i_!X}^\mathscr{V}A$ over $X$.
    \item For any $X \in \operatorname{Shv}_{\tau_{\mathrm{dri}}}(\operatorname{Aff})$ and any $A, B \in \operatorname{QAlg}(X)$, there is a canonical equivalence $\underline{\operatorname{Spec}}^\mathscr{V}_X(A \hat\otimes_X B) \simeq \underline{\operatorname{Spec}}^\mathscr{V}_X(A) \times_X \underline{\operatorname{Spec}}^\mathscr{V}_X(B) $.
    \item For any $X \in \operatorname{Shv}_{\tau_{\mathrm{rat}}}(\operatorname{Afnd})$ and any $A, B \in \operatorname{QAlg}(X)$, there is a canonical equivalence $\underline{\operatorname{Spec}}^\mathrm{an}_X(A \hat\otimes_X B) \simeq \underline{\operatorname{Spec}}^\mathrm{an}_X(A) \times_X \underline{\operatorname{Spec}}^\mathrm{an}_X(B) $.
\end{enumerate}
\end{lem}
\begin{proof}
(i): Let $Y \in \mathrm{Aff}_{/X}$ and let $\mathscr{C}^0_{/Y}$ be a $\tau_{\mathrm{dri}}$-covering sieve on $Y$. By Lemma \ref{lem:tautologicallimitlemma}, the canonical morphism $1_Y \to \lim_{Y^\prime \in \mathscr{C}^0_{/Y}} g_{Y^\prime,*}1_{Y^\prime}$ is an equivalence in $\operatorname{QAlg}(Y)$; here $g_{Y^\prime} : Y^\prime \to Y$ is the given morphism. Since the functor $\operatorname{Map}_{\operatorname{QAlg}(X)}(A, f_*(-))$ commutes with limits, we conclude. (ii): The proof is identical to (i). (iii) and (iv): This follows from Proposition \ref{prop:Slicetopoifunctoriality} and adjunction. (v): This is clear from \eqref{eq:SpecVfunctorofpoints} and \eqref{eq:Specanfunctorofpoints}, since $i^*$ is the restriction along $i: \operatorname{Afnd} \to \operatorname{Aff}$. (vi) and (vii): This is clear by looking at the functor of points in \eqref{eq:SpecVfunctorofpoints} and \eqref{eq:Specanfunctorofpoints}, recalling that $\hat{\otimes}_X$ is the coproduct on $\operatorname{QAlg}(X)$.
\end{proof}
\begin{lem}\label{lem:representable}
Let $f: X \to Y$ be a representable morphism in $\operatorname{Shv}_{\tau_{\mathrm{dri}}}(\operatorname{Aff})$. The the functor $f_*$ is compatible with base-change, satisfies the projection formula, preserves colimits, and is conservative.
\end{lem}
\begin{proof}
The question of whether $f_*$ is compatible with base-change, satisfies the projection formula, preserves colimits and is conservative is local on the target. Hence, we reduce to the case when $X = \operatorname{Spec}(B)$ and $Y = \operatorname{Spec}(A)$ are both affine, in which case all the assertions are clear. 
\end{proof}
\begin{lem}\label{lem:RelSpecRepresent}
Let $X \in \operatorname{Shv}_{\tau_{\mathrm{dri}}}(\operatorname{Aff})$ and $A \in \operatorname{QAlg}(X)$. Denote by $p: \underline{\operatorname{Spec}}^\mathscr{V}_X A \to X$ the structure morphism. Then $p$ is representable and there is a canonical equivalence $A \xrightarrow[]{\sim} p_*1 $ in $\operatorname{QAlg}(X)$ and a canonical equivalence $\operatorname{QCoh}(\underline{\operatorname{Spec}}^\mathscr{V}_X A) \simeq \operatorname{Mod}_A \operatorname{QCoh}(X)$ in $\operatorname{Pr}^{L, \otimes}_{\mathrm{st}, \mathscr{V}}$.
\end{lem}
\begin{proof}
To show that $p$ is representable, we may assume that $X = \operatorname{Spec}(B)$ is affine and $A \in \operatorname{DAlg}_{B}\mathscr{V}$. Then $\underline{\operatorname{Spec}}^\mathscr{V}_X A = \operatorname{Spec}(A)$ and $p$ is the canonical morphism $\operatorname{Spec}(A) \to \operatorname{Spec}(B)$. For the second part, for brevity we put $Y = \underline{\operatorname{Spec}}^\mathscr{V}_X A$.  The canonical morphism $A \to p_*1_{Y} $ corresponds under \eqref{eq:SpecVfunctorofpoints} to the identity morphism. By Lemma \ref{lem:representable}, $p_*$ is compatible with base-change, hence to prove that $A \to p_*1_{Y} $ is an equivalence we reduce to the case when $X = \operatorname{Spec}(B)$ is affine and $A \in \operatorname{DAlg}_B\mathscr{V}$. In this case, $\underline{\operatorname{Spec}}_X^\mathscr{V}A = \operatorname{Spec}(A)$ and $A \to p_*1_{Y} $ is just the identity $A \to A$. For the last part, by Barr--Beck and the projection formula we see that $ \operatorname{QCoh}(Y) \simeq \operatorname{Mod}_{p_*1_Y} \operatorname{QCoh}(X) = \operatorname{Mod}_A\operatorname{QCoh}(X)$.
\end{proof}
\begin{defn}
Let $X \in \operatorname{Shv}_{\tau_\mathrm{rat}}(\operatorname{Afnd})$. We say that $A \in \operatorname{QAlg}(X)$ is \emph{locally affinoid} if there exists a covering $\{f_i: U_i \to X\}$ by affinoids $U_i = \operatorname{Spec}(B_i)$ such that each $f_i^*A$ is an affinoid algebra. (To be precise, this means that the image of $f_i^*A$ under the forgetful functor $\operatorname{QAlg}(U_i) \to \operatorname{QAlg}(*)$ belongs to the full subcategory $\operatorname{AfndAlg} \subseteq \operatorname{QAlg}(*)$). 
\end{defn}
\begin{lem}\label{lem:affinoidrelative spectrum}
Let $X \in \operatorname{Shv}_{\tau_\mathrm{rat}}(\operatorname{Afnd})$ and suppose that $A \in \operatorname{QAlg}(X)$ is locally affinoid. Then there is a canonical equivalence $i_!\underline{\operatorname{Spec}}_X^\mathrm{an} A \xrightarrow[]{\sim} \underline{\operatorname{Spec}}^\mathscr{V}_{i_!X} A$ over $i_!X$. 
\end{lem}
\begin{proof}
The canonical morphism $
i_!\underline{\operatorname{Spec}}_X^\mathrm{an} A \rightarrow \underline{\operatorname{Spec}}^\mathscr{V}_{i_!X} A$ is deduced from Lemma \ref{lem:RelativespecLemma1}(v) and adjunction. In order to check that it is an equivalence we can work locally on $i_!X$, and so we may assume (using that $i_!$ preserves covers and representable objects) that $X = \operatorname{Spec}(B)$ is affinoid and $A \in \operatorname{QAlg}(X)$ is an affinoid algebra. In this case the canonical morphism is the equality $i_!\operatorname{Spec}(A) = \operatorname{Spec}(iA)$. 
\end{proof}
\begin{defn}\label{defn:vectorbundles}
\begin{enumerate}[(i)]
    \item Let $X \in \operatorname{Shv}_{\tau_\mathrm{dri}}(\operatorname{Aff})$ and let $E \in \operatorname{QCoh}(X)$. We define $V_X^\mathscr{V}(E) := \underline{\operatorname{Spec}}^\mathscr{V}_X \operatorname{LSym}_X E$, so that the functor of points of $V_X^\mathscr{V}(E)$ can be described as 
\begin{equation}
    V_X^\mathscr{V}(E)(Y) \simeq \operatorname{Map}_{\operatorname{QCoh}(X)}(E, f_*1_Y)
\end{equation}
for every $[f: Y \to X] \in \operatorname{Aff}_{/X}$.
    \item Let $X \in \operatorname{Shv}_{\tau_\mathrm{rat}}(\operatorname{Afnd})$ and let $E \in \operatorname{QCoh}(X)$. We define $V_X^\mathrm{an}(E) := \underline{\operatorname{Spec}}^\mathrm{an}_X \operatorname{LSym}_X E$. so that the functor of points of $V_X^\mathrm{an}(E)$ can be described as \begin{equation}
    V_X^\mathrm{an}(E)(Y) \simeq \operatorname{Map}_{\operatorname{QCoh}(X)}(E, f_*1_Y)
\end{equation}
for every $[f: Y \to X] \in \operatorname{Afnd}_{/X}$. 
\end{enumerate}
\end{defn}
\begin{lem}
Let $X \in \operatorname{dRig}$ and suppose that $E \in \operatorname{QCoh}(X)$ is a locally free. Then $V_X^\mathrm{an}(E)$ is representable by a derived rigid space. 
\end{lem}
\begin{proof}
The question is local on $X$, and using the base-change compatibility of Lemma \ref{lem:RelativespecLemma1}(iv), we may assume that $X = \operatorname{Spec}(A)$ is affinoid and $E \simeq  A^{\oplus n}$. By the base-change compatibility again, we further reduce to the case when $X = *$ and $E \simeq K^{\oplus n}$. In this case $V_X^\mathrm{an}(E)$ is the functor which sends $\operatorname{Spec}(B) \in \operatorname{Afnd}$ to (the underlying space of) $B^{\oplus n}$, which is representable by $\mathbf{A}^{n, \mathrm{an}}$ by Proposition \ref{prop:affineline}.
\end{proof}
\section{Cotangent complexes}\label{sec:cotangent}
\egapar In this section we study infinitesimal properties of derived rigid spaces. In particular we define cotangent complexes, smooth and \'etale morphisms, and quasi-smooth Zariski-closed immersions. For the latter we prove the expected structure theorem (Lemma \ref{lem:Zariskiclosedquasismooth}). 
\egapar Lemma \ref{lem:coherentdescent} below is necessary to prove that the cotangent complex of a morphism of derived rigid spaces is connective.
\begin{defn}\label{defn:Cohcd}
Let $X = \operatorname{Spec}(A) \in \operatorname{Afnd}$ and let $d \in \mathbf{Z}$. We define the full subcategory $\operatorname{Coh}^{\leqslant d}(X) \subseteq \operatorname{QCoh}(X)$ on those complexes whose cohomology modules are finitely generated as $H^0(A)$-modules, supported in the interval $(-\infty,d]$. We define $\operatorname{Coh}^{-}(X) \subseteq \operatorname{QCoh}(X)$ as the union of the full subcategories $\operatorname{Coh}^{\leqslant d}(X)$ for $d \in \mathbf{Z}$. 
\end{defn}
\begin{lem}\label{lem:coherentdescent}
The collection of full subcategories $\operatorname{Coh}^{\leqslant d}(X) \subseteq \operatorname{QCoh}(X)$ extends naturally to a sub-prestack $\operatorname{Coh}^{\leqslant d}: \operatorname{Afnd}^\mathrm{op} \to \operatorname{Cat}_\infty$ of $\operatorname{QCoh}: \operatorname{Afnd}^\mathrm{op} \to \operatorname{Cat}_\infty$ which satisfies descent in the topology $\tau_\mathrm{rat}$. The same holds for $\operatorname{Coh}^-$ in place of $\operatorname{QCoh}$.
\end{lem}
\begin{proof}
We prove the Lemma for $\operatorname{Coh}^{\leqslant d}$, as the case of $\operatorname{Coh}^-$ is essentially the same. For the first part, for any morphism $X =\operatorname{Spec}(B) \to \operatorname{Spec}(A) = Y$ in $\operatorname{Afnd}$ and any $M^\bullet \in \operatorname{Coh}^{\leqslant d}(Y)$ there is a convergent spectral sequence 
\begin{equation}\label{eq:Torspectral1}
    E_2^{pq}: \operatorname{Tor}_p^{H^*(A)}(H^*(B), H^{*}(M^\bullet))_q \Rightarrow H^{p-q}(B \hat{\otimes}^\mathbf{L}_AM^\bullet)
\end{equation}
Using that $H^0(A)$ (resp. $H^0(B)$) is Noetherian, that each $H^i(A)$ (resp. $H^i(B)$) is finitely-generated as a $H^0(A)$- (resp. $H^0(B)$-)module, and that each $H^j(M^\bullet)$ is finitely generated as a $H^0(A)$-module, it can be seen that each $E_2^{pq}$ is finitely-generated as a $H^0(B)$-module. By Noetherianity of $H^0(B)$ and convergence of \eqref{eq:Torspectral1} we see that $B\hat{\otimes}^\mathbf{L}_AM^\bullet \in \operatorname{Coh}^{\leqslant d}(X)$. For the second part, let $\{X_i = \operatorname{Spec}(B_i) \to \operatorname{Spec}(B) = X\}_i$ be a finite collection of rational open subspaces which is jointly surjective on underlying topological spaces. Put $Z := \coprod_i X_i$ and $C := \prod_i B_i$, so that $Z = \operatorname{Spec}(C)$. We will show that 
\begin{equation}\label{eq:cohdescent1}
    \operatorname{Coh}^{\leqslant d}(X) \xrightarrow[]{\sim} \underset{{[n] \in \Delta}}{\operatorname{lim}}\operatorname{Coh}^{\leqslant d}(Z^{n+1/X}). 
\end{equation}
By descent for quasi-coherent sheaves, the only thing to show is that if $(M_n^\bullet)_{[n] \in \Delta}$ is any Cartesian section belonging to the right-hand side of \eqref{eq:cohdescent1} the limit $\operatorname{lim}_{[n] \in \Delta} f_{n, *} M_n^\bullet$ belongs to the full subcategory $\operatorname{Coh}^{\leqslant d}(X) \subseteq \operatorname{QCoh}(X)$. Here $f_{n,*}: Z^{n+1/X} \to X$ is the canonical morphism. There is a convergent spectral sequence 
\begin{equation}
   E_2^{pq} : R^p\underset{{[n] \in \Delta}}{\operatorname{lim}}  H^p(f_{n,*} M_n^\bullet)  \Rightarrow H^{p+q}\big(\underset{{[n] \in \Delta}}{\operatorname{lim}} f_{n, *} M_n^\bullet\big).
\end{equation}
By \cite[Proposition 2.3.88]{DAnG}, for every cosimplicial morphism $Z^{m+1/X} \to Z^{n+1/X}$ and every $j \in \mathbf{Z}$ there is a canonical equivalence $H^0(C_{m}) \hat{\otimes}_{H^0(C_{n})}H^j(M_n^\bullet) \xrightarrow[]{\sim} H^j(C_{m} \hat{\otimes}^\mathbf{L}_{C_{n}} M_n^\bullet)$, here $C_{n} := C^{\hat{\otimes}^\mathbf{L}_B(n+1)}$. Hence by Kiehl's theorem and the classical theorem of descent for coherent sheaves one has $E_2^{pq} = 0$ for $p > 0$ and $E_2^{0q} \simeq \operatorname{eq}(H^q(M_0^\bullet) \rightrightarrows H^q(M_1^\bullet))$ is finitely-generated as a $H^0(B)$-module. Further, we see that $E_2^{0q} = 0$ for $q > d$. 
\end{proof}
\begin{rmk}
It can be proven that $\operatorname{Coh}^{\leqslant d}$ satisfies hyperdescent, even in the \'etale topology, c.f. \cite{kelly_localising_2025}. If $[c,d] \subseteq \mathbf{R}$ is any interval with $d < \infty$ it can also be shown that $\operatorname{Coh}^{[c,d]}$ (i.e., those complexes with cohomology supported in $[c,d]$), is a sheaf on the poset of rational subdomains of a fixed affinoid $X$. 
\end{rmk}
\egapar By right Kan extension along $\operatorname{Afnd}^\mathrm{op} \to \operatorname{Shv}_{\tau_\mathrm{rat}}(\operatorname{Afnd})^\mathrm{op}$ we obtain for any $d \in \mathbf{Z}$ a limit-preserving functor $\operatorname{Coh}^{\leqslant d} : \operatorname{Shv}_{\tau_\mathrm{rat}}(\operatorname{Afnd})^\mathrm{op} \to \operatorname{Cat}_\infty$, which is naturally a subfunctor of $\operatorname{QCoh}: \operatorname{Shv}_{\tau_\mathrm{rat}}(\operatorname{Afnd})^\mathrm{op} \to \operatorname{Cat}_\infty$. This will be used in the next Definition.
\begin{defn}\label{defn:CotangentComplex}
\begin{enumerate}[(i)]
    \item Let $Z \in \operatorname{Shv}_{\tau_\mathrm{dri}}(\operatorname{Aff})$ and let $M \in \operatorname{QCoh}(Z)$. We define $Z[M]_\mathscr{V} := \underline{\operatorname{Spec}}_Z^\mathscr{V}(1_Z \oplus M)$ where $1_Z \oplus M \in \operatorname{QAlg}(Z)$ denotes the trivial square-zero extension. We consider the Cartesian fibration $\operatorname{QCoh}_{/Y} \to \operatorname{Shv}_{\tau_\mathrm{dri}}(\operatorname{Aff})_{/Y}$ whose objects are pairs $(Z, M)$ where $Z \in \operatorname{Shv}_{\tau_\mathrm{dri}}(\operatorname{Aff})_{/Y}$ and $M \in \operatorname{QCoh}(Z)$. We say that a morphism $f: X \to Y$ in $\operatorname{Shv}_{\tau_\mathrm{dri}}(\operatorname{Aff})$ has a \emph{global cotangent complex} if the functor 
    \begin{equation}
        \operatorname{Der}_{X/Y} : (\operatorname{QCoh}_{/Y})^\mathrm{op} \to \infty\mathrm{Grpd}: (Z, M) \mapsto \operatorname{Map}_{/Y}(Z[M]_\mathscr{V},X) 
    \end{equation}
    is representable by an object of the form $(X, \mathbf{L}_{X/Y}^\mathscr{V})$. In this case $\mathbf{L}_{X/Y}^\mathscr{V} \in \operatorname{QCoh}(X)$ is called the \emph{cotangent complex} of $X/Y$. 
    \item Let $Z \in \operatorname{Shv}_{\tau_\mathrm{rat}}(\operatorname{Afnd})$ and let $M \in \operatorname{Coh}^{\leqslant0}(Z)$. We define $Z[M] := \underline{\operatorname{Spec}}_Z^\mathrm{an}(1_Z \oplus M)$ where $1_Z \oplus M \in \operatorname{QAlg}(Z)$ denotes the trivial square-zero extension. We consider the Cartesian fibration $\operatorname{Coh}^{\leqslant 0}_{/Y} \to \operatorname{Shv}_{\tau_\mathrm{rat}}(\operatorname{Afnd})_{/Y}$ whose objects are pairs $(Z, M)$ where $Z \in \operatorname{Shv}_{\tau_\mathrm{rat}}(\operatorname{Afnd})_{/Y}$ and $M \in \operatorname{Coh}^{\leqslant 0}(Z)$. We say that a morphism $f: X \to Y$ in $\operatorname{Shv}_{\tau_\mathrm{rat}}(\operatorname{Afnd})$ has a \emph{global cotangent complex} if the functor 
    \begin{equation}
        \operatorname{Der}_{X/Y} : (\operatorname{Coh}^{\leqslant 0}_{/Y})^\mathrm{op} \to \infty\mathrm{Grpd}: (Z, M) \mapsto \operatorname{Map}_{/Y}(Z[M],X) 
    \end{equation}
    is representable by an object of the form $(X, \mathbf{L}_{X/Y})$. In this case $\mathbf{L}_{X/Y} \in \operatorname{Coh}^{\leqslant 0}(X)$ is called the \emph{cotangent complex} of $X/Y$. 
\end{enumerate}
\end{defn}
\egapar We note the following standard properties of cotangent complexes. 
\begin{prop}\label{prop:cotangentprops}
\begin{enumerate}[(i)]
    \item If $f:X \to Y$ is any morphism in $\operatorname{Shv}_{\tau_\mathrm{dri}}(\operatorname{Aff})$ (resp. $\operatorname{Shv}_{\tau_\mathrm{rat}}(\operatorname{Afnd})$) such that $\mathbf{L}^\mathscr{V}_{X/Y}$ (resp. $\mathbf{L}_{X/Y}$) exists, and 
    \begin{equation}    
\begin{tikzcd}
	{X^\prime} & {Y^\prime} \\
	X & Y
	\arrow["{f^\prime}", from=1-1, to=1-2]
	\arrow["{g^\prime}"', from=1-1, to=2-1]
	\arrow["\lrcorner"{anchor=center, pos=0.125}, draw=none, from=1-1, to=2-2]
	\arrow["g", from=1-2, to=2-2]
	\arrow["f", from=2-1, to=2-2]
\end{tikzcd}
\end{equation}
    is a Cartesian square in $\operatorname{Shv}_{\tau_\mathrm{dri}}(\operatorname{Aff})$ (resp. $\operatorname{Shv}_{\tau_\mathrm{rat}}(\operatorname{Afnd})$), then $\mathbf{L}_{X^\prime/Y^\prime}^\mathscr{V}$ (resp. $\mathbf{L}_{X^\prime/Y^\prime}$) exists and is computed as $\mathbf{L}_{X^\prime/Y^\prime}^\mathscr{V} \simeq g^{\prime,*}\mathbf{L}^\mathscr{V}_{X/Y}$ (resp. $\mathbf{L}_{X^\prime/Y^\prime} \simeq g^{\prime,*}\mathbf{L}_{X/Y}$). 
    \item If $f:X \to Y$ is any morphism in $\operatorname{Shv}_{\tau_\mathrm{dri}}(\operatorname{Aff})$ (resp. $\operatorname{Shv}_{\tau_\mathrm{rat}}(\operatorname{Afnd})$) and $\{Y_i \to Y\}_i$ is a cover such that each base-change $f_i: X_i \to Y_i$ admits a cotangent complex, then $f: X \to Y$ admits a cotangent complex. 
    \item Let $X \xrightarrow[]{f} Y \xrightarrow{g} Z$ be morphisms in $\operatorname{Shv}_{\tau_\mathrm{dri}}(\operatorname{Aff})$ (resp. $\operatorname{Shv}_{\tau_\mathrm{rat}}(\operatorname{Afnd})$) . Assume that $g$ admits a cotangent complex. Then $gf$ admits a cotangent complex if and only if $f$ does. In this case there is a fiber sequence $g^*\mathbf{L}^\mathscr{V}_{Y/Z} \to \mathbf{L}_{X/Z}^\mathscr{V} \to \mathbf{L}^\mathscr{V}_{X/Y}$ in $\operatorname{QCoh}(X)$ (resp. there is a fiber sequence $g^*\mathbf{L}_{Y/Z} \to \mathbf{L}_{X/Z} \to \mathbf{L}_{X/Y}$ in $\operatorname{Coh}^{\leqslant 0}(X)$).  
\end{enumerate}
\end{prop}
\egapar The proof of the following Corollary is standard. 
\begin{cor}\label{cor:formallyetalecotangent}
Let $f: X \to Y$ be a morphism in $\operatorname{Shv}_{\tau_\mathrm{dri}}(\operatorname{Aff})$ (resp. $\operatorname{Shv}_{\tau_\mathrm{rat}}(\operatorname{Afnd})$) such that the cotangent complex $\mathbf{L}_{X/Y}^\mathscr{V}$ exists (resp. $\mathbf{L}_{X/Y}$ exists). Then $\mathbf{L}_{X/Y}^\mathscr{V} \simeq 0$ (resp. $\mathbf{L}_{X/Y} \simeq 0$) if and only if $\mathbf{L}_{X/X \times_Y X}^\mathscr{V} \simeq 0$ (resp. $\mathbf{L}_{X/X \times_Y X} \simeq 0$). In this situation, if $g: Y \to Z$ is any further morphism then there is an equivalence $f^* \mathbf{L}^\mathscr{V}_{Y/Z} \simeq \mathbf{L}^\mathscr{V}_{X/Z}$ (resp. $f^* \mathbf{L}_{Y/Z} \simeq \mathbf{L}_{X/Z}$).
\end{cor}
\egapar Corollary \ref{cor:formallyetalecotangent} will be most frequently applied in the case when $f$ is a monomorphism, e.g., when $f$ is the inclusion of an analytic subspace into a (derived) rigid space.
\egapar For notational ease, we have usually suppressed the equivalence $\Psi$ introduced in Paragraph \ref{par:PsiEquivalencePar}. For the sake of the next Lemma we temporarily suspend this notational suppression. In particular we recall the definition of $\operatorname{QCoh}^\mathscr{V}$ from Paragraph \ref{par:PsiEquivalencePar}. 
\begin{lem}\label{lem:analytificationofcotangent}
Let $f: X \to Y$ be a morphism in $\operatorname{Shv}_{\tau_\mathrm{dri}}(\operatorname{Aff})$. Let $\sigma : i_!i^*X \to X$ denote the morphism obtained from the counit of $i_! \dashv i^*$ and let $\Psi$ denote the equivalence from \eqref{eq:Psiequivalence}. If $\mathbf{L}^\mathscr{V}_{X/Y}$ exists and the object
\begin{equation}\label{eq:cotangenti*formula}
    \Psi_{i^*X} \sigma^*\mathbf{L}^\mathscr{V}_{X/Y}
\end{equation}
belongs to $\operatorname{Coh}^{\leqslant 0}(i^*X)$, then $\mathbf{L}_{i^*X/i^*Y}$ exists and is computed by the formula \eqref{eq:cotangenti*formula}. 
\end{lem}
\begin{proof}
By definition, a map from $(Z, N) \in \operatorname{QCoh}_{/i^*Y}$ to $(i^*X, \Psi_{i^*X} \sigma^*\mathbf{L}^\mathscr{V}_{X/Y})$ is equivalent to the datum of a map $u: Z \to i^*X$ over $i^*Y$ together with a morphism $\varsigma: u^*\Psi_{i^*X} \sigma^* \mathbf{L}^\mathscr{V}_{X/Y} \to N$ in $\operatorname{QCoh}(Z)$. Adjunction induces an equivalence $\operatorname{Map}_{/i^*Y}(Z, i^*X) \simeq \operatorname{Map}_{/Y}(i_!Z,X)$ which sends $u$ to $u^\prime := \sigma \circ i_!(u)$. By naturality of $\Psi$ there is a diagram 
\begin{equation}
\begin{tikzcd}
	{\operatorname{QCoh}^\mathscr{V}(X)} & {\operatorname{QCoh}^\mathscr{V}(i_!i^*X)} & {\operatorname{QCoh}^\mathscr{V}(i_!Z)} \\
	& {\operatorname{QCoh}(i^*X)} & {\operatorname{QCoh}(Z)}
	\arrow["{\sigma^*}", from=1-1, to=1-2]
	\arrow["{i_!(u)^*}", from=1-2, to=1-3]
	\arrow["{\Psi_{i^*X}}"', from=1-2, to=2-2]
	\arrow["{\Psi_Z}", from=1-3, to=2-3]
	\arrow["{u^*}", from=2-2, to=2-3]
\end{tikzcd}
\end{equation}
hence $u^*\Psi_{i^*X} \sigma^* \mathbf{L}^\mathscr{V}_{X/Y} \simeq \Psi_Zu^{\prime,*}\mathbf{L}_{X/Y}^\mathscr{V}$. Let $\Psi_Z^{-1}$ denote a quasi-inverse to $\Psi_Z$ and put $N^\prime := \Psi_Z^{-1}N$, $\varsigma^\prime := \Psi_Z^{-1}\varsigma$. In this way we obtain the first in the following chain of equivalences, informally given by $u \mapsto u^\prime, \varsigma \mapsto \varsigma^\prime$:
\begin{equation}
\begin{aligned}
\operatorname{Map}_{\operatorname{QCoh}_{/i^*Y}}((Z,N), (i^*X, \Psi_{i^*X} \sigma^*\mathbf{L}^\mathscr{V}_{X/Y})) &\simeq \operatorname{Map}_{\operatorname{QCoh}^\mathscr{V}_{/Y}}((i_!Z,N^\prime), (X,  \mathbf{L}_{X/Y}^\mathscr{V})) \\
    &\simeq \operatorname{Map}_{/Y}( (i_!Z)[N^\prime]_\mathscr{V}, X) \\
    &\simeq \operatorname{Map}_{/Y}(i_!Z[N], X) \\
    &\simeq \operatorname{Map}_{/i^*Y}(Z[N], i^*X), 
\end{aligned}
\end{equation}
here the second equivalence is by definition and the third is from Lemma \ref{lem:affinoidrelative spectrum}. Hence $\mathbf{L}_{i^*X/i^*Y}$ exists and is computed by the formula \eqref{eq:cotangenti*formula}. 
\end{proof}
\begin{thm}\label{thm:CoherentCotangent}
Let $f: X \to Y$ be any morphism in $\operatorname{dRig}$. Then the cotangent complex $\mathbf{L}_{X/Y} \in \operatorname{Coh}^{\leqslant 0}(X)$ exists. 
\end{thm}
\begin{proof}
By \cite[Proposition 8.12]{AnalyticHKR} there exists a complex $\mathbf{L}_{X/Y} \in \operatorname{Coh}^{-}(X)$ such that $(X, \mathbf{L}_{X/Y})$ represents the functor $\operatorname{Der}_{X/Y}$. What we need to show is that $\mathbf{L}_{X/Y} \in \operatorname{Coh}^{\leqslant 0}(X)$. By descent on $Y$ using Proposition \ref{prop:cotangentprops}(i) and Lemma \ref{lem:coherentdescent} we reduce to the case when $Y = \operatorname{Spec}(A)$ is affinoid. By Corollary \ref{cor:formallyetalecotangent}, for every affinoid subspace $g: U=\operatorname{Spec}(B) \hookrightarrow X$ there is an equivalence $g^*\mathbf{L}_{X/Y} \simeq \mathbf{L}_{B/A}$. Hence by descent on $X$ using Lemma \ref{lem:coherentdescent} it suffices to show that, for a morphism $A \to B$ of affinoid algebras, $\mathbf{L}_{B/A}$ is connective. This follows from \cite[Corollary 2.1.25]{DAnG}. 
\end{proof}
\begin{lem}
\begin{enumerate}[(i)]
    \item Let $X \in \operatorname{Shv}_{\tau_\mathrm{dri}}(\operatorname{Aff})$ and let $E \in \operatorname{QCoh}(X)$. Then the cotangent complex of $\varrho: V_X^\mathscr{V}(E) \to X$ exists and equals $\varrho^*E$.
    \item Let $X \in \operatorname{Shv}_{\tau_\mathrm{rat}}(\operatorname{Afnd})$ and let $E \in \operatorname{Coh}^{\leqslant 0}(X)$. Then the cotangent complex of $\varrho: V_X^\mathrm{an}(E) \to X$ exists and equals $\varrho^*E$. 
\end{enumerate}
\end{lem}
\begin{proof}
Both statements follow easily by comparing the definition of the cotangent complex with the functor of points of $V_X^\mathscr{V}(E)$ (resp. $V_X^\mathrm{an}(E)$) as in Definition \ref{defn:vectorbundles}.
\end{proof}
\begin{defn}\label{defn:Derivedsmooth}
Let $f: X \to Y$ be a morphism of derived rigid spaces and let $d \geqslant 0$.
\begin{enumerate}[(i)]
    \item We say that $f$ is \emph{smooth} (resp. \emph{smooth of dimension $d$}) if $\mathbf{L}_{X/Y}$ is, locally in the analytic topology on $X$, free of finite rank (resp. free of rank $d$), and $t_!(f)$ is a smooth morphism of classical rigid spaces in the sense of \cite[Definition 1.6.5]{HuberEtale}.
    \item We say that $f$ is \emph{\'etale} (resp. \emph{finite \'etale}) if $\mathbf{L}_{X/Y} \simeq 0$ and $t_!(f)$ is an \'etale (resp. finite \'etale) morphism of classical rigid spaces in the sense of \cite[Definition 1.6.5]{HuberEtale}. 
\end{enumerate}
\end{defn}
\begin{lem}\label{lem:finiteetale}
Let $f: A \to B$ be an \'etale morphism between derived affinoid algebras. Then  $f$ is finite \'etale in the sense of Definition \ref{defn:Derivedsmooth} if and only if $B$ is a retract of a finite free $A$-module. 
\end{lem}
\begin{proof}
 The \emph{if} part of the statement is clear since $H^0$ is an additive functor. For the \emph{only if} part, we follow the same argument as in \cite[Lemma 2.2.2.2]{toen_homotopical_2008}. Since $H^0(B)$ is finite projective as a $H^0(A)$-module there is some $n \geqslant 0$ and a sequence $H^0(B) \xrightarrow[]{i} H^0(A)^{\oplus n} \xrightarrow[]{r} H^0(B)$ of $H^0(A)$-modules such that $ri = \operatorname{id}$. The morphisms $r$ and $p:= ir$ can be lifted to morphisms $\tilde{r}: A^{\oplus n} \to B$ and $\tilde{p}: A^{\oplus n} \to A^{\oplus n}$ which are unique up to homotopy, and further one has $\tilde{p} \tilde{p} \simeq \tilde{p}$ and $\tilde{r}\tilde{p} \simeq \tilde{r}$. The projector $\tilde{p}$ then gives rise to a split fiber sequence $M \to A^{\oplus n} \to N$ and $\tilde{r}$ induces a morphism $N \to B$. This is a morphism between derived strong $A$-modules which is an isomorphism on $H^0$, hence, it is an isomorphism. 
\end{proof}
\begin{lem}\label{lem:classicalsmooth}
Let $f: \operatorname{Spec}(B) \to \operatorname{Spec}(A)$ be a smooth morphism (in the sense of \cite[Definition 1.6.5]{HuberEtale}) between classical affinoids. Then:
\begin{enumerate}[(i)]
    \item For any finitely-generated $A$-module $M$ the canonical morphism $B \hat{\otimes}_A^\mathbf{L}M \to B \hat{\otimes}_AM$ is an equivalence;
    \item For any morphism $A \to C$ of classical affinoid algebras the canonical morphism $B \hat{\otimes}_A^\mathbf{L}C \to B \hat{\otimes}_AC$ is an equivalence; 
    \item If $f$ is \'etale, then $\mathbf{L}_{B/A} \simeq 0$. In general, $\mathbf{L}_{B/A}$ is, locally in the analytic topology on $\operatorname{Spec}(B)$, free of finite rank. In particular, a morphism of classical rigid spaces is smooth in the sense of \cite[Definition 1.6.5]{HuberEtale} and only if it is smooth in the sense of Definition \ref{defn:Derivedsmooth}. 
\end{enumerate}
\end{lem}
\begin{proof}
(i): Let $M$ be a finitely-generated $A$-module; we need to show that the canonical morphism $B \hat{\otimes}^\mathbf{L}_AM \to B \hat{\otimes}_A M$ is an equivalence. Now, if $\{B \to B_i\}_i$ is a family of rational localizations corresponding to a rational cover of $X = \operatorname{Spec}(B)$, then the $ B_i \hat{\otimes}^\mathbf{L}_B (-)$ form a conservative family. Moreover, rational localizations are transversal to finitely-generated modules. Hence the property (i) is local in the analytic topology on $\operatorname{Spec}(B)$.

By the structure theorem for smooth morphisms \cite[Corollary 1.6.10]{HuberEtale} we reduce to the case when $f$ is \'etale or $f$ is the projection off a polydisk. The latter case is clear since the Tate algebra $K \langle T_1, \dots, T_n \rangle$ is projective as a $K$-Banach space. For \'etale morphisms, by \cite[Proposition 3.1.4]{vanderPutetale} the claim reduces to the case when $f$ is finite \'etale. This is clear, since then $B$ is a finite projective $A$-module.  (ii): The proof is the same as in (i), except we use that rational localizations of affinoid algebras are transversal to arbitrary morphisms of affinoid algebras, c.f. \cite[Lemma 5.14]{BBKNonArch}. (iii): Let us first treat the \'etale case. By definition, the morphism $B \hat{\otimes}_A B \to B$ is an affinoid localization, in particular a homotopy epimorphism. Hence by (ii), the morphism $B \hat{\otimes}^\mathbf{L}_A B \to B$ is a homotopy epimorphism, so $A \to B$ is formally unramified in the sense of \cite[Definition 2.1.39]{DAnG}. But being formally unramified is equivalent to being formally \'etale \cite[Proposition 2.1.41]{DAnG}, that is to say, $\mathbf{L}_{B/A} \simeq 0$. In general, if $f$ is smooth then $f$ factors locally in the analytic topology on $\operatorname{Spec}(B)$ as an \'etale morphism followed by the projection off a polydisk, and one uses the fiber sequence for cotangent complexes, that is, Proposition \ref{prop:cotangentprops}(iii). 
\end{proof}
\begin{lem}\label{lem:smoothderivedstrong}
Let $f: \operatorname{Spec}(B) \to \operatorname{Spec}(A)$ be a smooth morphism of derived affinoids. Then $f$ is derived strong, i.e., for every $n \leqslant 0$ the canonical morphism 
\begin{equation}
H^0(B) \hat{\otimes}^\mathbf{L}_{H^0(A)} H^n(A)  \to H^n(B) 
\end{equation}
is an equivalence. 
\end{lem}
\begin{proof}
By Lemma \ref{lem:classicalsmooth}(i) this is equivalent to asking for 
\begin{equation}\label{eq:pinBiso}
H^0(B) \hat{\otimes}_{H^0(A)} H^n(A)  \to H^n(B)  
\end{equation}
to be an isomorphism of $H^0(B)$-modules. If $\{B \to B_i\}_i$ is a family of rational localizations corresponding to a rational cover of $X = \operatorname{Spec}(B)$, then the $\{H^0(B) \to H^0(B_i)\}_i$ also form a rational cover, in particular  we can check the isomorphism \eqref{eq:pinBiso} after applying each $H^0(B_i) \hat{\otimes}_{H^0(B)} (-)$. Using Lemma \ref{lem:classicalsmooth}(iii) we may choose the $B_i$ such that each $\mathbf{L}_{B_i/A}$ and $\mathbf{L}_{H^0(B_i)/H^0(A)}$ are finite free. Then the claim follows from \cite[Proposition 2.6.160]{DAnG}. 
\end{proof}
\begin{defn}
\begin{enumerate}[(i)]
    \item Let $A$ be a derived affinoid algebra. Given morphisms\footnote{In $D(\operatorname{IndBan}_K)$.} $f_1, \dots, f_n : K \to A \hat{\otimes}_K^\mathbf{L}K \langle T_1, \dots, T_n\rangle$, we define $\Delta(f_1,\dots,f_n) :=\operatorname{det}\left(\frac{\partial f_i^0}{\partial T_j}\right)_{i,j}$, where $f_i^0 := H^0(f_i)(1) \in H^0(A) \hat{\otimes}_K K \langle T_1, \dots, T_n \rangle$.
    \item A \emph{standard \'etale morphism} is a morphism $f: A \to B$ between derived affinoid algebras such that $B$ admits a presentation
\begin{equation}
    B \simeq (A \hat{\otimes}^\mathbf{L}_K K \langle T_1, \dots, T_n \rangle)/^\mathbf{L}(f_1, \dots, f_n)
\end{equation}
for some $f_i: K \to  A \hat{\otimes}^\mathbf{L}_K K \langle T_1, \dots, T_n \rangle $ such that $\Delta(f_1,\dots,f_n)$ maps to a unit in $H^0(B)$. 
\end{enumerate}
\end{defn}
\begin{prop}\label{prop:derivedetalelifting}
\begin{enumerate}[(i)]
    \item Every standard \'etale morphism is \'etale;
    \item Every \'etale morphism between derived affinoid algebras is standard \'etale;
    \item If $A$ is a classical affinoid algebra and $A \to B$ is any \'etale morphism of derived affinoid algebras, then $B$ is necessarily classical;
    \item If $A$ is a derived affinoid and $f: H^0(A) \to B$ is any \'etale morphism then there is a unique (up to contractible choice) \'etale morphism $\tilde{f}: A \to \tilde{B}$ such that $H^0(\tilde{f}) \simeq f$. 
    \end{enumerate}
\end{prop}
\begin{proof}
(i): By \cite[\S 3.1]{vanderPutetale} the morphism $H^0(A) \to H^0(B)$ is an \'etale morphism between classical affinoids and by \cite[Proposition 4.5.83]{DAnG} one has $\mathbf{L}_{B/A} \simeq 0$, hence $f$ is \'etale according to Definition \ref{defn:Derivedsmooth}. (ii): By \cite[Proposition 1.7.1]{HuberEtale} (see also \cite[Observation 3.1.2]{vanderPutetale}) the morphism $H^0(A) \to H^0(B)$ is standard \'etale in the sense that there exists $n \geqslant 0$ and $f_1, \dots, f_n \in H^0(A) \hat{\otimes}_K K \langle T_1, \dots, T_n\rangle$ such that 
\begin{equation}
    H^0(B) \simeq (H^0(A) \hat{\otimes}_K K \langle T_1,\dots, T_n \rangle )/(f_1, \dots, f_n), 
\end{equation}
and $\Delta(f_1,\dots, f_n)$ maps to a unit in $H^0(B)$. As $K$ is projective, we may lift the $f_i$ to obtain morphisms $f_i: K \to A \hat{\otimes}^\mathbf{L}_K K \langle T_1, \dots, T_n \rangle$ and we define 
\begin{equation}
    \tilde{B} := (A \hat{\otimes}_K^\mathbf{L} K \langle T_1,\dots, T_n \rangle) /^\mathbf{L}(f_1, \dots, f_n). 
\end{equation}
The morphisms $A \to B$ and $A \to \tilde{B}$ are formally \'etale (for the latter we use part (i)). Hence the \'etale lifting property \cite[Corollary 2.1.36]{DAnG} implies that there is a (unique up to contractible choice) equivalence $\tilde{B} \simeq B$ under $A$. (iii): This follows immediately from the derived strong property, c.f. Lemma \ref{lem:smoothderivedstrong}. (iv): By part (iii) $B$ is classical, and again by \cite[Proposition 1.7.1]{HuberEtale} $f: H^0(A) \to B$ is standard \'etale in the classical sense. Arguing as in the proof of (iii) we obtain an \'etale morphism $\tilde{f}:A \to \tilde{B}$ such that $H^0(\tilde{f}) \simeq f$; the uniqueness of this lift again follows from \cite[Corollary 2.1.36]{DAnG}.
\end{proof}
\begin{lem}\label{lem:etaleliftinggeneral}
Let $X$ be a derived rigid space with classical truncation $t_!X$. Let $f: Y \to t_!X$ be an \'etale morphism. Then there is a unique (up to contractible choice) \'etale morphism $\tilde{f}: \tilde{Y} \to X$ of derived rigid spaces such that $t_!(\tilde{f}) \simeq f$. 
\end{lem}
\begin{proof}
The difficulty is to establish existence of the lift $\tilde{Y}$ and to show that it is represented by a derived rigid space. We consider the adjunction $t_! : \operatorname{Shv}_{\tau_{\mathrm{rat}}}(\operatorname{Afnd}) \leftrightarrows \operatorname{Shv}_{\tau_{\mathrm{rat}}}(\operatorname{Afnd}_\mathrm{cl}) : t^*$ from Lemma \ref{lem:classicaltruncationgeneralnonsense}. By \cite[Proposition 5.2.5.1]{HigherToposTheory} this induces an adjunction \begin{equation}
\operatorname{Shv}_{\tau_{\mathrm{rat}}}(\operatorname{Afnd})_{/X} \leftrightarrows \operatorname{Shv}_{\tau_{\mathrm{rat}}}(\operatorname{Afnd}_\mathrm{cl})_{/t_!X}
\end{equation}
in which the left adjoint is given by applying $t_!$, and the right adjoint is given on objects $Z \in \operatorname{Shv}_{\tau_{\mathrm{rat}}}(\operatorname{Afnd}_\mathrm{cl})_{/t_!X}$ by the formula $Z \mapsto \tilde{Z} :=t^*Z \times_{t^*t_!X} X$. We will show that, with $Y$ as in the statement of the Lemma, the sheaf $\tilde{Y}$ is representable. We first claim that, if $U \hookrightarrow Y$ is an analytic subspace in the sense of Definition \ref{defn:derivedrigidspace}(iii), then $\tilde{U} \hookrightarrow \tilde{Y}$ is also an analytic subspace. Let $S = \operatorname{Spec}(A) \to \tilde{Y}$ be given; then the associativity of fiber products implies that $\tilde{U} \times_{\tilde{Y}}S \to S$ is equivalent to the morphism $t^*U \times_{t^*Y}S \to S$. For any derived rigid space $T$ one has by adjunction
\begin{equation*}
\begin{aligned}
    \operatorname{Map}(T, t^*U \times_{t^*Y}S) &\simeq \operatorname{Map}(t_!T, U) \times_{\operatorname{Map}(t_!T, Y)} \operatorname{Map}(T, S) \\
    &\simeq \operatorname{Map}(t_!T, U) \times_{\operatorname{Map}(t_!T, Y)} \operatorname{Map}(t_!T, t_!S)\times_{\operatorname{Map}(t_!T, t_!S)} \operatorname{Map}(T, S)\\
    &\simeq \operatorname{Map}(t_!T, U \times_Y t_!S) \times_{\operatorname{Map}(t_!T, t_!S)} \operatorname{Map}(T,S),
\end{aligned}
\end{equation*}
where in the second line above we used that the map $\operatorname{Map}(T, S) \to \operatorname{Map}(t_!T, Y)$ factors through $\operatorname{Map}(T, S) \to \operatorname{Map}(t_!S, t_!S)$. This implies that $t^*U \times_{t^*Y}S \subseteq S$ is the unique analytic subspace lifting $U \times_Y t_!S \subseteq t_!S$, as it has the same universal property, c.f. \eqref{eq:liftfunctorofpoints}. Hence $\tilde{U} \hookrightarrow \tilde{Y}$ is an analytic subspace. 

Next, we claim that, if $\{U_i \to Y\}_{i \in I}$ is a cover of $Y \in \operatorname{dRig}$ by analytic subspaces, then $\{\tilde{U}_i \to \tilde{Y}\}_{i \in I}$ is also a cover. Let $S = \operatorname{Spec}(A) \to \tilde{Y}$ be given; then we obtain a morphism $t_!S \to Y$. There exists a cover $\{ S_j \to t_!S\}_{j \in J}$ and a function $\alpha: J \to I$ such that each $S_j \to Y$ factors as $S_j \to U_{\alpha(j)} \to Y$. By Theorem \ref{thm:topologicalinvariance} we may lift the the $S_j$ to analytic subspaces $\hat{S}_j$ of $S$; the universal property implies that there is a factorization $\hat{S}_j \to \tilde{U}_{\alpha(j)} \to \tilde{Y}$, and $\{\hat{S}_j \to S\}_{j \in J}$ is a cover. Hence $\{ \tilde{U}_i \to \tilde{U}\}_{i \in I}$ is a cover. 

Now let $X^\prime \to X$ be any morphism. Then the associativity of fiber products (together with the fact that $t^*$ preserves limits) implies that 
\begin{equation}
\tilde{Y} \times_X X^\prime \simeq \tilde{Y^\prime},
\end{equation} 
where $Y^\prime := Y \times_{t_!X} t_!X^\prime$. Hence by considering a cover of $X$ by affinoid subspaces, to prove that $\tilde{Y}$ is representable it suffices to treat the case when $X$ is affinoid, which we will assume in the next paragraph.

If $Y$ is itself affinoid, then by Proposition \ref{prop:derivedetalelifting}(iv) there is an \'etale morphism $\hat{Y} \to X$ lifting $f$ and by \cite[Corollary 2.1.36]{DAnG} we see that $\hat{Y}$ has the required universal property, so $\tilde{Y} \simeq \hat{Y}$ is representable. In general, $Y$ has a cover by affinoid analytic subspaces $\{U_i \hookrightarrow Y\}_i$; then by the previous discussion $\{ \tilde{U}_i \hookrightarrow \tilde{Y}\}_i$ is a cover by analytic subspaces, which we now know to be representable. Therefore, $\tilde{Y}$ is representable by a derived rigid space. It is clear that the morphism $\tilde{Y} \to X$ is \'etale. 

For uniqueness, assume that $\breve{f}: \breve{Y} \to X$ is another \'etale morphism such that $t_!(\breve{f}) \simeq f$. By definition of $\tilde{Y}$, one has $\operatorname{Map}_{/X}(\breve{Y},\tilde{Y}) \xrightarrow{\sim} \operatorname{Map}_{/t_!X}(Y,Y)$, so there is a unique (up to contractible choice) morphism $h:\breve{Y} \to \tilde{Y}$ over $X$ lifting $\operatorname{id}: Y \to Y$. We claim that $h$ is an isomorphism; this is local on $X$ so we may assume that $X$ is affinoid. This is also local on $\tilde{Y}$, hence by lifting an affinoid cover of $Y$ to $\tilde{Y}$, we may further assume that $Y = \operatorname{Spec}(A)$ and $\tilde{Y} \simeq \operatorname{Spec}(\tilde{A})$. If $V = \operatorname{Spec}(B) \subseteq \breve{Y}$ an affinoid subspace then $t_! V  \subseteq Y$ is also an affinoid subspace, so by \'etale lifting \cite[Corollary 2.1.36]{DAnG} we conclude that $h|_V : V \to \tilde{Y}$ is equivalent to the inclusion of the affinoid subspace $\tilde{t_!V} \subseteq \tilde{Y}$. Since $V$ was arbitrary we conclude that $h$ is an isomorphism.
\end{proof}
\begin{defn}
A morphism $f: X \to Y$ in $\operatorname{Shv}_{\tau_\mathrm{rat}}(\operatorname{Afnd})$ is called a \emph{Zariski-closed immersion} if $f$ is representable, and for every map $Y^\prime =\operatorname{Spec}(A) \to Y$ from an object of $\operatorname{Afnd}$, the pullback $X^\prime = \operatorname{Spec}(B) \to Y^\prime$ is such that $H^0(A) \to H^0(B)$ is surjective.
\end{defn}
\egapar This paragraph is not so relevant to this article and may be skipped. Let $A \to B$ be a morphism in $\operatorname{DAlg}^{\leqslant 0}(\mathscr{V})$. Then the condition that ``$A \to B$ is surjective on $H^0$" can be equivalently rephrased as ``$A\to B$ has connective fiber". The latter leads to the correct definition of a Zariski-closed immersion in the category $\operatorname{Aff} = \operatorname{DAlg}(\mathscr{V})^\mathrm{op}$, c.f. \cite{BenBassatHekkingBlowUp}.
\begin{defn}\label{defn:Zariskiclosedquasismooth}
A Zariski-closed immersion $f:X \to Y$ in $\operatorname{dRig}$ is called \emph{quasi-smooth} (resp. \emph{quasi-smooth of virtual dimension $-n$}) if $\mathbf{L}_{X/Y}[-1] \in \operatorname{QCoh}(X)$ is locally free of finite rank (resp. locally free of rank $n$). 
\end{defn}
\begin{defn}\label{defn:regularsurjection}
A morphism $A \to B$ in $\operatorname{DAlg}^{\leqslant 0}(\mathscr{V})$ is called a \emph{regular surjection} (of virtual dimension $-n$) if it is of the form $A \to A/^\mathbf{L}(f_1, \dots,f_n)$ for some morphisms $f_1, \dots, f_n: K \to A$.
\end{defn}
\egapar With notations as in Definition \ref{defn:regularsurjection}. The integer $n$ is intrinsic to the morphism $A \to B$ and does not depend on the choice of presentation: it can be recovered as the rank of the free $B$-module $\mathbf{L}_{B/A}[1]$. 
\begin{lem}\label{lem:Zariskiclosedquasismooth}
Let $f: X \to Y$ be a Zariski-closed immersion of derived rigid spaces. Then $f$ is quasi-smooth (resp. quasi-smooth of virtual dimension $-n$) if and only if locally in the analytic topology on $Y$, $f$ is of the form $\operatorname{Spec}(B) \to \operatorname{Spec}(A)$ where $A \to B$ is a regular surjection (resp. regular surjection of virtual dimension $-n$) in the sense of Definition \ref{defn:regularsurjection}.
\end{lem}
\begin{proof}
This is essentially the same as \cite[Proposition 2.3.8]{KhanVirtualCartier}. One direction is clear: if $B = A/^\mathbf{L}(f_1,\dots,f_n)$ then one may compute $\mathbf{L}_{B/A} \simeq B^{\oplus n}[1]$. For the converse, we reduce to the case when $Y = \operatorname{Spec}(A)$ and $X = \operatorname{Spec}(B)$ are affinoid and $\mathbf{L}_{B/A}[-1]$ is free of rank $n$, for some $n \geqslant 0$. Let $F := \operatorname{Fib}(A \to B)$ taken in $\operatorname{Mod}_A \mathscr{V}$ and put $I := \ker(H^0(A) \to H^0(B))$; note that $I$ is finitely-generated and that $F$ is connective. By properties of the Hurewicz homomorphism \cite[Theorem 2.1.21]{DAnG} there is an isomorphism $I/I^2=H^0(F \hat{\otimes}^\mathbf{L}_A B) \cong H^{-1}(\mathbf{L}_{B/A})$. Hence $I/I^2$ is free of rank $n$ as a $H^0(B)$-module. By Nakayama (see Lemma \ref{lem:Nakayama} below), after possibly localizing to a rational open subset of $Y$ containing $Z$, we may choose generators $f_1^0, \dots, f_n^0$ for $I$ whose images $\overline{f_1^0}, \dots \overline{f_n^0}$ in $I/I^2$ form a basis. The elements $f_1^0, \dots, f_n^0$ lift to morphisms $f_1, \dots, f_n: K \to A$ unique up to homotopy and moreover the composites $f_i : A \to B$ are homotopic to $0$, hence we obtain a factorization $A \to A/^\mathbf{L}(f_1,\dots,f_n) \to B$. We claim that the second morphism is an equivalence. We already know that it is an isomorphism on $H^0$, so by \cite[Corollary 2.1.27]{DAnG} it suffices to show that its cotangent complex is zero. In the fiber sequence
\begin{equation}
 \mathbf{L}_{(A/^\mathbf{L}(f_1,\dots,f_n))/A} \hat{\otimes}^\mathbf{L}_{A/^\mathbf{L}(f_1,\dots,f_n)} B \to   \mathbf{L}_{B/A} \to \mathbf{L}_{B/(A/^\mathbf{L}(f_1,\dots,f_n))}
\end{equation}
the first morphism is the equivalence $B^{\oplus n}[1] \xrightarrow[]{\sim} \mathbf{L}_{B/A}$ induced by the choice of basis $\overline{f_1^0}, \dots, \overline{f_n^0}$ for $I/I^2$. Hence $A/^\mathbf{L}(f_1,\dots,f_n) \xrightarrow[]{\sim} B$.
\end{proof}
\begin{lem}\label{lem:Nakayama}
Let $A$ be a classical affinoid algebra, let $I\subseteq A$ be an ideal and let $M$ be a finitely-generated $A$-module. Put $X := \operatorname{Spec}(A) \supseteq \operatorname{Spec}(A/I) =: Z$. Let $m_1, \dots m_n \in M$ whose images $\overline{m}_1, \dots, \overline{m}_n \in M/IM$ generate $M/IM$ as an $A/I$-module. Then there exists a rational open subset $Z \subseteq U \subseteq X$ such that $m_1 \otimes 1, \dots, m_n \otimes 1$ generate $M \hat{\otimes}_A A_U$ as an $A_U$-module. 
\end{lem}
\begin{proof}
Let $M^\prime \subseteq M$ be the submodule generated by $m_1, \dots, m_n$. By assumption, $M = M^\prime + IM$ so $M/M^\prime = I(M/M^\prime)$. By Nakayama's lemma there exists $f \in I$ such that $1-f$ acts by zero on $M/M^\prime$, in other words, $(1-f)M \subseteq M^\prime$. Therefore, we set $U := X(\frac{1}{1-f})$. 
\end{proof}
\begin{lem}\label{lem:ZariskiclosedRegular2}
Let $Z \xrightarrow{g} Y \xrightarrow{f} X$ be morphisms in $\operatorname{dRig}$ such that $f$ and $fg$ are both smooth and $g$ is a Zariski-closed immersion. Then $g$ is a quasi-smooth Zariski-closed immersion. 
\end{lem}
\begin{proof}
There is a fiber sequence $g^*\mathbf{L}_{Y/X} \to \mathbf{L}_{Z/X} \to \mathbf{L}_{Z/Y}$. We need to show that $\mathbf{L}_{Z/Y}[1]$ is locally free. The question is local on $Y$ and $Z$. Hence we may assume that $Z = \operatorname{Spec}(C)$, $Y= \operatorname{Spec}(B)$ and $X = \operatorname{Spec}(A)$ are all affinoid and $\mathbf{L}_{Z/X}$ and $\mathbf{L}_{Y/X}$ are both free. We have a fiber sequence 
\begin{equation}\label{eq:cotangent1}
    C \hat{\otimes}^\mathbf{L}_B \mathbf{L}_{B/A} \to \mathbf{L}_{C/A} \to \mathbf{L}_{C/B},
\end{equation}
which, after applying $H^0$, gives by \cite[Corollary 2.6.140]{DAnG} an exact sequence 
\begin{equation}\label{eq:Omega1}
   H^0(C) \hat{\otimes}_{H^0(B)} \Omega^1_{H^0(B)/H^0(A)} \to \Omega^1_{H^0(C)/H^0(A)} \to \Omega^1_{H^0(C)/H^0(B)} \to 0.
\end{equation}
But $\Omega^1_{H^0(C)/H^0(B)} = 0$, as $H^0(B) \to H^0(C)$ is surjective. Hence the first morphism in \eqref{eq:Omega1} is surjective and therefore admits a section since the middle term is free. By standard arguments, c.f. \cite[Lemma 2.2.2.2]{toen_homotopical_2008} or as in the proof of Lemma \ref{lem:finiteetale}, this lifts (up to homotopy) to a section of the first morphism in \eqref{eq:cotangent1}. Hence $\mathbf{L}_{C/B}[-1]$ is a retract of a finite free $C$-module. In particular, $\mathbf{L}_{C/B}[-1]$  is derived strong and $H^0(\mathbf{L}_{C/B}[-1])$ is a finite projective and hence locally finite free $H^0(C)$-module. Further, $\mathbf{L}_{C/B}[-1]$ belongs to $\operatorname{Coh}^{\leqslant 0}(\operatorname{Spec}(A))$ by Theorem \ref{thm:CoherentCotangent}. Therefore the claim follows from Lemma \ref{lem:projectivederivedmodule} below.
\end{proof}
\begin{lem}\label{lem:projectivederivedmodule}
Let $X = \operatorname{Spec}(A) \in \operatorname{Afnd}$. Assume that $P \in \operatorname{Coh}^{\leqslant 0}(X)$ is derived strong and $H^0(P)$ is locally finite free as a $H^0(A)$-module. Then $P$ is locally finite free. 
\end{lem}
\begin{proof}
Let $\{A \to A_i\}_i$ be a finite family of (derived) rational localizations corresponding to a cover of $\operatorname{Spec}(A)$, such that each $H^0(A_i) \hat{\otimes}^\mathbf{L}_{H^0(A)} H^0(P)$ is finite free. The assumption that $P \in \operatorname{Coh}^{\leqslant 0}(X)$ guarantees that the transversality hypotheses in \cite[Corollary 2.3.89]{DAnG} are satisfied and therefore $A_i \hat{\otimes}^\mathbf{L}_A P$ is derived strong as an $A_i$-module. By assumption, there exists $n \geqslant 0$ and an equivalence $H^0(A_i)^{\oplus n} \xrightarrow[]{\sim} H^0(A_i \hat{\otimes}^\mathbf{L}_A P)$; this lifts (up to homotopy) to a morphism $A_i^{\oplus n} \to A_i \hat{\otimes}^\mathbf{L}_A P$. This is a morphism between derived strong $A_i$-modules which is an isomorphism on $H^0$, therefore, it is an equivalence. Hence $P$ is locally finite free. 
\end{proof}
\egapar The following Lemma can be deduced from Lemma \ref{lem:ZariskiclosedRegular2}. However, in this particular case the proof is much simpler and so we decided to include it. The notion of a separated morphism is defined in Definition \ref{defn:separated} below.
\begin{lem}\label{lem:SmoothSection}
Let $\pi: Y \to X$ be a smooth separated morphism of derived rigid spaces (of dimension $n$). Assume that $\pi$ admits a section $f: X \to Y$. Then $f$ is a quasi-smooth Zariski-closed immersion (of virtual dimension $-n$).
\end{lem}
\begin{proof}
As $\pi$ has representable diagonal, the section $f: X \to Y$ is automatically representable and $\pi f \simeq \operatorname{id}$ implies that $f$ is a Zariski-closed immersion. From $X \to Y \to X$ we obtain a fiber sequence $f^*\mathbf{L}_{Y/X} \to 0 \to \mathbf{L}_{X/Y}$, hence $\mathbf{L}_{X/Y}[-1] \simeq f^*\mathbf{L}_{Y/X} $ and we deduce the claim.
\end{proof}
\section{Partially-proper morphisms}
\label{sec:partiallyproper}
\egapar The notion of ``boundary" is an important feature of analytic geometry which does not appear in algebraic geometry. For the formulation and proof of the Grothendieck duality theorem, it is perhaps most straightforward to restrict to morphisms ``without boundary". These are the \emph{partially-proper} morphisms. All Stein spaces and all proper rigid spaces are partially proper, but, for instance, affinoids are generally not partially-proper.
\begin{defn}\label{defn:separated}
Let $f: X \to Y$ be a morphism of derived rigid spaces with diagonal $\Delta_f: X \to X \times_Y X$. We say that $f$ is \emph{quasi-separated} (resp. \emph{semi-separated}, resp. \emph{separated}), if $\Delta_f$ is a quasi-compact morphism (resp. a representable morphism, resp. a Zariski-closed immersion).
\end{defn}
\begin{defn}\cite[Definition 1.3.3]{HuberEtale}
A morphism $f: X \to Y$ of derived rigid spaces is called \emph{partially proper} if it is separated and for every morphism $Y^\prime \to Y$ in $\mathrm{dRig}$ and every $x \in |X \times_Y Y^\prime|$, the image $|f^\prime|(\overline{\{x\}}) \subseteq |Y^\prime|$ is closed. Here $f^\prime : X \times_Y Y^\prime \to Y^\prime$ denotes the base-change of $f$. A morphism of derived rigid spaces is called \emph{proper} if it is partially-proper and quasi-compact. A derived rigid space $X$ is called partially-proper (resp. proper), if the morphism $X \to *$ is partially-proper (resp. proper). 
\end{defn}
\begin{defn}
Let $X \in \mathrm{dRig}$. An analytic subspace is called \emph{partially proper} if the inclusion $U \hookrightarrow X$ is a partially-proper morphism.
We define $\operatorname{An}^\circ(X)$ to be the subframe of $\operatorname{An}(X)$ on the partially-proper subspaces. In this way we obtain a Grothendieck topology called the \emph{partially-proper analytic topology} on $X$ and denoted by $X_{\mathrm{pp}}$.
\end{defn}
\begin{lem}\label{lem:partiallypropersourcetarget}
The property of being partially-proper is local for the analytic topology on the target and local on the source for the partially-proper analytic topology. More precisely, if $f: X \to Y$ is a separated morphism of derived rigid spaces, then:
\begin{enumerate}[(i)]
    \item If $\{Y_i \to Y\}_i$ a cover of $Y$ by analytic subspaces such that each base-change $f_i: X_i \to Y_i$ is partially-proper, then  $f$ is partially-proper;
    \item If $\{X_i \to X\}_i$ is a cover of $X$ by partially-proper analytic subspaces such that each restriction $f|_{X_i}:X_i \to Y$ is partially-proper, then $f$ is partially-proper. 
\end{enumerate}
\end{lem}
\begin{proof}
(i): Take $Y^\prime \to Y$, let $f^\prime: X^\prime \to Y^\prime$ (resp. $f^\prime_i: X^\prime_i \to Y^\prime_i$) be the base change of $f$ (resp. $f_i$) and let $x \in |X^\prime|$. By assumption $|f_i^\prime|(\overline{\{x\}} \cap |X_i^\prime|) = |f^\prime|( \overline{\{x\}}) \cap |Y_i^\prime|$ is closed for each $i$. Since the $|Y_i^\prime|$ are open subsets which cover $|Y^\prime|$, then $|f^\prime|( \overline{\{x\}})$ is closed in $|X^\prime|$. 

(ii): Take $Y^\prime \to Y$ and let $f^\prime :X^\prime \to Y^\prime$ (resp. $f^\prime|_{X_i}:X_i^\prime \to Y^\prime$) denote the base-changes. Let $x \in |X^\prime|$, then $x \in |X_i^\prime|$ for some $i$; let $j: X_i^\prime \hookrightarrow X^\prime$ be the inclusion. By assumption, $|j|(\overline{\{x\}} \cap |X_i^\prime|) \subseteq |X^\prime|$ is closed and contains $x$, so we must have $\overline{\{x\}} \subseteq |X_i^\prime|$. Then $|f^\prime|(\overline{\{x\}}) = |f^\prime|_{X_i^\prime}|(\overline{\{x\}})$ is closed in $|Y^\prime|$ by assumption. 
\end{proof}
\begin{lem}
A separated morphism $f: X \to Y$ of derived rigid spaces is partially-proper (resp. proper) if and only if its classical truncation $t_!(f) : t_!X \to t_!Y$ is partially-proper in the sense of Huber \cite[Definition 0.4.8]{HuberEtale}. 
\end{lem}
\begin{proof}
This follows from Theorem \ref{thm:topologicalinvariance} and the fact that $t_!$ is limit-preserving.
\end{proof}
\begin{lem}\label{lem:Zariskiclosedclosed}
Let $f: Z \to X$ be a Zariski-closed immersion of derived rigid spaces. Then $|f|:|Z| \to |X|$ is a closed immersion of topological spaces.
\end{lem}
\begin{proof}
One reduces to the case when $X$ is affinoid, and then using the invariance of $|X|$ under classical truncation, to the case when $Z = \operatorname{Spec}(B) \to \operatorname{Spec}(A) = X$ is a morphism of classical affinoids such that $A \to B$ is surjective. We make use of the connection to valuation theory, c.f. Paragraph \ref{par:HuberSpace}: it seems to be difficult to prove the Lemma without this. Put $I = \operatorname{ker}(A \to B)$, then the morphism $|Z| \to |X|$ identifies with the morphism of Huber spectra $\operatorname{Spa}(B, B^\circ) \to \operatorname{Spa}(A,A^\circ)$ which is a homeomorphism onto $ \{ x \in \operatorname{Spa}(A,A^\circ): |f|_x = 0 \ \forall f \in I\} $, which is closed\footnote{This follows directly from the definition of the topology on the valuation spectrum, \cite{huber_continuous_1993}.} in $\operatorname{Spa}(A,A^\circ)$. 
\end{proof}
\begin{cor}\label{cor:partiallyproperproperties}
\begin{enumerate}[(i)]
    \item Any Zariski-closed immersion of derived rigid spaces is proper;
    \item Let $f, g$ be composable morphisms of derived rigid spaces such that $f$ is separated and $fg$ is partially-proper (resp. proper). Then $g$ is partially-proper (resp. proper). 
\end{enumerate}
\end{cor}
\begin{proof}
(i): Since the class of Zariski-closed immersions is closed under base-change, this follows from Lemma \ref{lem:Zariskiclosedclosed}. 
(ii): Since the class of partially-proper morphisms is closed under base-change and composition, this is equivalent (by the graph trick) to the assertion that the diagonal of a partially-proper morphism is partially proper, which is (i). 
\end{proof}
\egapar The reason why we have to prove the following Proposition is because we don't know \emph{a priori} that our analytification agrees with Huber's \cite{HuberEtale}.
\begin{prop}
Let $f: X \to Y$ be a separated morphism between derived schemes $X, Y$ over $K$ which are locally homotopically of finite presentation. Then $f^\mathrm{an}: X^\mathrm{an} \to Y^\mathrm{an}$ is a partially-proper morphism of derived rigid spaces.
\end{prop}
\begin{proof}
Suppose first that $X = \operatorname{Spec}(A)$ is affine and $A$ is finite cellular in the sense of Definition \ref{defn:cellular}. Using that partially-proper morphisms are stable under base-change and composition, it suffices to show that the analytifications of both $\operatorname{Spec}(K[D^{n+1}]) \to \operatorname{Spec}(K[S^{n}]) $ and $\operatorname{Spec}(K[T]) \to *$ are both partially-proper. By Proposition \ref{prop:affineline}, the analytification of the latter is $\mathbf{A}^{1,\mathrm{an}} \to *$, which is partially-proper by \cite[Lemma 5.7.3]{HuberEtale}. For the former, we note that $K[S^{n}] \simeq K \otimes_{K[S^{n-1}]}^\mathbf{L}K$ and $K[S^1] = K \otimes^\mathbf{L}_{K[T]}K$. Hence by induction we see that the analytification of $\operatorname{Spec}(K[D^{n+1}]) \to \operatorname{Spec}(K[S^{n}])$ induces the identity on classical truncations when $n \geqslant 1$, and so is partially proper. When $n = 0$, a $1$-cell attachment is the base-change of $\operatorname{Spec}(K) \to \operatorname{Spec}(K[T])$, which analytifies to $* \to \mathbf{A}^{1, \mathrm{an}}$, which is partially-proper as it is a Zariski-closed immersion. Hence $X^\mathrm{an}$ is partially-proper. If $X = \operatorname{Spec}(A)$ with $A$ homotopically of finite presentation, then by Proposition \ref{prop:AniAlgcompact}, $A$ is a retract of a finite cellular algebra. Using the previous, it follows that $X^\mathrm{an}$ admits a Zariski-closed immersion to a partially-proper derived rigid space, hence itself is partially proper. 

Now it follows from Corollary \ref{cor:partiallyproperproperties}(ii) that if $f: X= \operatorname{Spec}(A) \to \operatorname{Spec}(B) =Y$ is a morphism between hfp affine derived schemes, then $f^\mathrm{an}: X^\mathrm{an} \to Y^\mathrm{an}$ is partially-proper. Now we treat the general case, i.e., let $f: X \to Y$ be a separated morphism between locally hfp derived schemes; we will show that $f^\mathrm{an}$ is partially-proper. 
Using Lemma \ref{lem:partiallypropersourcetarget}(i) we reduce to the case when $Y$ is affine. Now let $\{X_i \to X\}_i$ be an affine cover of $X$; by Lemma \ref{lem:partiallypropersourcetarget}(ii) and the affine case treated above, it suffices to show that each $X_i^\mathrm{an} \to X^\mathrm{an}$ is partially-proper. By Lemma \ref{lem:partiallypropersourcetarget}(i) it suffices to show that each $X_i^\mathrm{an} \times_{X^\mathrm{an}} X_j^\mathrm{an} \to X_i^\mathrm{an}$ is partially-proper. By separatedness, each $X_i \times_X X_j \to X_i$ is a morphism between affines, so this follows from the affine case.
\end{proof}
\egapar Our convention for specializations and generalizations follows Huber \cite[Definition 1.1.9]{HuberEtale}. We feel it is important to emphasize this point because it appears that different authors have adopted opposite conventions for the meaning of the words ``specialization" and ``generalization". 
\begin{defn}\cite[Definition 1.1.9]{HuberEtale}
Let $X$ be a topological space and let $x, y \in X$. If $y \in \overline{\{x\}}$ then we say that \emph{$x$ is a specialization of $y$} (resp. \emph{$y$ is a generalization of $x$}). If $x$ is a specialization of $y$ we write $y \succ x$. The relation $\succ$ is then called the \emph{specialization preorder} on $X$. A \emph{maximal point} of $X$ is a point which is maximal with respect to the specialization preorder. We write $X_{\mathrm{max}}$ for the set of maximal points of $X$. 
\end{defn}
\egapar We remark that if $X$ is a (derived) rigid space and $U \subseteq X$ is an open subset, then $U$ is a partially-proper open subset if and only if it is closed under specializations in $X$. 
\begin{lem}\label{lem:maximalpoints}
Let $X$ be a (derived) rigid space. Then for every $x \in X$ there is a maximal $y \in X$ such that $y \succ x$, i.e., every $x \in X$ is a specialization of a maximal point. 
\end{lem}
\begin{proof}
By \cite[Lemma 1.1.10(i)]{HuberEtale} the set of generalizations of $x$ forms a chain; now use Zorn's lemma.
\end{proof}
\begin{defn}\cite[Definition 5.1.2]{HuberEtale}
A (derived) rigid space $X$ is called \emph{taut} if it is quasi-separated and for every quasi-compact open subset $U \subseteq |X|$, the closure $\overline{U}$ of $U$ in $|X|$ is also quasi-compact. A morphism $f: X \to Y$ of (derived) rigid spaces is called \emph{taut} if for every taut open subset $U \subseteq Y$ the preimage $f^{-1}(U)$ is taut.  
\end{defn}
\egapar From \cite[Lemma 5.1.3, Lemma 5.1.4]{HuberEtale} we have the following: Every qcqs morphism and every partially proper morphism of derived rigid varieties is taut. Further, the class of taut morphisms is stable under composition and base-change, and is right-cancellative.
\begin{thm}[Huber]\label{thm:Huberpartiallyproperetale}
If $f: X \to Y$ is a morphism between taut classical rigid spaces which is \'etale and partially-proper, then for every $x \in X_{\mathrm{max}}$ there exists partially-proper open subsets $f(x) \in V_x \subseteq Y$ and $x \in U_x \subseteq f^{-1}(V_x)$ such that $f$ restricts to a finite \'etale morphism $U_x \to V_x$.
\end{thm}
\begin{proof}
This follows from claim (a) in the proof of \cite[Proposition 8.3.4]{HuberEtale}. 
\end{proof}
\begin{rmk}
With notations as in Theorem \ref{thm:Huberpartiallyproperetale}. We note that the subsets $\{U_x\}_{x \in X_\mathrm{max}}$ cover $X$. Indeed, one has $X_\mathrm{max} \subseteq U := \bigcup_{x \in X_{\mathrm{max}}} U_x$, and $U$ is closed under specializations, so the claim follows from Lemma \ref{lem:maximalpoints}.
\end{rmk}
\egapar We can bootstrap Theorem \ref{thm:Huberpartiallyproperetale} to obtain the corresponding result for morphisms between \emph{derived} rigid spaces.
\begin{prop}\label{prop:ppetalestructure}
If $f: X \to Y$ is a morphism between taut derived rigid spaces which is étale and partially-proper, then for every $x \in X_{\mathrm{max}}$ there exists partially-proper open subsets $f(x) \in V_x \subseteq Y$ and $x \in U_x \subseteq f^{-1}(V_x)$ such that $f$ restricts to a finite \'etale morphism $U_x \to V_x$.
\end{prop}
\begin{proof}
This follows from Theorem \ref{thm:topologicalinvariance}, Theorem \ref{thm:Huberpartiallyproperetale}, and Lemma \ref{lem:etaleliftinggeneral}.
\end{proof}
\egapar We recall that $\mathbf{D}^n$ (resp. $\mathring{\mathbf{D}}^n$) is our notation for the $n$-dimensional closed (resp. open) unit polydisk, regarded as a rigid space over $K$.
\egapar Let $Y$ and $U$ be classical affinoid rigid spaces and let $U^\prime \subseteq U$ be an affinoid subdomain. We recall \cite[\nopp 9.6.2]{BGR} that we write $U^\prime \Subset_Y U$ and say that \emph{$U^\prime$ is relatively compact in $U$ over $Y$} if there exists $n \geqslant 0$ and a Zariski-closed immersion $U \to Y \times \mathbf{D}^n$ over $Y$ such that $U^\prime \to Y \times \mathbf{D}^n$ factors over $Y \times \mathring{\mathbf{D}}^n \to Y \times \mathbf{D}^n$.
\begin{defn}
A morphism $f: X \to Y$ of classical rigid spaces over $K$ is called \emph{Kiehl partially-proper} if it is separated and for every $K$-affinoid variety $Y^\prime = \operatorname{Spec}(A)$ the pullback $f^\prime : X^\prime \to Y^\prime$ satisfies the following: there exists two affinoid covers $\{U_i\}_{i \in I}$ and $\{U_i^\prime\}_{i \in I}$ of $X^\prime$ such that $U_i^\prime \Subset_Y U_i$ for every $i \in I$. A morphism of derived rigid spaces is called Kiehl partially-proper if it is separated and its classical truncation is Kiehl partially-proper. A morphism of derived rigid spaces is called \emph{Kiehl proper} if it is Kiehl partially-proper and quasi-compact. 
\end{defn}
\begin{thm}\label{thm:KiehlPartiallyProper}
If $f: X \to Y$ is a Kiehl partially-proper morphism of rigid spaces over $K$ then it is partially proper. The converse holds if $K$ is discretely valued.
\end{thm}
\begin{proof}
The forwards direction is \cite[Lemma 1.5.9]{HuberEtale}. The converse is a hard theorem essentially due to L\"utkebohmert \cite{Lutkebohmert1990} and Huber \cite[Remark 1.3.18(iii)]{HuberEtale}.
\end{proof}
\begin{prop}\label{prop:SmooothPartiallyProper}
Let $f: X \to Y$ be a smooth and Kiehl partially-proper morphism of derived rigid spaces over $K$ with $Y = \operatorname{Spec}(A)$ affinoid. Then there exists a cover $\{ V_i \}_{i \in I}$ of $X$ by partially-proper open subsets and integers $d_i \geqslant 0$ such that $f|_{V_i}$ factors as $V_i \to \mathring{\mathbf{D}}^{d_i} \times Y \to Y $ where the first morphism is a quasi-smooth closed immersion and the second morphism is the projection.
\end{prop}
\begin{proof}
Since $f$ is Kiehl partially-proper, there exists two covers $\{U_i\}_{i \in I}$ and $\{U_i^\prime\}_{i \in I}$ and $d_i \geqslant 0$ such that, for every $i \in I$ there is a commutative square
\begin{equation}
\begin{tikzcd}
	{U_i^\prime} & {Y \times \mathring{\mathbf{D}}^{d_i}} \\
	{U_i} & {Y \times \mathbf{D}^{d_i}}
	\arrow[from=1-1, to=1-2]
	\arrow[hook, from=1-1, to=2-1]
	\arrow[hook, from=1-2, to=2-2]
	\arrow[from=2-1, to=2-2]
\end{tikzcd}
\end{equation}
where the vertical arrows are open immersions and the bottom arrow is a Zariski-closed immersion, and all morphisms are over $Y$. We define $V_i := U_i \times_{\mathbf{D}^{d_i}} \mathring{\mathbf{D}}^{d_i}$. Then $V_i \to Y \times  \mathring{\mathbf{D}}^{d_i}$ is a Zariski-closed immersion, $V_i \to X$ is an open immersion, and $U_i^\prime \subseteq V_i$, so the $\{V_i\}_{i \in I}$ cover $X$. The composite $V_i \to Y \times  \mathring{\mathbf{D}}^{d_i} \to Y$ is $f|_{V_i}$, which is smooth by assumption, so by Lemma \ref{lem:ZariskiclosedRegular2}, the Zariski-closed immersion $V_i \to Y \times  \mathring{\mathbf{D}}^{d_i}$ is quasi-smooth. Further, $V_i \to Y \times  \mathring{\mathbf{D}}^{d_i} \to Y$ is partially-proper by Corollary \ref{cor:partiallyproperproperties}, hence, by Corollary \ref{cor:partiallyproperproperties}(ii) the morphism $V_i \to X$ is partially proper as $f: X \to Y$ is, c.f. Theorem \ref{thm:KiehlPartiallyProper}.
\end{proof}
\section{Some features of the six-functor formalism}\label{sec:6FF1}
\egapar The theory of six-functor formalisms has become an invaluable tool in the proof of all kinds of duality theorems, c.f. \cite{ScholzeSixFunctors}. In this section we freely make use of the language of six-functor formalisms as developed in \cite{heyer_6-functor_2024}. 
\egapar We briefly recall the construction of the six-functor formalism on $\operatorname{Shv}_{\tau_{\mathrm{rat}}}(\operatorname{Afnd})$ from \cite[\S 3.1.5]{SoorThesis}. Firstly, there is a ``trivial" six-functor formalism on $(\operatorname{Afnd}, \operatorname{all})$ in which for every $X = \operatorname{Spec}(A) \in \operatorname{Afnd}$ we set $\operatorname{QCoh}(X) := \operatorname{Mod}_A\mathscr{V}$ equipped with its natural (closed) monoidal structure $\hat{\otimes}^\mathbf{L}_A$. For every morphism $f: \operatorname{Spec}(A) \to \operatorname{Spec}(B)$ in $\operatorname{Afnd}$, we define $f^* := A \hat{\otimes}^\mathbf{L}_B-$, $f_*$ to be the forgetful functor $\operatorname{Mod}_A \mathscr{V} \to \operatorname{Mod}_B \mathscr{V}$, we define $f_! := f_*$, and $f^! := R\underline{\operatorname{Hom}}_B(A,-)$. By the extension procedure of \cite[Theorem 3.4.11]{heyer_6-functor_2024} or \cite[\S 2.3.2]{SoorThesis}, this six-functor formalism extends uniquely to a six-functor formalism
\begin{equation}\label{eq:6ff1}
    \operatorname{QCoh}: \operatorname{Corr}(\operatorname{Shv}_{\tau_{\mathrm{rat}}}(\operatorname{Afnd}), E)^\otimes \to \operatorname{Pr}^{L, \otimes}_{\mathrm{st},\mathscr{V}}
\end{equation}
here $E$ is the class of edges for which $!$-functors are defined. The functor $\operatorname{QCoh}^*: \operatorname{Shv}_{\tau_{\mathrm{rat}}}(\operatorname{Afnd})^\mathrm{op} \to \operatorname{Pr}^{L, \otimes}_{\mathrm{st}, \mathscr{V}}$ is right Kan extended from $\operatorname{Afnd}^\mathrm{op}$, in particular it agrees with the definition of $\operatorname{QCoh}$ from \eqref{eq:Qcohanalytic}. The class $E$ of $!$-able morphisms is stable under base change and composition, and is right-cancellative. The class $E$ is \emph{$!$-local on the source and target}, \emph{$*$-local on the target}, \emph{stable under disjoint unions}, and \emph{tame}, c.f. \cite[Definition 2.3.9]{SoorThesis} and \cite[Theorem 3.4.11]{heyer_6-functor_2024}. The class $E$ contains all representable morphisms, and for such morphisms $f$ there is a canonical equivalence $f_! \simeq f_*$. Further, by \cite[Theorem 3.1.40]{SoorThesis}, the class $E$ contains all qcqs morphisms of derived rigid spaces and, for such a morphisms $f$, there is again a canonical equivalence $f_! \simeq f_*$. 
\egapar The first goal of this section is to prove the following. 
\begin{thm}\label{thm:taut!able}
Let $f: X \to Y$ be a taut morphism of derived rigid spaces. Then $f$ belongs to the class $E$ of $!$-able morphisms.
\end{thm}
\egapar The proof of Theorem \ref{thm:taut!able} relies on particular properties of certain subsets of the underlying topological space. Hence, we recall the following Definition.
\begin{defn}\cite[\nopp 1.1.13]{HuberEtale}
Let $X$ be a locally spectral topological space. A subset $T \subseteq X$ is called:
\begin{enumerate}[(i)]
    \item \emph{retro-compact} if $T \cap U$ is quasi-compact for every quasi-compact open subset $U \subseteq X$. 
    \item \emph{globally constructible} if it is an element of the Boolean algebra generated by retro-compact subsets of $X$, 
    \item \emph{constructible} if every point of $X$ has an open neighbourhood $U$ such that $U \cap T$ is globally constructible in $U$,
    \item \emph{open constructible} (resp. \emph{closed constructible}) if it is both constructible and open (resp. closed) in $X$.
    \item \emph{ind-constructible} (resp. \emph{pro-constructible}) if every point of $x$ has an open neighbourhood $U$ such that $T \cap U$ is the union (resp. intersection) of constructible subsets of $U$.
\end{enumerate}
\end{defn}
\begin{rmk}
The open constructible subsets are precisely those constructible subsets $T$ such that $T \cap U$ is a quasi-compact open subset of $U$, for every quasi-compact open subset $U \subseteq X$. (The proof of this fact is trivial). 
\end{rmk}
\begin{warning}
The open (resp. closed) constructible subsets as defined above are not the same as the open (resp. closed) subsets in the constructible topology: these are in fact the ind-constructible (resp. pro-constructible) subsets. This follows the nomenclature used in \cite{HuberEtale}. We warn the reader that not all authors use this convention.
\end{warning}
\egapar If $X \in \operatorname{dRig}$ and $Z \subseteq |X|$ is a closed subset with open complement $j: V \hookrightarrow X$, we define the full subcategory $\Gamma_Z\operatorname{QCoh}(X) \subseteq \operatorname{QCoh}(X)$ on those objects $M$ such that $j^*M \simeq 0$. We note that the tautological inclusion $\Gamma_Z\operatorname{QCoh}(X)$ preserves colimits. 
\begin{lem}\label{lem:GammaZcolimit}
If $Z \subseteq |X|$ is closed constructible (with open constructible complement $j: U \hookrightarrow X$), then the tautological inclusion $\operatorname{incl}_Z: \Gamma_Z\operatorname{QCoh}(X) \to \operatorname{QCoh}(X)$ a right adjoint:
\begin{equation}
    \operatorname{incl}_Z \dashv  \Gamma_Z := \operatorname{Fib}(\operatorname{id} \to j_*j^*). 
\end{equation}
Further, the right adjoint $\Gamma_Z$ is colimit-preserving, and $\Gamma_Z\operatorname{QCoh}(X)$ is presentable. 
\end{lem}
\begin{proof}
Since $j:U \hookrightarrow X$ is open constructible, it is a quasi-compact morphism, and hence $j_! \xrightarrow{\sim} j_*$ preserves limits and colimits and satisfies base-change, so that $j^*j_* \xrightarrow{\sim} \operatorname{id}$. The Lemma follows formally from this.
\end{proof}
\egapar The proof of Theorem \ref{thm:tautuniversal!descent} rests on a delicate cofinality argument. For this it is necessary to introduce the following Definition.
\begin{defn}
Let $X \in \mathrm{dRig}$ and let $\mathscr{C} \in \mathrm{Pr}^{L}_{\mathrm{st}}$. A fully-faithful colimit-preserving functor $F: \mathscr{C} \hookrightarrow \operatorname{QCoh}(X)$ with target $\operatorname{QCoh}(X)$, is called \emph{compactly supported} if there exists a quasi-compact closed constructible subset $Z \subseteq |X|$ such that $F$ factors over $\Gamma_Z\operatorname{QCoh}(X) \subseteq \operatorname{QCoh}(X)$. We denote the full subcategory of ${\mathrm{Pr}^L_{st}}_{/\operatorname{QCoh}(X)}$ on such by $cs(X)$. 
\end{defn}
\begin{thm}\label{thm:tautuniversal!descent}
Let $X$ be a taut derived rigid space and denote by $qc(X)$ the poset of quasi-compact open subsets of $X$. Then $\{U \to X\}_{U \in qc(X)}$ is a universal $!$-cover. 
\end{thm}
\begin{proof}
For the sake of brevity we may write $\mathscr{Q} := \operatorname{QCoh}(X)$. We consider the colimit
\begin{equation}\label{eq:colim1}
   \underset{[F:\mathscr{C} \to \mathscr{Q}] \in cs(X)}{\operatorname{colim}} \mathscr{C},
\end{equation}
taken in ${\mathrm{Pr}^{L}_{\mathrm{st}}}_{/\mathscr{Q}}$. Now we consider two further colimits (again both taken in ${\mathrm{Pr}^{L}_{\mathrm{st}}}_{/\mathscr{Q}}$). Firstly, the colimit \begin{equation}\label{eq:colim2}
\underset{[i: U \hookrightarrow X] \in qc(X)}{\operatorname{colim}} \operatorname{QCoh}_!(U)
\end{equation}
here we note that $i_!$ is defined since $i: U \hookrightarrow X$ is quasi-compact as $X$ is quasi-separated, and secondly
\begin{equation}\label{eq:colim3}
\underset{[\iota: Z \hookrightarrow |X|] \in ccc(X)} {\operatorname{colim}}\Gamma_Z \operatorname{QCoh}(X)
\end{equation}
where $ccc(X)$ is\footnote{Here $ccc$ stands for quasi-Compact and Closed Constructible.} the poset of quasi-compact closed constructible subsets of $|X|$. 
We note that the indexing category $qc(X)$ for \eqref{eq:colim2} maps to the indexing category $cs(X)$ for \eqref{eq:colim1} by $[i: U \to X] \mapsto [i_!: \operatorname{QCoh}(U) \to \operatorname{QCoh}(X)]$. Indeed, since $X$ is taut, by \cite[Lemma 5.3.4(i)]{HuberEtale} there is a quasi-compact closed constructible subset $Z$ such that $U \subseteq Z$. If $j: V \hookrightarrow X$ is the complement of $Z$, then $j^*i_! \simeq 0$ by base-change, and so $i_!$ factors over $\Gamma_Z\operatorname{QCoh}(X)$. We claim that this functor is right cofinal. We use Quillen's Theorem A. We need to show that, for any fixed $[\mathscr{C} \to \mathscr{Q}] \in cs(X)$ that $qc(X) \times_{cs(X)} cs(X)_{\mathscr{C}/}$ is weakly contractible. Firstly, it is nonempty: let $[Z \to |X|] \in ccc(X)$ be such that $\mathscr{C} \to \mathscr{Q}$ factors over $\Gamma_Z\operatorname{QCoh}(X)$ and let $U \supseteq Z$ be a quasi-compact open subset containing $Z$, with $i: U \hookrightarrow X$ the inclusion. By $*$-descent we see that the natural transformation $\operatorname{id} \to i_*i^*$ restricts to an equivalence on $\Gamma_Z\operatorname{QCoh}(X)$. Hence using that $i_! \xrightarrow{\sim} i_*$ (as $i$ is qcqs), we see that $\mathscr{C} \to \mathscr{Q}$ factors over $i_!$ and hence $qc(X) \times_{cs(X)} cs(X)_{\mathscr{C}/}$ is nonempty. It is easily seen to be filtered, by using that $qc(X)$ is filtered. Hence, it is weakly contractible. 

It is clear that the indexing category $ccc(X)$ for \eqref{eq:colim3} maps to $cs(X)$ by $[\iota: Z \to |X|] \mapsto \Gamma_Z\operatorname{QCoh}(X)$. It is essentially tautological that this functor is right cofinal. Putting this together we obtain an equivalence in ${\mathrm{Pr}^{L}_{\mathrm{st}}}_{/\mathscr{Q}}$: 
\begin{equation}\label{eq:cofinalityQCoh!}
\underset{[i: U \hookrightarrow X] \in qc(X)}{\operatorname{colim}} \operatorname{QCoh}_!(U) \simeq \underset{[\iota: Z \hookrightarrow |X|] \in ccc(X)} {\operatorname{colim}}\Gamma_Z \operatorname{QCoh}(X).
\end{equation}
Next, we claim that the canonical morphism 
\begin{equation}
    \underset{[\iota: Z \hookrightarrow |X|] \in ccc(X)} {\operatorname{colim}}\Gamma_Z \operatorname{QCoh}(X) \to \operatorname{QCoh}(X)
\end{equation}
is an equivalence. By passing to right adjoints and recalling how colimits in $\operatorname{Pr}^L_{\mathrm{st}}$ are computed, it is equivalent to show that the natural functor 
\begin{equation}\label{eq:GammaZdescent}
    \operatorname{QCoh}(X) \to  \underset{[\iota: Z \hookrightarrow |X|] \in ccc(X)^{\mathrm{op}}} {\operatorname{lim}}\Gamma_Z \operatorname{QCoh}(X) 
\end{equation}
is an equivalence. First we show that it is fully-faithful. This amounts to showing that for each $M \in \operatorname{QCoh}(X)$, the canonical morphism
\begin{equation}
\underset{[\iota: Z \hookrightarrow |X|] \in ccc(X)}{\operatorname{colim}}\Gamma_Z M  \to M
\end{equation}
is an equivalence. For each $Z \in ccc(X)$ we write $j_Z:V \hookrightarrow X$ for its open constructible complement. Since we are in a stable setting, it is equivalent to show that $\operatorname{colim}_Zj_{Z,*}j^*_ZM \simeq 0$. If $i_0: U \hookrightarrow X$ is the inclusion of an affinoid subspace then by \cite[Lemma 5.3.4(i)]{HuberEtale} there exists $Z_0 \in ccc(X)$ whose complement $[j_{Z_0}: V_0 \hookrightarrow X] \in ccc(X)$ satisfies $V_0 \cap U = \emptyset$. Since $j_{Z_0,*}$ satisfies base-change one has $i^*_0j_{Z_0,*} \simeq 0$. Hence $i_0^* \operatorname{colim}_Zj_{Z,*}j^*_ZM \simeq \operatorname{colim}_Zi_0^*j_{Z,*}j^*_ZM \simeq 0$. Since $i_0$ was arbitrary, we conclude by $*$-descent that $\operatorname{colim}_Zj_{Z,*}j^*_ZM \simeq 0$. 

Secondly, we show that \eqref{eq:GammaZdescent} is essentially surjective. This amounts to showing that for each Cartesian section $(M_Z)_{Z \in ccc(X)}$ on the right-hand side of \eqref{eq:GammaZdescent} and each $Z^\prime \in ccc(X)$, the canonical morphism $M_{Z^\prime} \to \Gamma_{Z^{\prime}} \operatorname{colim}_Z \operatorname{incl}_Z M_Z$ is an equivalence. This follows straightforwardly from the fact that $\Gamma_Z$ preserves colimits, c.f. Lemma \ref{lem:GammaZcolimit}. Hence by \eqref{eq:cofinalityQCoh!} and passing to right adjoints we have shown that the canonical morphism
\begin{equation}
    \operatorname{QCoh}^!(X) \to \underset{qc(X)^{\mathrm{op}}}{\operatorname{lim}} \operatorname{QCoh}^!(U)
\end{equation}
is an equivalence. It is easy to show that $qc(X)$ is cofinal in the sieve that it generates, hence, we obtain the statement of the Theorem. 
\end{proof}
\egapar Now we can complete the proof of Theorem \ref{thm:taut!able}.
\begin{proof}[Proof of Theorem \ref{thm:taut!able}]
Since the class $E$ is $*$-local on the target, we reduce to the case when $Y$ is affinoid, hence $X$ is taut. Since the class $E$ is $!$-local on the source, the Theorem follows from Theorem \ref{thm:tautuniversal!descent} together with the fact that every qcqs morphism belongs to $E$. 
\end{proof}
\begin{rmk}
Let $X$ be a taut derived rigid space with structure morphism $f: X \to *$. Then Theorem \ref{thm:tautuniversal!descent} allows for a very intuitive and canonical description of the functor $f_!$. One has $f_! \simeq \operatorname{colim}_{U \in qc(X)} f_{U,*}i_U^!$, where for each $U \in qc(X)$ we write $i_U: U \hookrightarrow X$ for the inclusion and $f_U := f i_U$ for the restriction of $f$ to $U$. Hence in this case we should think of $f_!$ as ``global sections with quasi-compact support".  
\end{rmk}
\begin{lem}
Let $f: X \to Y$ be a separated \'etale morphism of derived rigid spaces which belongs to the class $E$. (By Theorem \ref{thm:taut!able} this is satisfied if $f$ is taut). Then there is a canonical morphism $f^! \to f^*$.  
\end{lem}
\begin{proof}
We first note that if $Z =U \coprod V$ is a derived rigid space written as a disjoint union of two subspaces and $s :U \to Z, t:V \to Z$ are the inclusions then there is a canonical equivalence $s^! \xrightarrow[]{\sim} s^*$, see for instance the argument given in the proof of \cite[Theorem 2.3.17]{SoorThesis}. Now if $f$ is separated \'etale, then $\Delta_f$ is an open immersion and also a Zariski-closed immersion. Hence $\Delta_f$ corresponds to an open and closed inclusion into $X \times_Y X$ and so by the preceding there is a natural equivalence $\Delta_f^! \xrightarrow{\sim} \Delta_f^*$. As in \cite[Lemma 4.6.4]{heyer_6-functor_2024} this yields the canonical morphism $f^! \to f^*$. 
\end{proof}
\begin{prop}\label{prop:partiallyproperopen}
Let $j: U \to X$ be a partially proper open immersion of derived rigid spaces. Then the canonical morphism $j^! \xrightarrow{\sim} j^*$ is an equivalence.
\end{prop}
\begin{proof}
\emph{Step 1.} We first prove the theorem in a simple case. Let $X = \operatorname{Spec}(A)$ be an affinoid derived rigid space. Let $f, g \in H^0(A)$ without common zero and let $V = \operatorname{Spec}(A\langle f/g \rangle)$ be the corresponding rational localization. Let $U$ be the open complement of $S:=\overline{V}$ in $|X|$, equivalently, $U$ is the interior of the closed constructible subset $|X| \setminus V$. Then by \cite[Lemma 5.3.4(ii)]{HuberEtale} the morphism $j:U \hookrightarrow X$ is partially-proper. We define the following rational open subsets of $X$:
\begin{equation}
\begin{aligned}
   V_n := \operatorname{Spec}(A\langle \varpi f^n/g^n\rangle) && \text{ and } && U_n:= \operatorname{Spec}(A\langle g^n/\varpi f^n\rangle), 
\end{aligned}
\end{equation}
it is not hard to see that $U_n \cup V_n = X$ and $U_n \cap V_{n+1} = \emptyset$ for every $n \geqslant 1$, and that $\bigcup_{n \geqslant 1} U_n = U$. Hence the hypotheses of \cite[Proposition 3.1.53]{SoorThesis} are satisfied in this case, and so the canonical morphism $j^! \to j^*$ is an equivalence. 

\emph{Step 2.} To prove the proposition in general, we use descent to reduce to the case when $X$ is affinoid. If $j: U \hookrightarrow X$ is a partially-proper open subspace of an affinoid space, then by \cite[Lemma 8.1.8(iii)]{huber_continuous_1993} we can write $U = \bigcup_i U_i $ where each $j_i: U_i \to X$ is of the form considered in Step 1. By Step 1, we know that each $j_i$ satisfies $j_i^! \xrightarrow[]{\sim} j_i^*$. Hence by descent on $U$ we conclude that $j^* \xrightarrow{\sim} j^!$.
\end{proof}
\begin{prop}\label{prop:finiteetale!}
Let $f:X \to Y$ be a finite \'etale morphism of derived rigid spaces. Then the canonical morphism $f^! \to f^*$ is an equivalence.
\end{prop}
\begin{proof}
By descent on $Y$ we reduce to the case when $X = \operatorname{Spec}(A)$ and $Y = \operatorname{Spec}(B)$ are affinoid. We construct an explicit inverse to the natural morphism $f^! \to f^*$. By Lemma \ref{lem:finiteetale} $A$ is dualizable as a $B$-module. Hence it suffices to construct an inverse to the natural morphism $A^\lor := R\underline{\operatorname{Hom}}_B(A,B)=f^!B \to f^*B=A$. This is done using the theory of the trace which we now recall. As $A$ is dualizable there is a canonical equivalence $R\underline{\operatorname{End}}_B(A) \simeq A \hat{\otimes}^\mathbf{L}_BA^\lor$ which we may compose with the evaluation map to obtain a canonical morphism $R\underline{\operatorname{End}}_B(A) \to B$. On the other hand the adjoint to the multiplication morphism $A \hat{\otimes}^\mathbf{L}_B A \to A$ gives a canonical morphism $A \to R\underline{\operatorname{End}}_B(A)$. The \emph{trace morphism} $\operatorname{tr}: A \to B$ is defined to be the composite $A \to R\underline{\operatorname{End}}_B(A) \to B$. The \emph{trace pairing} is defined to be the composite $A \hat{\otimes}^\mathbf{L}_BA \to A \xrightarrow[]{\mathrm{tr}} B$ where the first morphism is again multiplication. By adjunction we obtain a canonical morphism $A \to A^\lor$ which is the required inverse. 
\end{proof}
\begin{thm}\label{thm:partiallyproperDetale}
Let $f: X \to Y$ be a partially proper \'etale morphism of derived rigid spaces. Then the canonical morphism $f^! \to f^*$ is an equivalence.
\end{thm}
\begin{proof}
By $*$-descent, the condition that $f^! \to f^*$ is an equivalence is local in the analytic topology on $Y$. By Proposition \ref{prop:partiallyproperopen} and descent on $X$, this condition is also local in the partially-proper analytic topology on $X$. Hence by Proposition \ref{prop:ppetalestructure} we reduce to the case when $f$ is finite \'etale. This follows from Proposition \ref{prop:finiteetale!}.
\end{proof}
\egapar In the language of abstract six-functor formalisms, Theorem \ref{thm:partiallyproperDetale} implies that every partially proper \'etale morphism of derived rigid spaces is $\mathrm{D}$-\'etale in the sense of \cite[Definition 4.6.1]{heyer_6-functor_2024} for the six-functor formalism $\mathrm{D} = \operatorname{QCoh}$ of \eqref{eq:6ff1}, c.f. \cite[Lemma 4.6.4]{heyer_6-functor_2024}.
\egapar In the remainder of this section we make use of the notion of a \emph{suave} morphism, introduced in \cite[\S4.5]{heyer_6-functor_2024}. We say that a morphism $f: X \to Y$ in $\operatorname{Shv}_{\tau_{\mathrm{rat}}}(\operatorname{Afnd})$ is \emph{suave} if it belongs to the class $E$ of $!$-able morphisms and it is suave in the sense of \cite[Definition 4.5.1]{heyer_6-functor_2024}. By \cite[Lemma 4.5.9]{heyer_6-functor_2024}, the class of suave morphisms is stable under base change and composition. The relevance to Grothendieck duality theory is that for every suave morphism $f: X \to Y$ there is a canonical equivalence $f^* \hat{\otimes}_X f^!1_Y \simeq f^!$, c.f. \cite[Corollary 4.5.11]{heyer_6-functor_2024}. By Theorem \ref{thm:taut!able} it makes sense to ask whether any taut morphism is suave. 
\begin{lem}\label{lem:suavelocal}
Let $f: X \to Y$ be a taut morphism in $\operatorname{dRig}$. Then, the property that $f$ is suave is local in the analytic topology on $Y$ and local in the partially-proper analytic topology on $X$.
\end{lem}
\begin{proof}
The first part is \cite[Lemma 4.5.7]{heyer_6-functor_2024}, and the second part follows from Proposition \ref{prop:partiallyproperopen} and \cite[Lemma 4.5.8]{heyer_6-functor_2024}.
\end{proof}
\egapar By substituting Theorem \ref{thm:partiallyproperDetale} for Proposition \ref{prop:partiallyproperopen}, it should be the case that, for a taut morphism $f: X \to Y$ in $\operatorname{dRig}$, the property of being suave is local in the partially-proper \'etale topology on $X$. The reason why we did not state this in Lemma \ref{lem:suavelocal} is because, in this paper we have not proven \'etale descent for quasi-coherent sheaves. 
\begin{prop}\label{prop:Opendisksuave}
The canonical morphism $f: \mathring{\mathbf{D}}^1 \to *$ is suave and there is an equivalence $f^!1_* \simeq 1_{\mathring{\mathbf{D}}^1}[1]$. 
\end{prop}
\begin{proof}
The proof is essentially the same as \cite[Proposition 4.31]{soor_quasicoherent_2023} which was in turn based on \cite[Theorem 12.17]{CondensedComplexGeometry}. We recall that the canonical morphism $K[T] \to K \langle T \rangle$ is a homotopy monomorphism, so that the forgetful functor $\operatorname{Mod}_{K\langle T \rangle}\mathscr{V} \to \operatorname{Mod}_{K[T]} \mathscr{V}$ is fully-faithful. We define the following ind-Banach algebra:
\begin{equation}
    K[T, T^{-1}\rangle := \left\{ \sum_{n \in \mathbf{Z}}a_n T^n: a_n = 0 \text{ for }n\gg 0, a_n\to 0 \text{ as } n \to -\infty \right\}.  
\end{equation}
In other words, it is those functions on the closed annulus $\{ |T| \geqslant 1\}$ with a pole at $\infty$. There is a canonical equivalence $ K \langle T \rangle\hat{\otimes}^\mathbf{L}_{K[T]}   K [T, T^{-1}\rangle\xrightarrow[]{\sim} K \langle T, T^{-1}\rangle$. In particular, if an object $M \in \operatorname{Mod}_{K\langle T \rangle}\mathscr{V}$ dies on base-change to $K \langle T, T^{-1} \rangle$, when viewed in $\operatorname{Mod}_{K[T]}\mathscr{V}$ the object $M$ dies on base-change to $K[T, T^{-1} \rangle$.

We factor $f$ as $f = gj$ where $j: \mathring{\mathbf{D}}^1 \hookrightarrow \mathbf{D}$ is the inclusion and $g: \mathbf{D}^1 \to *$ is the structure morphism. Let $M \in \operatorname{QCoh}(\mathring{\mathbf{D}}^1)$. By base-change, $j_!M$ dies on base-change to $K\langle T, T^{-1} \rangle$ and so, when viewed as an object of $\operatorname{Mod}_{K[T]} \mathscr{V}$ it dies on base-change to $K[T, T^{-1} \rangle$. In particular it dies on base-change to $K(\!(T^{-1})\!)$. Now for $N \in \operatorname{QCoh}(*)$ one computes as in \cite[Theorem 12.17]{CondensedComplexGeometry}, using that any morphism from $M$ to a $K(\!(T^{-1})\!)$-module vanishes:
\begin{equation}
\begin{aligned}
R\underline{\operatorname{Hom}}_{\mathring{\mathbf{D}}^1}(M, f^!N) &\simeq R\underline{\operatorname{Hom}}_{K }(j_! M, N)  \\
&\simeq R\underline{\operatorname{Hom}}_{K[T]}(j_!M, R\underline{\operatorname{Hom}}_K(K[T],N)) \\ 
&\simeq R\underline{\operatorname{Hom}}_{K[T]}(j_!M, N(\!(T^{-1})\!)/N[T]) \\
&\simeq R\underline{\operatorname{Hom}}_{K[T]}(j_!M, N[T][1]).
\end{aligned}
\end{equation}
Now since $j_!M$ dies on base-change to $K[T, T^{-1} \rangle $ we have further that 
\begin{equation}
\begin{aligned}
R\underline{\operatorname{Hom}}_{K[T]}(j_!M, N[T][1]) &\simeq R\underline{\operatorname{Hom}}_{K[T]}(j_!M, N[T, T^{-1}\rangle/N[T])  \\ 
&\simeq R\underline{\operatorname{Hom}}_{K[T]}(j_!M, N\langle T, T^{-1}\rangle/N\langle T \rangle) \\
&\simeq R\underline{\operatorname{Hom}}_{K\langle T \rangle}(j_!M, N\langle T, T^{-1}\rangle/N\langle T \rangle)  \\
&\simeq R\underline{\operatorname{Hom}}_{K\langle T \rangle}(j_!M, N\langle T \rangle [1])  \\
&\simeq R\underline{\operatorname{Hom}}_{\mathring{\mathbf{D}}^1}(M, f^*N[1]).
\end{aligned}
\end{equation}
so that $f^! \xrightarrow{\sim} f^*[1]$ and $f$ is suave. Here in the second line we used that $N[T, T^{-1}\rangle/N[T] = N\langle T, T^{-1}\rangle/N\langle T \rangle$, in the third line we used fully-faithfulness of $\operatorname{Mod}_{K\langle T \rangle} \mathscr{V} \to \operatorname{Mod}_{K[T]} \mathscr{V}$, in the fourth line we used that $j_!M$ dies on base-change to $K \langle T, T^{-1} \rangle$, and in the last line we used that $j^! \xrightarrow{\sim} j^*$ (Proposition \ref{prop:partiallyproperopen}). 
\end{proof}
\begin{prop}\label{prop:ZariskiclosedSuave}
Let $f: X \to Y$ be a quasi-smooth Zariski-closed immersion of derived rigid spaces. Then $f$ is suave. 
\end{prop}
\begin{proof}
By Lemma \ref{lem:Zariskiclosedquasismooth} and Lemma \ref{lem:suavelocal}, we reduce to the case when $X = \operatorname{Spec}(B)$ and $Y= \operatorname{Spec}(A)$ are affinoid and $A \to B$ is of the form $A \to A/^\mathbf{L}(f_1,\dots,f_n)$ for some $f_i: K \to A$. We recall that $A/^\mathbf{L}(f_1,\dots,f_n)$ is computed via the Koszul complex and therefore is perfect and hence dualizable in $\operatorname{Mod}_A \mathscr{V}$. Therefore the claim follows from \cite[Lemma 4.5.10]{heyer_6-functor_2024}.
\end{proof}
\begin{thm}\label{thm:GD1}
Let $f: X \to Y$ be a smooth and Kiehl partially-proper morphism of derived rigid spaces. Then $f$ is suave, $f^!1_Y$ is invertible and there is a canonical equivalence $f^!1_Y \simeq (\Delta^!_f1_{X \times_Y X})^{\otimes -1}$, where $\Delta_f: X \to X \times_Y X$ is the diagonal of $f$.
\end{thm}
\begin{proof}
For the first part, by Lemma \ref{lem:suavelocal}, the question is local in the partially-proper analytic topology on $X$ and the analytic topology on $Y$. Hence, by Theorem \ref{thm:KiehlPartiallyProper}, we may assume that $f$ factors as $X  \to Y \times \mathring{\mathbf{D}}^d \to Y$ where the first morphism is a quasi-smooth Zariski-closed immersion and the second is the projection. The first morphism is suave by Proposition \ref{prop:ZariskiclosedSuave}. The second morphism is suave because it is the composite of $d$ suave morphisms, by Proposition \ref{prop:Opendisksuave}. Hence, $f$ is suave. For the second part, by Lemma \ref{lem:SmoothSection} and Proposition \ref{prop:ZariskiclosedSuave}, the morphism $\Delta_f$ is suave, hence, the claim follows from \cite[Corollary 4.5.18]{heyer_6-functor_2024}, see also \cite[Remark 4.5.12]{heyer_6-functor_2024}. 
\end{proof}
\section{Duality}\label{sec:duality}
\egapar It will be necessary in this section to consider a six-functor formalism defined on the larger topos $\operatorname{Shv}_{\tau_{\mathrm{dri}}}(\operatorname{Aff})$. As in the beginning of \S\ref{sec:6FF1}, there is a ``trivial" six-functor formalism $\operatorname{QCoh}$ on $(\operatorname{Aff}, \operatorname{all})$ in which for every morphism $f: X \to Y$ in $\operatorname{Aff}$ one has $f_! = f_*$. By the extension procedure of \cite[Theorem 3.4.11]{heyer_6-functor_2024} or \cite[\S 2.3.2]{SoorThesis} this six-functor formalism extends uniquely to a six-functor formalism
\begin{equation}\label{eq:Big6FF}
    \operatorname{QCoh}: \operatorname{Corr}(\operatorname{Shv}_{\tau_\mathrm{dri}}(\operatorname{Aff}), E_\mathscr{V})^\otimes \to \operatorname{Pr}^{L, \otimes}_{\mathrm{st}, \mathscr{V}},
\end{equation}
where we have denoted by $E_\mathscr{V}$ the class of $!$-able morphisms in this six-functor formalism. The class $E_\mathscr{V}$ is again \emph{$!$-local on the source and target}, \emph{$*$-local on the target}, \emph{stable under disjoint unions} and \emph{tame} in the sense of \cite[Definition 2.3.9]{SoorThesis} or \cite[Theorem 3.4.11]{heyer_6-functor_2024}. The class $E_\mathscr{V}$ contains all representable morphisms and for such $f$ there is an equivalence $f_! \simeq f_*$. 
\begin{prop}
The functor $i_!: \operatorname{Shv}_{\tau_{\mathrm{rat}}}(\operatorname{Afnd}) \to \operatorname{Shv}_{\tau_{\mathrm{dri}}}(\operatorname{Aff})$ induces a morphism of geometric setups\footnote{See \cite[Definition 2.1.1]{heyer_6-functor_2024}.} $(\operatorname{Shv}_{\tau_{\mathrm{rat}}}(\operatorname{Afnd}), \operatorname{rep}) \to (\operatorname{Shv}_{\tau_{\mathrm{dri}}}(\operatorname{Aff}), \operatorname{rep})$. The restriction of the six-functor formalism of \eqref{eq:Big6FF} to $(\operatorname{Shv}_{\tau_{\mathrm{rat}}}(\operatorname{Afnd}), \operatorname{rep})$ is equivalent to the restriction of the six-functor formalism of \eqref{eq:6ff1} to $(\operatorname{Shv}_{\tau_{\mathrm{rat}}}(\operatorname{Afnd}), \operatorname{rep})$. 
\end{prop}
\begin{proof}
The first part follows from Proposition \ref{prop:NCrepresentable}. The second part is then clear since in both cases, for any representable morphism $f$ one has $f_! = f_*$ is right adjoint to $f^*$, hence, the restriction to $(\operatorname{Shv}_{\tau_{\mathrm{rat}}}(\operatorname{Afnd}), \operatorname{rep})$ is determined by the underlying functor $\operatorname{QCoh}^*: \operatorname{Shv}_{\tau_{\mathrm{rat}}}(\operatorname{Afnd}) \to \operatorname{Cat}_\infty$, which agrees in both cases by the discussion in Paragraph \ref{par:PsiEquivalencePar}. 
\end{proof}
\egapar For the next Lemma it is helpful to recall the notion of a \emph{prim} morphism from \cite[\S 4.5]{heyer_6-functor_2024}. We say that a morphism $f: X \to Y$ in $\operatorname{Shv}_{\tau_{\mathrm{dri}}}(\operatorname{Aff})$ is \emph{prim} if it belongs to the class $E_\mathscr{V}$ of $!$-able morphisms of \eqref{eq:Big6FF} and it is prim in the sense of \cite[Definition 4.5.1]{heyer_6-functor_2024}. The class of prim morphisms is stable under base-change and composition. For the purposes of this article, the most important feature of a prim morphism $f$ is that the functor $f_*$ is colimit-preserving, satisfies the projection formula and is compatible with base change. The first two assertions here follow from \cite[Corollary 4.5.11]{heyer_6-functor_2024} and the latter is \cite[Lemma 4.5.13(ii)]{heyer_6-functor_2024}. 
\begin{lem}\label{lem:A1Gm}
We consider the structure morphism $\pi: a_! \mathbf{A}^{1, \mathrm{alg}}/a_!\mathbf{G}_m^\mathrm{alg} \to *$.
In the six-functor formalism of \eqref{eq:Big6FF}:
\begin{enumerate}[(i)]
    \item The morphism $\pi$ belongs to the class $E_\mathscr{V}$ of $!$-able morphisms;
    \item The morphism $\pi$ is prim.
    \item The unit morphism $\operatorname{id} \to \pi_*\pi^*$ is an equivalence, so that $\pi^*$ is fully-faithful. 
\end{enumerate}
The same is true for any base-change of $\pi$. 
\end{lem}
\begin{proof}
(i): First we consider the canonical morphism $p: * \to */a_! \mathbf{G}_{m}^\mathrm{alg}$. We claim that this is descendable. The adjunction $p^* \dashv p_*$ is comonadic (this is essentially equivalent to the fact that $\operatorname{QCoh}^*$ preserves limits), and identifies $\operatorname{QCoh}(*/a_! \mathbf{G}_{m}^\mathrm{alg})$ with the category of comodule objects for the Hopf algebra $K[T, T^{-1}]$. In particular $p_*1$ corresponds to the cofree comodule $K[T, T^{-1}]$. The comodule morphism $K \to K[T, T^{-1}]$ admits a retraction given by $T^{n} \mapsto \delta_{n0} T^n$. Hence, $p_*1$ is descendable. By \cite[Lemma 4.7.4]{heyer_6-functor_2024} this implies that $p$ is a universal $*$- and $!$-cover. Now we factor the map $\pi$ as $a_! \mathbf{A}^{1, \mathrm{alg}}/a_!\mathbf{G}_m^\mathrm{alg} \xrightarrow[]{\pi_1} */a_!\mathbf{G}_m^\mathrm{alg} \xrightarrow[]{\pi_2} *$. The pullback of $\pi_1$ along $p$ is the structure morphism $a_! \mathbf{A}^{1, \mathrm{alg}} \to *$. Hence, since the class $E_\mathscr{V}$ is \emph{$!$-local on the target}, the map $\pi_1$ belongs to $E_\mathscr{V}$. Since the class $E_\mathscr{V}$ is \emph{$!$-local on the source}, the morphism $\pi_2$ belongs to $E_\mathscr{V}$. Hence $\pi$ belongs to $E_\mathscr{V}$. (ii): By \cite[Lemma 4.5.7]{heyer_6-functor_2024} and the preceding discussion, the morphisms $\pi_1$ and $p$ are prim. The latter combined with \cite[Corollary 4.7.5(ii)]{heyer_6-functor_2024} implies that $\pi_2$ is prim. Hence $\pi$ is prim.  (iii):  By (ii) and the projection formula, it is sufficient to show that the canonical morphism $K \to \pi_*K[T]$ is an equivalence. The object $\pi_*K[T]$ is computed as the derived $a_!\mathbf{G}_m^\mathrm{alg}$-invariants of $K[T]$, which is indeed equivalent to $K$.
\end{proof}
\begin{defn}
Let $s: X \to Y$ be a morphism in $\operatorname{Shv}_{\tau_{\mathrm{dri}}}(\operatorname{Aff})$. 
\begin{enumerate}[(i)]
    \item \begin{enumerate}
        \item The \emph{equivariant deformation to the normal cone of $s$}, denoted $\mathscr{D}_{X/Y}^\mathscr{V}$, is defined the be the Weil restriction of $s \times \operatorname{id}: X \times */a_!\mathbf{G}_m^\mathrm{alg} \to Y \times */a_!\mathbf{G}_m^\mathrm{alg}$ along $0_Y: Y \times */\mathbf{G}_m^\mathrm{alg}\to Y \times a_!\mathbf{A}^{1, \mathrm{alg}}/a_!\mathbf{G}_m^\mathrm{alg}$.
        \item The \emph{deformation to the normal cone of $s$}, denoted $D_{X/Y}^\mathscr{V}$, is defined to be the pullback of $\mathscr{D}_{X/Y}^\mathscr{V}$ along $Y \times a_!\mathbf{A}^{1, \mathrm{alg}} \to Y \times a_!\mathbf{A}^{1,\mathrm{alg}}/a_!\mathbf{G}_m^\mathrm{alg}$. Equivalently, $D_{X/Y}^\mathscr{V}$ is the Weil restriction of $s$ along $0_Y: Y \to Y \times a_!\mathbf{A}^{1,\mathrm{alg}}$.
    \end{enumerate} 
    \item \begin{enumerate}
        \item The \emph{ equivariant normal bundle of $s$}, denoted $\mathscr{N}_{X/Y}^\mathscr{V}$, is defined to be the pullback of $\mathscr{D}_{X/Y}^\mathscr{V}$ along $0_Y: Y \times */a_!\mathbf{G}_m^\mathrm{alg}\to Y \times a_!\mathbf{A}^{1, \mathrm{alg}}/a_!\mathbf{G}_m^\mathrm{alg}$.  
        \item The \emph{normal bundle of $s$}, denoted $N_{X/Y}^\mathscr{V}$, is defined to be the pullback of $\mathscr{N}_{X/Y}^\mathscr{V}$ along $Y \to Y \times */a_!\mathbf{G}_m^\mathrm{alg}$.  
    \end{enumerate}
\end{enumerate}
\end{defn}
\begin{lem}\label{lem:deformationspace1}\cite[Theorem 4.15]{BenBassatHekkingBlowUp}
Let $s: X \to Y$ be a morphism in $\operatorname{Shv}_{\tau_\mathrm{dri}}(\operatorname{Aff})$. Then there is a $a_!\mathbf{G}_m^\mathrm{alg}$-equivariant diagram in $\operatorname{Shv}_{\tau_\mathrm{dri}}(\operatorname{Aff})$ in which all squares are Cartesian: 
\begin{equation}\label{eq:deformationdiagram1}
\begin{tikzcd}
	X & {X \times a_!\mathbf{A}^{1,\mathrm{alg}}} & {X \times a_!\mathbf{G}^{\mathrm{alg}}_m} \\
	{N_{X/Y}^\mathscr{V}} & {D_{X/Y}^\mathscr{V}} & {Y \times a_!\mathbf{G}^{\mathrm{alg}}_m} \\
	Y & {Y \times a_!\mathbf{A}^{1, \mathrm{alg}}} & {Y \times a_!\mathbf{G}^{\mathrm{alg}}_m}
	\arrow["{0_X}", from=1-1, to=1-2]
	\arrow["{s_0}"', from=1-1, to=2-1]
	\arrow["{\tilde{s}}"', from=1-2, to=2-2]
	\arrow[from=1-3, to=1-2]
	\arrow["{s \times \mathrm{id}}", from=1-3, to=2-3]
	\arrow[from=2-1, to=2-2]
	\arrow["v", from=2-1, to=3-1]
	\arrow["u"', from=2-2, to=3-2]
	\arrow[from=2-3, to=2-2]
	\arrow["{\operatorname{id}}", from=2-3, to=3-3]
	\arrow["{0_Y}", from=3-1, to=3-2]
	\arrow[from=3-3, to=3-2]
\end{tikzcd}
\end{equation}
\end{lem}
\begin{proof}
We only need to show that the bottom right square is Cartesian. The pullback of $0_Y: Y \to Y \times a_!\mathbf{A}^{1, \mathrm{alg}}$ along $Y \times a_!\mathbf{G}_m^\mathrm{alg} \to Y \times a_!\mathbf{A}^{1, \mathrm{alg}}$ is $\emptyset$; this follows from the fact that $a_!$ is limit-preserving. Hence, by the compatibility of Weil restrictions with pullback (Lemma \ref{lem:weilresbasechange}), the pullback of $D_{X/Y}^\mathscr{V}$ to $Y \times a_!\mathbf{G}_m^\mathrm{alg}$ is the Weil restriction of the empty morphism $[\emptyset \to \emptyset]$ along $\emptyset \to Y \times a_! \mathbf{G}_m^\mathrm{alg}$. But Weil restriction preserves the final object (it is a right adjoint), so this is nothing but $Y \times a_! \mathbf{G}_m^\mathrm{alg}$. 
\end{proof}
\begin{rmk}\label{rmk:Normalbundle}
\begin{enumerate}[(i)]
\item By definition, one has $D_{X/Y}^\mathscr{V} \simeq 0_{Y, \flat}X$. The morphism $\tilde{s}: X \times a_!\mathbf{A}^{1, \mathrm{alg}} \to D_{X/Y}^\mathscr{V}$ is deduced by adjunction from $\operatorname{id}: X = 0_Y^\sharp(X \times a_!\mathbf{A}^{1, \mathrm{alg}}) \to X$. In particular $\tilde{s}$ is a morphism over $Y \times a_!\mathbf{A}^{1, \mathrm{alg}}$ so that one has $s \times \operatorname{id} \simeq u \tilde{s}$.
\item We note also that $N_{X/Y}^\mathscr{V} = 0_Y^\sharp 0_{Y, \flat} X$. Hence the counit $0_Y^\sharp 0_{Y, \flat} \to \operatorname{id}$ yields a canonical morphism $p: N_{X/Y}^\mathscr{V} \to X$ over $Y$, in particular one has $v \simeq sp$. 
\end{enumerate}
\end{rmk}
\begin{thm}\label{thm:benbassatextendedRees}
\begin{enumerate}[(i)]
\item If $A \to B$ is a morphism in $\operatorname{DAlg}^{\leqslant 0}(\mathscr{V})$, corresponding to a morphism $X = \operatorname{Spec}(B) \to \operatorname{Spec}(A) =Y$ in $\operatorname{Aff}$, then $D_{X/Y}^\mathscr{V}$ is representable: one has $D_{X/Y}^\mathscr{V} \simeq \operatorname{Spec}(R_{B/A})$ where $R_{B/A} \in \operatorname{DAlg}(\mathscr{V})$ is the extended Rees algebra of \cite[Definition 4.9]{BenBassatHekkingBlowUp}.
\item With notations as in (i). If $A \to B$ is of the form $A \to A/^\mathbf{L}(f_1, \dots, f_n)$ for some $f_i: K \to A$, then there is an equivalence $R_{B/A} \simeq A[T, S_1, \dots, S_n]/^\mathbf{L}(TS_1 - f_1, \dots, TS_n -f_n)$. 
\end{enumerate}
\end{thm}
\begin{proof}
(i): This is part of \cite[Theorem 4.6]{BenBassatHekkingBlowUp}. (ii): The proof is identical to \cite[Example 4.12]{BenBassatHekkingBlowUp}. 
\end{proof}
\begin{lem}\label{lem:IdentifyNormal}
With notations as in Lemma \ref{lem:deformationspace1}. Assume that $s: X \to Y$ admits a cotangent complex in the sense of Definition \ref{defn:CotangentComplex}(i). Then there is a canonical equivalence $N_{X/Y}^\mathscr{V} \simeq V_X^\mathscr{V}(\mathbf{L}_{X/Y}^\mathscr{V}[-1])$ in $\operatorname{Shv}_{\tau_{\mathrm{dri}}}(\operatorname{Aff})$. Under this equivalence, the morphism $s_0: X \to N_{X/Y}^\mathscr{V}$ identifies with the zero-section $X \to V_X^\mathscr{V}(\mathbf{L}_{X/Y}^\mathscr{V}[-1])$, and the morphism $p$ of Remark \ref{rmk:Normalbundle}(ii) identifies with the canonical projection $V_X^\mathscr{V}(\mathbf{L}_{X/Y}^\mathscr{V}[-1]) \to X$. 
\end{lem}
\begin{proof}
The proof is essentially the same as in \cite[Proposition 3.16]{hekking_khan_deformation_2025}. We reproduce the argument here for the reader's convenience. We recall that $N_{X/Y}^\mathscr{V} = 0_Y^\sharp 0_{Y, \flat} X$, c.f. Remark \ref{rmk:Normalbundle}, so that by adjunction for any $S \in \operatorname{Shv}_{\tau_\mathrm{dri}}(\operatorname{Aff})_{/Y}$ one has $\operatorname{Map}_{/Y}(S, N_{X/Y}^\mathscr{V}) \simeq \operatorname{Map}_{/Y}(0_Y^\sharp 0_{Y, \sharp} S, X)$ compatibly with the natural maps to $\operatorname{Map}_{/Y}(S,X)$. Now 
\begin{equation}
0_Y^\sharp 0_{Y, \sharp} S = S \times (\{0\} \times_{a_!\mathbf{A}^{1,\mathrm{alg}}} \{0\}) = S \times \operatorname{Spec}(K \oplus K[1]) = S[1_S[1]]
\end{equation}
is the trivial square-zero extension of $S$ by $1_S[1]$. If $[t:T \to X] \in \operatorname{Shv}_{\tau_\mathrm{dri}}(\operatorname{Aff})_{/X}$ there is a chain of equivalences
\begin{equation}\label{eq:normalbundle2}
\begin{aligned}
    \operatorname{Map}_{/X}(T,V_X^\mathscr{V}(\mathbf{L}^\mathscr{V}_{X/Y}[-1])) &\simeq \operatorname{Map}_{\operatorname{QCoh}(T)}(t^*\mathbf{L}_{X/Y}^\mathscr{V}, 1_T[1])\\
    &\simeq \operatorname{Map}_{/Y}(T[1_T[1]], X)_t \\
    &\simeq \operatorname{Map}_{/Y}(T, N_{X/Y}^\mathscr{V})_t \\
    &\simeq  \operatorname{Map}_{/X}(T, N_{X/Y}^\mathscr{V}).
\end{aligned}
\end{equation}
Here the first equivalence is by Definition \ref{defn:vectorbundles}, the second is by Definition \ref{defn:CotangentComplex}, the third is by the above discussion, and the last is tautological. Hence there is an equivalence $N_{X/Y}^\mathscr{V} \simeq V_X^\mathscr{V}(\mathbf{L}_{X/Y}^\mathscr{V}[-1])$ over $X$. The subscript $(-)_t$ means that we take the fiber over $t: T \to X$ under the natural map to $\operatorname{Map}_{/Y}(T,X)$. For the last statement, we note that under the chain of equivalences \eqref{eq:normalbundle2} with $t = \operatorname{id}: X \to X$, the canonical map $X \to N_{X/Y}^\mathscr{V}$ corresponds to the zero derivation $X[1_X[1]] \to X$, which corresponds to the zero morphism $\mathbf{L}_{X/Y}^\mathscr{V} \to 1_X[1]$, which gives the zero-section $X \to V_X^\mathscr{V}(\mathbf{L}_{X/Y}^\mathscr{V}[-1])$. 
\end{proof}
\egapar We denote by $\breve{s}: X \times a_! \mathbf{A}^{1,\mathrm{alg}}/a_!\mathbf{G}_m^\mathrm{alg} \to \mathscr{D}_{X/Y}^\mathscr{V}$ the morphism obtained from $\tilde{s}$ after ``quotienting by the action of $a_!\mathbf{G}_m^\mathrm{alg}$", so that there is a Cartesian square:
\begin{equation}\label{eq:CartesianDeformation}
\begin{tikzcd}
	{X \times a_!\mathbf{A}^{1, \mathrm{alg}}} & {D_{X/Y}^\mathscr{V}} \\
	{X \times a_!\mathbf{A}^{1, \mathrm{alg}}/a_!\mathbf{G}_m^\mathrm{alg}} & {\mathscr{D}_{X/Y}^\mathscr{V}}
	\arrow["{\tilde{s}}", from=1-1, to=1-2]
	\arrow[from=1-1, to=2-1]
	\arrow["\lrcorner"{anchor=center, pos=0.125}, draw=none, from=1-1, to=2-2]
	\arrow[from=1-2, to=2-2]
	\arrow["{\breve{s}}", from=2-1, to=2-2]
\end{tikzcd}
\end{equation}
\egapar For every $n \in \mathbf{Z}$ we may view the $a_!\mathbf{G}_m^\mathrm{alg}$-equivariant $K[T]$-module $T^{-n}K[T]$ as an (invertible) object  $1(n) \in \operatorname{QCoh}(a_!\mathbf{A}^{1,\mathrm{alg}}/a_!\mathbf{G}_m^\mathrm{alg})$. For $X \in \operatorname{Shv}_{\tau_{\mathrm{dri}}}(\operatorname{Aff})$ we use the shorthand $M \mapsto M(n)$ to denote the endofunctor of $\operatorname{QCoh}(X \times a_!\mathbf{A}^{1,\mathrm{alg}}/a_!\mathbf{G}_m^\mathrm{alg})$ given by tensoring with $\pi_2^*1(n)$. Here $\pi_2: X \times a_!\mathbf{A}^{1,\mathrm{alg}}/a_!\mathbf{G}_m^\mathrm{alg} \to a_!\mathbf{A}^{1,\mathrm{alg}}/a_!\mathbf{G}_m^\mathrm{alg}$ is the projection. 
\begin{prop}\label{prop:deformationofdualizing}
Let $s: X \to Y$ be a morphism in $\operatorname{Shv}_{\tau_{\mathrm{dri}}}(\operatorname{Aff})$. Suppose that there exists a cover $\{Y_i = \operatorname{Spec}(A_i) \to Y\}_i$ of $Y$ by objects of $\operatorname{Aff}$ such that:
\begin{itemize}
    \item[$\star$] Each pullback $X_i := Y_i \times_Y X$ also belongs to $\operatorname{Aff}$. Let us say that $X_i = \operatorname{Spec}(B_i)$.
    \item[$\star$] Each $A_i$ and $B_i$ belongs to $\operatorname{DAlg}^{\leqslant 0}(\mathscr{V})$ and the induced morphism $A_i \to B_i$ is a regular surjection of virtual dimension $-n$, in the sense of Definition \ref{defn:regularsurjection}. 
\end{itemize}
Then: 
\begin{enumerate}[(i)]
    \item The morphism $\breve{s}: X \times a_! \mathbf{A}^{1,\mathrm{alg}}/a_!\mathbf{G}_m^\mathrm{alg} \to \mathscr{D}_{X/Y}^\mathscr{V}$ is representable, and further is \emph{suave} with respect to the six-functor formalism \eqref{eq:Big6FF}. In particular the formation of $\breve{s}^!1$ commutes with pullback.
    \item Let us denote by $\pi_X$ the canonical morphism $X \times a_! \mathbf{A}^{1,\mathrm{alg}}/a_!\mathbf{G}_m^\mathrm{alg} \to X$. Then, the canonical morphism $\pi_X^*\pi_{X,*}(\breve{s}^!1 (n))\to \breve{s}^!1(n)$ is an equivalence. 
    \item With $s_0$ as in Lemma \ref{lem:deformationspace1}: There is a canonical equivalence $s_0^!1 \simeq s^!1$.  
\end{enumerate}
\end{prop}
\begin{proof}
(i): By Proposition \ref{prop:NCrepresentable} and \cite[Lemma 4.5.7]{heyer_6-functor_2024}, we reduce to proving the analogous claim for the ``unquotiented" morphism $\tilde{s}: X \times a_! \mathbf{A}^{1,\mathrm{alg}} \to D_{X/Y}^\mathscr{V}$. Further, by \emph{loc. cit.} the question is local on the target. Hence by the compatibility of Weil restrictions with base-change, we may assume that $Y= \operatorname{Spec}(A)$, $X = \operatorname{Spec}(B)$ are such that $A, B \in \operatorname{DAlg}^{\leqslant 0}(\mathscr{V})$ and $A \to B$ is of the form $A \to A/^\mathbf{L}(f_1, \dots,f_n)$ for some $f_i : K \to A$. In this case, by Theorem \ref{thm:benbassatextendedRees}, $D_{X/Y}^\mathscr{V} = \operatorname{Spec}(R_{B/A})$, where $R_{B/A}$ can be computed by the formula $R_{B/A} = A[T, S_i]/^\mathbf{L}(TS_i-f_i)$. In particular, note that $R_{B/A}/^\mathbf{L}(S_i) \simeq B[T]$ is dualizable as an $R_{B/A}$-module, hence by \cite[Lemma 4.5.10]{heyer_6-functor_2024} the morphism $\tilde{s}$ is suave. (ii): Using Lemma \ref{lem:A1Gm} and prim base-change \cite[Lemma 4.5.13]{heyer_6-functor_2024} together with part (i) and the compatibility of Weil restrictions with base-change, we see that the question is local on $Y$. Hence we may again assume that assume that $Y= \operatorname{Spec}(A)$, $X = \operatorname{Spec}(B)$ are such that $A, B \in \operatorname{DAlg}^{\leqslant 0}(\mathscr{V})$ and $A \to B$ is of the form $A \to A/^\mathbf{L}(f_1, \dots,f_n)$ for some $f_i : K \to A$. By suave base-change, the underlying $K[T]$-module of $\breve{s}^!1$ is $\tilde{s}^!1$, and using notations as before, the Koszul complex yields a $a_!\mathbf{G}_m^\mathrm{alg}$-equivariant equivalence 
\begin{equation}
\tilde{s}^!1 \simeq R\underline{\operatorname{Hom}}_{R_{B/A}} (B[T], R_{B/A}) \simeq T^{n}B[T][-n],
\end{equation}
hence $\breve{s}^!1(n) \simeq 1[-n]$ is in the essential image of $\pi_X^*$, so that $\pi_X^*\pi_{X,*}(\breve{s}^!1 (n))\to \breve{s}^!1(n)$ is an equivalence by Lemma \ref{lem:A1Gm}(iii).  (iii): Let $\varrho_X: X \times a_!\mathbf{A}^{1,\mathrm{alg}} \to X$ denote the projection on to the first factor. By (i) and (ii) and using the Cartesian square \eqref{eq:CartesianDeformation}, the canonical morphism $\varrho^*_X \pi_{X,*}(\breve{s}^!1(n)) \to \tilde{s}^!1$ is an equivalence. Hence $\tilde{s}^!1$ is in the essential image of the functor $\varrho_X^*$, so that the fibers of $\tilde{s}^!1$ over the zero-section and the unit section are canonically equivalent. But by (i) the formation of $\tilde{s}^!1$ is compatible with pullback, which gives the canonical equivalence $s_0^!1 \simeq s^!1$.
\end{proof}
\egapar In the following Definition we recall that $\operatorname{AniAlg}$ refers to the category of abstract animated $K$-algebras, and $D(K)$ refers to the derived category of abstract $K$-modules. 
\begin{defn}\label{defn:VectorBundleStack}
Let $n \geqslant 1$. 
\begin{enumerate}[(i)]
\item We define the functor
\begin{equation}
\operatorname{GL}_n^\mathscr{V}: \operatorname{Aff}^\mathrm{op} \to \infty\mathrm{Grpd}: \operatorname{Spec}(A) \mapsto \operatorname{Aut}_{\operatorname{Mod}_A\mathscr{V}}(A^{\oplus n}).
\end{equation} 
\item We define the functor 
\begin{equation}
    \operatorname{GL}_n^\mathrm{alg}: \operatorname{AniAlg} \to \infty\mathrm{Grpd}: A \mapsto \operatorname{Aut}_{\operatorname{Mod}_AD(K)}(A^{\oplus n}).  
\end{equation}
\item We define $\operatorname{Vect}^\mathscr{V}_n$ to be the functor from $\operatorname{Aff}^\mathrm{op}$ to $\infty\mathrm{Grpd}$ which sends $\operatorname{Spec}(A)$ to the $\infty$-groupoid of $\tau_\mathrm{dri}$-locally free objects $M \in \operatorname{QCoh}(\operatorname{Spec}(A))$, of rank $n$. 
\item We define $\operatorname{Vect}^\mathrm{alg}_n$ to be the functor from $\operatorname{AniAlg}$ to $\infty\mathrm{Grpd}$ which sends $A$ to the $\infty$-groupoid of Zariski-locally free $A$-modules of rank $n$. 
\end{enumerate}
\end{defn}
\begin{prop}\label{prop:VectorBundlesGL_n}
With notations as in Definition \ref{defn:VectorBundleStack}. 
\begin{enumerate}[(i)]
    \item The functors $\operatorname{GL}_n^\mathscr{V}$ and $\operatorname{GL}_n^\mathrm{alg}$ are representable. In particular $\operatorname{GL}_n^\mathscr{V}$ (resp. $\operatorname{GL}_n^\mathrm{alg}$) is a sheaf in the topology $\tau_\mathrm{dri}$ (resp. the Zariski topology). 
    \item When $\operatorname{GL}_n^\mathscr{V}$ is regarded as an object of $\operatorname{Shv}_{\tau_{\mathrm{dri}}}(\operatorname{Aff})$ and $\operatorname{GL}_n^\mathrm{alg}$ is regarded as an object of $\operatorname{Shv}_{\tau_{\mathrm{Zar}}}(\operatorname{AniAlg}^\mathrm{op})$, there is a canonical equivalence $a_!\mathrm{GL}_n^\mathrm{alg} \simeq \mathrm{GL}_n^\mathscr{V}$. 
    \item There is a canonical equivalence $\operatorname{Vect}^\mathscr{V}_n \simeq */\operatorname{GL}_n^\mathscr{V}$ in $\operatorname{Shv}_{\tau_{\mathrm{dri}}}(\operatorname{Aff})$ and a canonical equivalence $\operatorname{Vect}^\mathrm{alg}_n \simeq */\operatorname{GL}_n^\mathrm{alg}$ in $\operatorname{Shv}_{\tau_\mathrm{Zar}}(\operatorname{AniAlg}^\mathrm{op})$. 
\end{enumerate}
\end{prop}
\begin{proof}
(i): We consider the functors
\begin{equation}
\begin{aligned}
   \operatorname{Mat}_n^\mathscr{V}&: \operatorname{Aff}^\mathrm{op} \to \infty\mathrm{Grpd}: \operatorname{Spec}(A) \mapsto \operatorname{End}_{\operatorname{Mod}_A\mathscr{V}}(A^{\oplus n}), \\  
   \operatorname{Mat}_n^\mathrm{alg}&: \operatorname{AniAlg} \to \infty\mathrm{Grpd}: A \mapsto \operatorname{End}_{\operatorname{Mod}_AD(K)}(A^{\oplus n}).
\end{aligned}
\end{equation}
Since finite direct sums are the same as direct products, we see that $\operatorname{Mat}_n^\mathscr{V}$ (resp. $\operatorname{Mat}_n^\mathrm{alg}$)  is represented by $a_!\mathbf{A}^{n^2, \mathrm{alg}}$ (resp. $\mathbf{A}^{n^2, \mathrm{alg}}$). Further, there are Cartesian squares
\begin{equation}
\begin{aligned}
\begin{tikzcd}
	{\operatorname{GL}_n^\mathscr{V}} & {\operatorname{Mat}_n^\mathscr{V} \times \operatorname{Mat}_n^\mathscr{V}  } \\
	{\{*\}} & {\operatorname{Mat}_n^\mathscr{V}}
	\arrow[from=1-1, to=1-2]
	\arrow[from=1-1, to=2-1]
	\arrow["\lrcorner"{anchor=center, pos=0.125}, draw=none, from=1-1, to=2-2]
	\arrow["m", from=1-2, to=2-2]
	\arrow[from=2-1, to=2-2]
\end{tikzcd} && \text {and} &&
\begin{tikzcd}
	{\operatorname{GL}_n^\mathrm{alg}} & {\operatorname{Mat}_n^\mathrm{alg} \times \operatorname{Mat}_n^\mathrm{alg}  } \\
	{\{*\}} & {\operatorname{Mat}_n^\mathrm{alg}}
	\arrow[from=1-1, to=1-2]
	\arrow[from=1-1, to=2-1]
	\arrow["\lrcorner"{anchor=center, pos=0.125}, draw=none, from=1-1, to=2-2]
	\arrow["m", from=1-2, to=2-2]
	\arrow[from=2-1, to=2-2]
\end{tikzcd}
\end{aligned}
\end{equation}
where the bottom arrow is the identity section and the right arrow is the multiplication morphism. Therefore $\operatorname{GL}_n^\mathscr{V}$ and $\operatorname{GL}_n^\mathrm{alg}$ are representable and there is a canonical equivalence $a_!\operatorname{GL}_n^\mathrm{alg} \simeq \operatorname{GL}_n^\mathscr{V}$. (iii): Let us treat the case of $\operatorname{Vect}^\mathscr{V}_n$, as the case of $\operatorname{Vect}_n^\mathrm{alg}$ is similar and perhaps even well-known. 
By definition, $*/\mathrm{GL}^\mathscr{V}_n$ is the sheaf which sends $\operatorname{Spec}(A)$ to the $\infty$-groupoid of $\mathrm{GL}_n^\mathscr{V}$-torsors (in the topology $\tau_\mathrm{dri}$) over $\operatorname{Spec}(A)$. The canonical morphism $\operatorname{Vect}_n^\mathscr{V} \to */ \mathrm{GL}_n^\mathscr{V}$ is therefore given on $\operatorname{Spec}(A)$-points by $M \mapsto \underline{\operatorname{Isom}}(A^{\oplus n},M)$. Here $\underline{\operatorname{Isom}}(A^{\oplus n},M)$ is the $\tau_\mathrm{dri}$-sheafification of the presheaf on $\operatorname{Aff}_{/\operatorname{Spec}(A)}$ which sends $\operatorname{Spec}(B) \to \operatorname{Spec}(A)$ to $\operatorname{Isom}(B^{\oplus n}, B \hat{\otimes}^\mathbf{L}_AM)$. Conversely, by $*$-descent the standard representation of $\operatorname{GL}_n^\mathscr{V}$ on $K^{\oplus n}$ determines a locally free rank $n$ object $\operatorname{Std} \in \operatorname{QCoh}(*/ \mathrm{GL}_n^\mathscr{V})$, so that any $f: \operatorname{Spec}(A) \to */\mathrm{GL}_n^\mathscr{V}$ determines a locally free rank $n$ object $M =f^*\operatorname{Std} \in \operatorname{QCoh}(\operatorname{Spec}(A))$. These are mutually inverse and we obtain the canonical equivalence $\operatorname{Vect}_n^\mathscr{V} \simeq */ \mathrm{GL}_n^\mathscr{V}$. 
\end{proof}
\begin{cor}
There is a canonical equivalence $a_!\operatorname{Vect}_n^\mathrm{alg} \simeq \operatorname{Vect}^\mathscr{V}_n$. In particular $\operatorname{Vect}^\mathscr{V}_n$ is left Kan extended from $\operatorname{AniAlg}^\mathrm{op}$. 
\end{cor}
\begin{proof}
Immediate from Proposition \ref{prop:VectorBundlesGL_n}.
\end{proof}
\egapar Let $\mathscr{X}$ be an $\infty$-topos and let $G, H$ be group objects in $\mathscr{X}$. By the enhancement of Giraud's axiom on effective quotients \cite[Theorem 2.15]{NikolausPrincipal}, there is a canonical equivalence $\operatorname{Map}_*(*/G,*/H) \simeq \operatorname{Hom}(G,H)$. Specializing to the case when $\mathscr{X} = \operatorname{Shv}_{\tau_{\mathrm{dri}}}(\operatorname{Aff})$, $G = \operatorname{GL}_n^\mathscr{V}$ and $H = \operatorname{GL}_1^\mathscr{V}$, there is a canonical equivalence $\operatorname{Hom}(\operatorname{GL}_n^\mathscr{V},\operatorname{GL}_1^\mathscr{V}) \simeq \mathbf{Z}$ furnished by the determinant. Indeed, this follows from the corresponding statement about $\operatorname{GL}_n^\mathrm{alg}$ together with Proposition \ref{prop:VectorBundlesGL_n} and the fact that the functor $a_!$ is fully-faithful when restricted to representable objects. We conclude that $\mathbf{Z} \xrightarrow[]{\sim}\operatorname{Map}_*(*/\operatorname{GL}_n^\mathscr{V}, */\operatorname{GL}_1^\mathscr{V})$. From this and Proposition \ref{prop:VectorBundlesGL_n} it follows that there is a canonical equivalence
\begin{equation}\label{eq:MapVect}
    \mathbf{Z} \xrightarrow[]{\sim}\operatorname{Map}_*(\operatorname{Vect}^\mathscr{V}_n, \operatorname{Vect}_1^\mathscr{V}).
\end{equation}
Concretely, the space of pointed maps on the right identifies with those morphisms such that the induced morphism $\operatorname{Vect}^\mathscr{V}_n(*) \to \operatorname{Vect}^\mathscr{V}_1(*)$ carries $K^{\oplus n}$ to $K$. We obtain for every $k \in \mathbf{Z}$ a point $\operatorname{det}^k$ in the space \eqref{eq:MapVect}. When $k=1$ we abbreviate this to $\operatorname{det}$. In this way, for every $X \in \operatorname{Shv}_{\tau_{\mathrm{dri}}}(\operatorname{Aff})$ we obtain a canonical morphism $\operatorname{det}^k_X: \operatorname{Vect}^\mathscr{V}_n(X) \to  \operatorname{Vect}_1^\mathscr{V}(X)$.
\egapar For every $m, n \geqslant 0$, there is a canonical morphism of group objects $\operatorname{GL}_m^\mathscr{V} \times \operatorname{GL}_n^\mathscr{V} \to \operatorname{GL}_{m+n}^\mathscr{V}$ given by block sum. Since $\operatorname{GL}_1^\mathscr{V}$ is abelian, the multiplication $\operatorname{GL}_1^\mathscr{V} \times \operatorname{GL}_1^\mathscr{V} \to \operatorname{GL}_1^\mathscr{V}$ is a morphism of group objects. There is a commutative square of group objects 
\begin{equation}
\begin{tikzcd}
	{\operatorname{GL}_m^\mathscr{V} \times \operatorname{GL}_n^\mathscr{V}} & {\operatorname{GL}_1^\mathscr{V} \times \operatorname{GL}_1^\mathscr{V}} \\
	{\operatorname{GL}_{m+n}^\mathscr{V}} & {\operatorname{GL}_1^\mathscr{V}}
	\arrow["{\operatorname{det} \times \operatorname{det}}", from=1-1, to=1-2]
	\arrow[from=1-1, to=2-1]
	\arrow[from=1-2, to=2-2]
	\arrow["{\operatorname{det}}", from=2-1, to=2-2]
\end{tikzcd}
\end{equation}
hence, using that geometric realization commutes with products, we obtain a commutative square 
\begin{equation}
\begin{tikzcd}
	{*/\operatorname{GL}_m^\mathscr{V} \times */\operatorname{GL}_n^\mathscr{V}} & {*/\operatorname{GL}_1^\mathscr{V} \times */\operatorname{GL}_1^\mathscr{V}} \\
	{*/\operatorname{GL}_{m+n}^\mathscr{V}} & {*/\operatorname{GL}_1^\mathscr{V}}
	\arrow["{\operatorname{det} \times \operatorname{det}}", from=1-1, to=1-2]
	\arrow[from=1-1, to=2-1]
	\arrow[from=1-2, to=2-2]
	\arrow["{\operatorname{det}}", from=2-1, to=2-2]
\end{tikzcd}
\end{equation}
in $\operatorname{Shv}_{\tau_\mathrm{dri}}(\operatorname{Aff})$. Transporting this through the equivalence provided by Proposition \ref{prop:VectorBundlesGL_n}, we see that for each $X \in \operatorname{Shv}_{\tau_\mathrm{dri}}(\operatorname{Aff})$, the functor $\operatorname{det}_X: \operatorname{Vect}^\mathscr{V}_n(X) \to \operatorname{Vect}_1^\mathscr{V}(X)$ takes direct sums to tensor products.
\egapar\label{par:detdual} The inverse gives a canonical anti-automorphism of group objects $(-)^{-1}: \operatorname{GL}_n^\mathscr{V} \to \operatorname{GL}_n^{\mathscr{V},\mathrm{op}}$. The determinant morphism commutes with this anti-automorphism. Hence $\operatorname{det}: */\mathrm{GL}_n^\mathscr{V} \to */\mathrm{GL}_1^\mathscr{V}$ commutes with the induced morphism $(-)^{-1} : */\operatorname{GL}_n^\mathscr{V} \to */\operatorname{GL}_n^{\mathscr{V},\mathrm{op}}$. Transporting this through the equivalence provided by Proposition \ref{prop:VectorBundlesGL_n}, we see that for each $X \in \operatorname{Shv}_{\tau_\mathrm{dri}}(\operatorname{Aff})$, the functor $\operatorname{det}_X: \operatorname{Vect}^\mathscr{V}_n(X) \to \operatorname{Vect}_1^\mathscr{V}(X)$ commutes with duality. We recall that for $\operatorname{Vect}_1^\mathscr{V}(X)$ the dual coincides with the $\otimes$-inverse. 
\egapar The proof of the following Lemma is adapted from \cite[Proposition 7.7.2]{camargo_notes_2026}.
\begin{lem}\label{lem:DualizingOfVectorBundle}
Let $X \in \operatorname{Shv}_{\tau_{\mathrm{dri}}}(\operatorname{Aff})$ and let $E \in \operatorname{Vect}_n^\mathscr{V}(X)$ be locally free of rank $n$ with dual $E^\lor$. Let $s_0: X \to V_X^\mathscr{V}(E^\lor)$ be the zero-section. Then:
\begin{enumerate}[(i)]
    \item $s_0$ is suave with respect to the six-functor formalism \eqref{eq:Big6FF}, and $s_0^!1[n]$ is locally free of rank $1$. 
    \item When $X = *$ and $E = K^{\oplus n}$, the object $s_0^!1[n]$ is canonically equivalent to $K$.
    \item The assignment $E \mapsto s_0^!1[n]$ takes direct sums to tensor products. 
    \item There is a canonical equivalence $s_0^!1[n] \simeq \operatorname{det}_XE$ compatible with pullbacks.
\end{enumerate}  
\end{lem} 
\begin{proof}
(i): By Lemma \ref{lem:RelSpecRepresent}, the morphism $s_0$ is representable as it is a section of the representable morphism $p$. In particular $s_0$ belongs to the class $E_\mathscr{V}$ of $!$-able morphisms. By \cite[Lemma 4.5.7]{heyer_6-functor_2024}, we reduce to the case $X = \operatorname{Spec}(A)$ is affine and $E \simeq A^{\oplus n}$ is free of rank $n$. In this case $s_0$ is equivalent to the zero-section $X \to X \times a_!\mathbf{A}^{n, \mathrm{alg}}$, and one computes directly using the Koszul complex that $s_0^!1 \simeq A[-n]$. (ii): We note that, in this case that Koszul complex yields a \emph{canonical} equivalence $s_0^!1 \simeq K[-n]$. (iii): First, we note that, if $E \in \operatorname{Vect}_n^\mathscr{V}(X)$ and $F \in \operatorname{Vect}_m^\mathscr{V}(X)$ then by Lemma \ref{lem:RelativespecLemma1}(vi) there is a canonical equivalence $V_X^\mathscr{V}((E \oplus F)^\lor) \simeq V_X^\mathscr{V}(E^\lor) \times_X V_X^\mathscr{V}(F^\lor)$. We denote by $t_0:X \to V_X^\mathscr{V}(F^\lor)$ the zero-section. Hence using the associativity of fiber products we see that the square 
\begin{equation}
\begin{tikzcd}
	X & {V_X^\mathscr{V}(E^\lor)} \\
	{V_X^\mathscr{V}(F^\lor)} & {V_X^\mathscr{V}((E \oplus F)^\lor)}
	\arrow["{s_0}", from=1-1, to=1-2]
	\arrow["{t_0}"', from=1-1, to=2-1]
	\arrow["\lrcorner"{anchor=center, pos=0.125}, draw=none, from=1-1, to=2-2]
	\arrow["{t_0^\prime}", from=1-2, to=2-2]
	\arrow["{s_0^\prime}", from=2-1, to=2-2]
\end{tikzcd}
\end{equation}
is Cartesian. All maps here are suave, hence by suaveness and suave-base change, we obtain a chain of equivalences $(t_0^\prime s_0)^! 1 \simeq s_0^! t_0^{\prime,!}1 \simeq s_0^!1 \hat{\otimes}_X s_0^* t_0^{\prime,!}1 \simeq s_0^!1 \hat{\otimes}_X t_0^!1$, proving the claim. 
(iv): By (i) the assignment $E \mapsto s_0^!1[n]$ determines for each $X = \operatorname{Spec}(A) \in \operatorname{Aff}$ a morphism $\operatorname{Vect}_n^\mathscr{V}(X) \to \operatorname{Vect}_1^\mathscr{V}(X)$ and hence by Yoneda, a  morphism $\delta:\operatorname{Vect}_n^\mathscr{V} \to \operatorname{Vect}_1^\mathscr{V}$, which is even a pointed morphism, by (ii). By the preceding discussion we know that $\delta \simeq \operatorname{det}^k$ for some $k \in \mathbf{Z}$, and it remains to show that $k = 1$. By (iii), this can be checked after restriction along the morphism $\coprod_{i=1}^n \operatorname{Vect}_1^\mathscr{V} \to \operatorname{Vect}_n^\mathscr{V}$ induced by direct sum. In this way we reduce to the case when $n=1$ and we need to show that $\delta$ is equivalent to the identity. Now we use the equivalence $\operatorname{Vect}_1^\mathscr{V} \simeq */\operatorname{GL}_1^\mathscr{V}$ of Proposition \ref{prop:VectorBundlesGL_n}. The universal locally free rank $1$ module over\footnote{We recall that $\mathbf{G}_m^\mathrm{alg} = \operatorname{Spec}(K[T, T^{-1}])$.} $X = */\mathrm{GL}_1^\mathscr{V} = */a_!\mathbf{G}_m^\mathrm{alg}$ is $E = K \cdot T^{-1}$. Hence in this case one has $V_X^\mathscr{V}(E^\lor) = a_!\mathbf{A}^{1, \mathrm{alg}}/a_!\mathbf{G}_m^\mathrm{alg}$. Considering the zero-section $s_0: */a_!\mathbf{G}_m^\mathrm{alg} \to a_!\mathbf{A}^{1, \mathrm{alg}}/a_!\mathbf{G}_m^\mathrm{alg}$, the Koszul complex gives a $a_!\mathbf{G}_m^\mathrm{alg}$-equivariant equivalence $s_0^!1 = R\underline{\operatorname{Hom}}_{K[T]}(K, K[T]) \simeq K\cdot T^{-1}[-1]$. This shows that $s_0^!1[1]$ is the universal locally free rank $1$ module over $*/a_!\mathbf{G}_m^\mathrm{alg}$ and hence $\delta \simeq \operatorname{id}$.
\end{proof}

\begin{thm}\label{thm:MainThmFinal}
Let $f: X \to Y$ be a smooth and Kiehl partially-proper morphism of derived rigid spaces, of dimension $d$. Then $f$ is suave with respect to the six-functor formalism \eqref{eq:6ff1}, and there is a canonical equivalence $f^!1_Y \simeq (\operatorname{det}_X\mathbf{L}_{X/Y}) [d]$ compatible with pullbacks.
\end{thm}
\begin{proof}
By Theorem \ref{thm:GD1}, it remains to identify the object $\Delta_f^! 1$, which is inverse to $f^!1$. here $\Delta_f: X \to X \times_Y X$ is the diagonal of $f$. By Lemma \ref{lem:SmoothSection} the morphism $\Delta_f$ is quasi-smooth of virtual dimension $-d$, hence by Lemma \ref{lem:Zariskiclosedquasismooth} it is locally regular of virtual dimension $-d$. Hence we may apply Proposition \ref{prop:deformationofdualizing} to the morphism $i_!(\Delta_f)$; this uses that the functor $i_!$ preserves covers, pullbacks and representable objects, c.f. Proposition \ref{prop:analyticinclusion}. We deduce that there is a canonical equivalence $\Delta_f^!1 \simeq s_0^!1$, where $s_0: i_!X \to N_{i_!X/i_!X \times_Y X}$ is the zero-section of the normal bundle. By Lemma \ref{lem:IdentifyNormal}, the latter identifies with the zero-section $s_0: i_!X \to V_{i_!X}^\mathscr{V}(\mathbf{L}_{i_!X/i_!X \times_YX}^\mathscr{V}[-1])$. Hence using Lemma \ref{lem:DualizingOfVectorBundle} and Lemma \ref{lem:analytificationofcotangent} together with the usual fiber sequence on cotangent complexes, we deduce that there is a canonical equivalence $s_0^!1 \simeq (\operatorname{det}_X \mathbf{L}_{X/Y}^\lor)[-d]$. Hence $\Delta_f^!1 \simeq (\operatorname{det}_X \mathbf{L}_{X/Y}^\lor)[-d]$. Finally, by the discussion of Paragraph \ref{par:detdual}, the functor $\operatorname{det}_X$ takes duals to inverses. Hence $(\Delta_f^!1)^{\otimes -1} \simeq (\operatorname{det}_X \mathbf{L}_{X/Y})[d]$. 
\end{proof}

\section{Bonus}\label{sec:bonus}
\egapar It is natural to ask if the deformation to the normal cone can also be implemented as a derived rigid space. In this section we construct such a space and investigate its geometry. This can be regarded as an ``analytified" counterpart to the results of \cite{BenBassatHekkingBlowUp}. It gives a possibly more rigorous way to do the gluing in \cite[Remark 4.2.5]{zavyalov_poincare_2023} and also upgrades some of the results of \emph{loc. cit.} to the derived setting. We emphasize that that results of this section are not necessary for the proof of the Main Theorem \ref{thm:MainThmFinal}. 
\begin{defn}
Let $s: X \to Y$ be a morphism in $\operatorname{Shv}_{\tau_\mathrm{rat}}(\operatorname{Afnd})$. 
\begin{enumerate}[(i)]
\item \begin{enumerate}
    \item We define the \emph{analytic equivariant deformation to the normal cone of $s$}, denoted $\mathscr{D}_{X/Y}$, to be the Weil restriction of $s \times \operatorname{id} : X \times */\mathbf{G}_m^\mathrm{an} \to Y \times */\mathbf{G}_m^\mathrm{an}$ along $0_Y: Y \times */\mathbf{G}_m^\mathrm{an} \to Y \times \mathbf{A}^{1, \mathrm{an}}/\mathbf{G}_m^\mathrm{an}$.
    \item We define the \emph{analytic deformation to the normal cone of $s$}, denoted $D_{X/Y}$, to be the pullback of $\mathscr{D}_{X/Y}$ along $Y \times \mathbf{A}^{1, \mathrm{an}} \to Y \times \mathbf{A}^{1, \mathrm{an}}/\mathbf{G}_m^\mathrm{an}$. Equivalently, $D_{X/Y}$ is the Weil restriction of $s$ along $0_Y: Y \to Y \times \mathbf{A}^{1, \mathrm{an}}$. 
\end{enumerate}
\item \begin{enumerate}
    \item The \emph{analytic equivariant normal bundle of $s$}, denoted $\mathscr{N}_{X/Y}$, is defined to be the pullback of $\mathscr{D}_{X/Y}$ along $0_Y: Y \times */\mathbf{G}_m^\mathrm{an} \to Y \times \mathbf{A}^{1, \mathrm{an}}/\mathbf{G}_m^\mathrm{an}$.
    \item The \emph{analytic normal bundle of $s$}, denoted $N_{X/Y}$, is defined to be the pullback of $\mathscr{N}_{X/Y}$ along $Y \to Y \times */\mathbf{G}_m^\mathrm{an}$. 
\end{enumerate}
\end{enumerate}
\end{defn}

\begin{lem}\label{lem:deformationAnalytic}
Let $s: X \to Y$ be a morphism in $\operatorname{Shv}_{\tau_{\mathrm{rat}}}(\operatorname{Afnd})$. Then:
\begin{enumerate}[(i)]
    \item There is a canonical equivalence $D_{X/Y} \simeq i^*D_{i_!X/i_!Y}^\mathscr{V}$ (and hence a canonical equivalence $N_{X/Y} \simeq i^*N^\mathscr{V}_{i_!X/i_!Y}$); 
    \item There is a $\mathbf{G}_m^\mathrm{an}$-equivariant diagram in $\operatorname{Shv}_{\tau_{\mathrm{rat}}}(\operatorname{Afnd})$ in which all squares are Cartesian: 
    \begin{equation}
\begin{tikzcd}
	X & {X \times \mathbf{A}^{1,\mathrm{an}}} & {X \times \mathbf{G}^{\mathrm{an}}_m} \\
	{N_{X/Y}} & {D_{X/Y}} & {Y \times \mathbf{G}^{\mathrm{an}}_m} \\
	Y & {Y \times \mathbf{A}^{1, \mathrm{an}}} & {Y \times \mathbf{G}^{\mathrm{an}}_m}
	\arrow["{0_X}", from=1-1, to=1-2]
	\arrow["{s_0}"', from=1-1, to=2-1]
	\arrow["{\tilde{s}}"', from=1-2, to=2-2]
	\arrow[from=1-3, to=1-2]
	\arrow["{s \times \mathrm{id}}", from=1-3, to=2-3]
	\arrow[from=2-1, to=2-2]
	\arrow["v", from=2-1, to=3-1]
	\arrow["u"', from=2-2, to=3-2]
	\arrow[from=2-3, to=2-2]
	\arrow["{\operatorname{id}}", from=2-3, to=3-3]
	\arrow["{0_Y}", from=3-1, to=3-2]
	\arrow[from=3-3, to=3-2]
\end{tikzcd}
    \end{equation}
\end{enumerate}
\end{lem}
\begin{proof}
(i): This follows from Lemma \ref{lem:weilresfunctorial} in which we take the functor $F = i_!$. (ii): Using part (i), the diagram is constructed in the following way: one applies Lemma \ref{lem:deformationspace1} to the morphism $i_!(s): i_!X \to i_!Y$ to obtain a diagram of the form \eqref{eq:deformationdiagram1}, and then one applies the functor $i^*$ to the whole diagram. 
\end{proof}
\egapar\label{paragraph:NormalBundleAnalytic} With notations as in Lemma \ref{lem:deformationAnalytic}. An identical discussion to Remark \ref{rmk:Normalbundle} holds. That is, one has $s \times \operatorname{id} \simeq u \tilde{s}$,  and there is a canonical morphism $p: N_{X/Y} \to X$ over $Y$. 
\begin{lem}
With notations as in Lemma \ref{lem:deformationAnalytic}. If the morphism $s:X \to Y$ admits a cotangent complex $\mathbf{L}_{X/Y}$, then there is a canonical equivalence $N_{X/Y} \simeq V_X^\mathrm{an}(\mathbf{L}_{X/Y}[-1])$.
\end{lem}
\begin{proof}
This is the same, \emph{mutatis mutandis}, as the proof of Lemma \ref{lem:IdentifyNormal}. Alternatively, it can be deduced by applying Lemma \ref{lem:IdentifyNormal} to the morphism $i_!(s): i_!X \to i_!Y$ and then applying the functor $i^*$ to the result, using Lemma \ref{lem:deformationAnalytic}(i) and Lemma \ref{lem:RelativespecLemma1}(v). 
\end{proof}
\begin{lem}
With notations as in Lemma \ref{lem:deformationAnalytic}. If the morphism $s: X \to Y$ is a quasi-smooth Zariski-closed immersion of derived rigid spaces, then $D_{X/Y}$ is representable by a derived rigid space. 
\end{lem}
\begin{proof}
By the compatibility of Weil restrictions with base-change, the question is local on $Y$.
Hence using Lemma \ref{lem:Zariskiclosedquasismooth} we reduce to the case when $Y = \operatorname{Spec}(A)$ and $X = \operatorname{Spec}(B)$ where $B \simeq A/^\mathbf{L}(f_1, \dots f_n)$ for some $f_i: K \to A$. In this case, by Theorem \ref{thm:benbassatextendedRees}, $D_{i_!X/i_!Y}^\mathscr{V} \simeq \operatorname{Spec}(R_{B/A})$ where $R_{B/A} \simeq  A[T, S_i]/^\mathbf{L}(T S_i  - f_i)$, so that one has $D_{i_!X/i_!Y}^\mathscr{V}  \simeq \{0\} \times_{a_!\mathbf{A}^{n, \mathrm{alg}}} (i_!Y \times a_!\mathbf{A}^{n+1, \mathrm{alg}})$ and hence $D_{X/Y} \simeq \{0\} \times_{\mathbf{A}^{n, \mathrm{an}}} (Y \times \mathbf{A}^{n+1, \mathrm{an}})$ is a derived rigid space. 
\end{proof}
\begin{prop}
Let $f:Y \to X$ be a smooth separated morphism of derived rigid spaces which admits a section $s: X \to Y$. (Hence by Lemma \ref{lem:SmoothSection}, $s$ is a quasi-smooth Zariski-closed immersion). Then there is a commutative diagram of derived rigid spaces 
\begin{equation}
\begin{tikzcd}
	{N_{X/Y}} & {D_{X/Y}} & {Y \times \mathbf{G}^{\mathrm{an}}_m} \\
	X & {X \times \mathbf{A}^{1, \mathrm{an}}} & {X \times \mathbf{G}^{\mathrm{an}}_m}
	\arrow[from=1-1, to=1-2]
	\arrow["{p}"', shift right, from=1-1, to=2-1]
	\arrow["{\tilde{f}}"', shift right, from=1-2, to=2-2]
	\arrow[from=1-3, to=1-2]
	\arrow["{f \times \operatorname{id}}"', shift right, from=1-3, to=2-3]
	\arrow["{s_0}"{description}, shift right, curve={height=6pt}, from=2-1, to=1-1]
	\arrow[from=2-1, to=2-2]
	\arrow["{\tilde{s}}"{description}, shift right, curve={height=6pt}, from=2-2, to=1-2]
	\arrow["{s \times \operatorname{id}}"{description}, shift right, curve={height=6pt}, from=2-3, to=1-3]
	\arrow[from=2-3, to=2-2]
\end{tikzcd}
\end{equation}
Here $s_0, \tilde{s}$ are as in Lemma \ref{lem:deformationAnalytic}, $\tilde{f}$ is the composite $(f \times \operatorname{id}) \circ u$, and $p$ is as in Paragraph \ref{paragraph:NormalBundleAnalytic}. Further, the morphism $\tilde{f}$ is smooth, and if $f$ is partially-proper, then $\tilde{f}$ is also partially proper. 
\end{prop}
\begin{proof}
The only thing to establish is that $\tilde{f}$ is smooth. Let $x \in X$ be any point of the underlying topological space and let $x \in U \subseteq X$ be an affinoid open neighbourhood. Put $V := f^{-1}(U)$, then $s$ restricts to $s: U \to V$. Let $s(x) \in V^\prime \subseteq V$ be an affinoid open subset chosen to be small enough so that $\mathbf{L}_{Y/X}$ is free when restricted to $V^\prime$. Let $U^\prime$ be the pullback of $V^\prime$ along $U \to V$. Then $x \in U^\prime$ is an affinoid neighbourhood of $x$. By universal property of the fiber product, $U^\prime \to V^\prime$ factors over $V^{\prime \prime} := V^\prime \times_V f^{-1}(U^\prime) = V^\prime \times_U U^\prime$. We note that $s(x) \in V^{\prime \prime}$ is an affinoid open neighbourhood and $s: X \leftrightarrows Y:f$ restricts to $s: U^\prime \leftrightarrows V^{\prime \prime} : f $. By the compatibility of Weil restrictions with base-change one computes $D_{X/Y} \times_{X} U^\prime  = D_{U^\prime/f^{-1}(U^\prime)}$ and $D_{U^\prime/f^{-1}(U^\prime)} \times_{f^{-1}(U^\prime)} V^{\prime \prime} = D_{U^\prime/V^{\prime \prime}}$, since $U^\prime \to f^{-1}(U^\prime)$ factors over $V^{\prime \prime}$. As $x$ varies, the subspaces $U^\prime$ and $V^{\prime \prime}$ constructed in this way cover $X$ and $s(X)$, respectively. If $W \subseteq Y \setminus s(X)$ is a subspace contained in the complement, then by Lemma \ref{lem:deformationempty} below, $D_{X/Y} \times_Y W = D_{\emptyset/W} = W \times \mathbf{G}^{\mathrm{an}}_m$. Hence to prove the Proposition, we can always replace $Y$ by an open neighbourhood containing $s(X)$. We have effectively reduced to the situation that $X = \operatorname{Spec}(B)$, $Y = \operatorname{Spec}(A)$ are both derived affinoids, $f:Y \to X$ is a smooth morphism with section $s$, and $\mathbf{L}_{Y/X} = \mathbf{L}_{A/B}$ is free, of rank $n$ say. We will now work under these assumptions. 

From $B \to A \to B$ one obtains a fiber sequence $B\hat{\otimes}^\mathbf{L}_A\mathbf{L}_{A/B} \to 0 \to \mathbf{L}_{B/A}$. The connecting homomorphism $d: \mathbf{L}_{B/A}[1] \to B\hat{\otimes}^\mathbf{L}_A\mathbf{L}_{A/B}$ induces on $H^0$ the isomorphism $I/I^2 \xrightarrow[]{\sim} H^0(B) \hat{\otimes}_{H^0(A)}H^0(\mathbf{L}_{A/B})$ induced by $d: I \to H^0(\mathbf{L}_{A/B})$. By Nakayama (see Lemma \ref{lem:Nakayama}), after possibly replacing $Y$ by a rational open subset containing $s(X)$, we may choose $f_i^0 \in I$ whose images $\overline{f^0_i}$ in $I/I^2$ form a basis. Then the $df_i^0$ form a basis for $H^0(\mathbf{L}_{A/B})$ modulo $I$. But by Nakayama's lemma again we see that, after possibly replacing $Y$ by a rational open subset containing $s(X)$, the $df_i^0$ form a basis for $H^0(\mathbf{L}_{A/B})$. If $f_i: K \to A$ are lifts of $f_i^0$, then $df_i: K \to \mathbf{L}_{A/B}$ are lifts of $df_i^0$ and hence form a basis for $\mathbf{L}_{A/B}$. Further, proceeding as in Lemma \ref{lem:Zariskiclosedquasismooth} we see that $A \to B$ factors over an equivalence $A \to A/^\mathbf{L}(f_1,\dots,f_n) \xrightarrow[]{\sim} B$. In other words, we may and will assume the following:
\begin{itemize}
    \item[($\star$)] $A \to B$ is of the form $A \to A/^\mathbf{L}(f_1,\dots,f_n)$ where the $f_i:K \to A$ are such that $df_1,\dots,df_n$ forms a basis for $\mathbf{L}_{A/B}$.
\end{itemize}
Now by Theorem \ref{thm:benbassatextendedRees}, the Rees algebra may be computed as $R_{B/A} = A[T,S_i]/^\mathbf{L}(TS_i-f_i)$. Hence $\mathbf{L}_{R_{B/A}/A[T,S_i]} \simeq \oplus_{i=1}^n R_{B/A}u_i[1]$, here $u_i$ is just a placeholder for the basis element. We consider the sequence $B[T] \to A[T,S_i] \to R_{B/A}$. This induces a fiber sequence 
\begin{equation}
R_{B/A}\hat{\otimes}^\mathbf{L}_{A[T,S_i]} \mathbf{L}_{A[T,S_i]/B[T]}  \to \mathbf{L}_{R_{B/A}/B[T]} \to \mathbf{L}_{R_{B/A}/A[T,S_i]}
\end{equation}
The connecting homomorphism $\oplus_{i=1}^n R_{B/A}u_i \to R_{B/A} \hat{\otimes}^\mathbf{L}_{A[T,S_i]}\mathbf{L}_{A[T,S_i]/B[T]}$ is determined by $u_i \mapsto 1 \otimes (T dS_i - df_i)$. By $(\star)$ the cotangent complex $\mathbf{L}_{A[T,S_i]/B[T]}$ is free of rank $2n$ on (the images of) $dS_i$ and $df_i$. By rotating the triangle, it is then clear that $\mathbf{L}_{R_{B/A}/B[T]}$ is free of rank $n$. In other words, we have computed that the cotangent complex of $D^\mathscr{V}_{i_!X/i_!Y} \to i_!X \times a_!\mathbf{A}^{1, \mathrm{alg}}$ is free of rank $n$. By Lemma \ref{lem:analytificationofcotangent} this implies that the cotangent complex of $D_{X/Y} \to X \times \mathbf{A}^{1, \mathrm{an}}$ is free of rank $n$. To complete the proof we need to show that the classical truncation is smooth. In the factorization $t_!D_{X/Y} \to t_!Y \times \mathbf{A}^{n+1,\mathrm{an}} \to t_!X \times \mathbf{A}^{1, \mathrm{an}}$ the first morphism exhibits $t_!D_{X/Y}$ as the zero-locus of the $TS_i-f_i^0$. Hence smoothness of $t_!D_{X/Y} \to t_!X \times \mathbf{A}^{1, \mathrm{an}}$ may be deduced from the Jacobi criterion \cite[Proposition 2.5]{Lutkebohmertmaximum}.
\end{proof}
\begin{lem}\label{lem:deformationempty}
Let $Y \in \mathrm{dRig}$. Then there is a canonical equivalence $ Y \times \mathbf{G}_m^\mathrm{an} \xrightarrow[]{\sim} D_{\emptyset/Y}$ over $Y \times \mathbf{A}^{1, \mathrm{an}}$. 
\end{lem}
\begin{proof}
The canonical morphism is deduced by adjunction from the empty morphism $\emptyset = 0_Y^\sharp(Y \times \mathbf{G}_m^\mathrm{an}) \to \emptyset$ over $Y$. To check this is an equivalence we reduce to the case when $Y = \operatorname{Spec}(A)$ is affinoid. By Theorem \ref{thm:benbassatextendedRees} we see that $R_{0/A} = A[T,S]/^\mathbf{L}(TS-1)$ so that $D^\mathscr{V}_{\emptyset/i_!Y} = i_!Y \times a_! \mathbf{G}_m^\mathrm{alg}$. Hence $Y \times \mathbf{G}_m^\mathrm{an} \xrightarrow[]{\sim} D_{\emptyset/Y}$.
\end{proof}
\printbibliography

@misc{hekking2021gradedalgebrasprojectivespectra,
      title={Graded algebras, projective spectra and blow-ups in derived algebraic geometry}, 
      author={Jeroen Hekking},
      year={2021},
      eprint={2106.01270},
      archivePrefix={arXiv},
      primaryClass={math.AG},
      url={https://arxiv.org/abs/2106.01270}, 
}

@misc{kedlaya2015reifiedvaluationsadicspectra,
      title={Reified valuations and adic spectra}, 
      author={Kiran S. Kedlaya},
      year={2015},
      eprint={1309.0574},
      archivePrefix={arXiv},
      primaryClass={math.NT},
      url={https://arxiv.org/abs/1309.0574}, 
}

@article{ShiftedSymplectic,
    AUTHOR = {Pantev, Tony and To\"en, Bertrand and Vaqui\'e, Michel and
              Vezzosi, Gabriele},
     TITLE = {Shifted symplectic structures},
   JOURNAL = {Publ. Math. Inst. Hautes \'Etudes Sci.},
  FJOURNAL = {Publications Math\'ematiques. Institut de Hautes \'Etudes
              Scientifiques},
    VOLUME = {117},
      YEAR = {2013},
     PAGES = {271--328},
      ISSN = {0073-8301,1618-1913},
   MRCLASS = {14F05 (14A15 18F20 18G30 53D05 53D12)},
  MRNUMBER = {3090262},
MRREVIEWER = {Andrey\ Yu.\ Lazarev},
       DOI = {10.1007/s10240-013-0054-1},
       URL = {https://doi.org/10.1007/s10240-013-0054-1},
}

@misc{schneider2023reciprocitylaws,
      title={Reciprocity laws for $(\varphi_L,\Gamma_L)$-modules over {L}ubin-{T}ate extensions}, 
      author={Peter Schneider and Otmar Venjakob},
      year={2023},
      eprint={2301.11606},
      archivePrefix={arXiv},
      primaryClass={math.NT},
      url={https://arxiv.org/abs/2301.11606}, 
}

@book{GaitsgoryStudy2,
    AUTHOR = {Gaitsgory, Dennis and Rozenblyum, Nick},
     TITLE = {A study in derived algebraic geometry. {V}olume {II}.
              {D}eformations, {L}ie theory and formal geometry},
    SERIES = {Mathematical Surveys and Monographs},
    VOLUME = {221},
 PUBLISHER = {American Mathematical Society, Providence, RI},
      YEAR = {2017},
     PAGES = {xxxv+436},
      ISBN = {978-1-4704-3570-7},
   MRCLASS = {14F05 (18D05 18G55)},
  MRNUMBER = {3701353},
MRREVIEWER = {Adrian\ Langer},
       DOI = {10.1090/surv/221.2},
       URL = {https://doi.org/10.1090/surv/221.2},
}

@article{Roch,
    AUTHOR = {Roch, G.},
     TITLE = {Ueber die {A}nzahl der willk\"urlichen {C}onstanten in
              algebraischen {F}unctionen},
   JOURNAL = {J. Reine Angew. Math.},
  FJOURNAL = {Journal f\"ur die Reine und Angewandte Mathematik. [Crelle's
              Journal]},
    VOLUME = {64},
      YEAR = {1865},
     PAGES = {372--376},
      ISSN = {0075-4102,1435-5345},
   MRCLASS = {99-04},
  MRNUMBER = {1579304},
       DOI = {10.1515/crll.1865.64.372},
       URL = {https://doi.org/10.1515/crll.1865.64.372},
}

@book{ChevalleyIntroduction,
    AUTHOR = {Chevalley, Claude},
     TITLE = {Introduction to the {T}heory of {A}lgebraic {F}unctions of
              {O}ne {V}ariable},
    SERIES = {Mathematical Surveys},
    VOLUME = {No. VI},
 PUBLISHER = {American Mathematical Society, New York},
      YEAR = {1951},
     PAGES = {xi+188},
   MRCLASS = {14.0X},
  MRNUMBER = {42164},
MRREVIEWER = {O.\ Zariski},
}

@article{HasseDuality,
    AUTHOR = {Hasse, Helmut},
     TITLE = {Theorie der {D}ifferentiale in algebraischen
              {F}unktionenk\"orpern mit vollkommenem {K}onstantenk\"orper},
   JOURNAL = {J. Reine Angew. Math.},
  FJOURNAL = {Journal f\"ur die Reine und Angewandte Mathematik. [Crelle's
              Journal]},
    VOLUME = {172},
      YEAR = {1935},
     PAGES = {55--64},
      ISSN = {0075-4102,1435-5345},
   MRCLASS = {99-04},
  MRNUMBER = {1581436},
       DOI = {10.1515/crll.1935.172.55},
       URL = {https://doi.org/10.1515/crll.1935.172.55},
}

@article{SerreComplex,
    AUTHOR = {Serre, Jean-Pierre},
     TITLE = {Un th\'eor\`eme de dualit\'e},
   JOURNAL = {Comment. Math. Helv.},
  FJOURNAL = {Commentarii Mathematici Helvetici},
    VOLUME = {29},
      YEAR = {1955},
     PAGES = {9--26},
      ISSN = {0010-2571,1420-8946},
   MRCLASS = {56.0X},
  MRNUMBER = {67489},
MRREVIEWER = {S.\ Chern},
       DOI = {10.1007/BF02564268},
       URL = {https://doi.org/10.1007/BF02564268},
}

@inproceedings{SerreICM,
    AUTHOR = {Serre, Jean-Pierre},
     TITLE = {Cohomologie et g\'eom\'etrie alg\'ebrique},
 BOOKTITLE = {Proceedings of the {I}nternational {C}ongress of
              {M}athematicians, 1954, {A}msterdam, vol. {III}},
     PAGES = {515--520},
 PUBLISHER = {Erven P. Noordhoff N. V., Groningen},
      YEAR = {1956},
   MRCLASS = {14.0X},
  MRNUMBER = {87208},
MRREVIEWER = {F.\ Hirzebruch},
}

@book{hartshorne_residues_1966,
	address = {Berlin, Heidelberg},
	series = {Lecture {Notes} in {Mathematics}},
	title = {Residues and {Duality}},
	volume = {20},
	isbn = {978-3-540-03603-6 978-3-540-34794-1},
	url = {http://link.springer.com/10.1007/BFb0080482},
	doi = {10.1007/BFb0080482},
	urldate = {2022-09-18},
	publisher = {Springer},
	author = {Hartshorne, Robin},
	year = {1966},
}

@article{NikolausPrincipal,
    AUTHOR = {Nikolaus, Thomas and Schreiber, Urs and Stevenson, Danny},
     TITLE = {Principal {$\infty$}-bundles: general theory},
   JOURNAL = {J. Homotopy Relat. Struct.},
  FJOURNAL = {Journal of Homotopy and Related Structures},
    VOLUME = {10},
      YEAR = {2015},
    NUMBER = {4},
     PAGES = {749--801},
      ISSN = {2193-8407,1512-2891},
   MRCLASS = {55R99 (18G60 55U35)},
  MRNUMBER = {3423073},
MRREVIEWER = {Timothy\ Porter},
       DOI = {10.1007/s40062-014-0083-6},
       URL = {https://doi.org/10.1007/s40062-014-0083-6},
}

@misc{camargo_notes_2026,
	title = {Notes on {Solid} {Geometry}},
	url = {http://arxiv.org/abs/2603.03012},
	doi = {10.48550/arXiv.2603.03012},
	urldate = {2026-05-26},
	publisher = {arXiv},
	author = {Camargo, Juan Esteban Rodríguez},
	month = mar,
	year = {2026},
	note = {arXiv:2603.03012 [math.AG]},
}

@article {XiaLiuMorphism,
    AUTHOR = {Xia, Mingchen},
     TITLE = {On {L}iu morphisms in non-{A}rchimedean geometry},
   JOURNAL = {Israel J. Math.},
  FJOURNAL = {Israel Journal of Mathematics},
    VOLUME = {255},
      YEAR = {2023},
    NUMBER = {2},
     PAGES = {821--850},
      ISSN = {0021-2172,1565-8511},
   MRCLASS = {14G22 (18F20 32C99 32P05)},
  MRNUMBER = {4619557},
       DOI = {10.1007/s11856-022-2456-6},
       URL = {https://doi.org/10.1007/s11856-022-2456-6},
}

@incollection{BraveNew,
    AUTHOR = {To\"en, Bertrand and Vezzosi, Gabriele},
     TITLE = {Brave new algebraic geometry and global derived moduli spaces
              of ring spectra},
 BOOKTITLE = {Elliptic cohomology},
    SERIES = {London Math. Soc. Lecture Note Ser.},
    VOLUME = {342},
     PAGES = {325--359},
 PUBLISHER = {Cambridge Univ. Press, Cambridge},
      YEAR = {2007},
      ISBN = {978-0-521-70040-5; 0-521-70040-X},
   MRCLASS = {14F20 (14A20 55P43 55U40)},
  MRNUMBER = {2330521},
MRREVIEWER = {Thomas\ H\"uttemann},
       DOI = {10.1017/CBO9780511721489.018},
       URL = {https://doi.org/10.1017/CBO9780511721489.018},
}

@book {BerkeleyLectures,
    AUTHOR = {Scholze, Peter and Weinstein, Jared},
     TITLE = {Berkeley lectures on {$p$}-adic geometry},
    SERIES = {Annals of Mathematics Studies},
    VOLUME = {207},
 PUBLISHER = {Princeton University Press, Princeton, NJ},
      YEAR = {2020},
     PAGES = {x+250},
      ISBN = {978-0-691-20209-9; 978-0-691-20208-2; 978-0-691-20215-0},
   MRCLASS = {14G45 (14A15 14F30 14G22 14G35 14M15)},
  MRNUMBER = {4446467},
MRREVIEWER = {Lance\ Edward\ Miller},
}

@article {BrantnerMathew,
    AUTHOR = {Brantner, D. Lukas B. and Mathew, Akhil},
     TITLE = {Deformation theory and partition {L}ie algebras},
   JOURNAL = {Acta Math.},
  FJOURNAL = {Acta Mathematica},
    VOLUME = {235},
      YEAR = {2025},
    NUMBER = {1},
     PAGES = {1--148},
      ISSN = {0001-5962,1871-2509},
   MRCLASS = {14D06 (17B55 32G13)},
  MRNUMBER = {4995932},
       DOI = {10.4310/acta.2025.v235.n1.a1},
       URL = {https://doi.org/10.4310/acta.2025.v235.n1.a1},
}

@misc{raksit_hochschild_2026,
	title = {Hochschild homology and the derived de {Rham} complex revisited},
	url = {http://arxiv.org/abs/2007.02576},
	doi = {10.48550/arXiv.2007.02576},
	urldate = {2026-04-08},
	publisher = {arXiv},
	author = {Raksit, Arpon},
	month = jan,
	year = {2026},
	note = {arXiv:2007.02576 [math]},
}

@misc{hekking_khan_deformation_2025,
	title = {Deformation to the normal bundle and blow-ups via derived {Weil} restrictions},
	url = {http://arxiv.org/abs/2511.19412},
	doi = {10.48550/arXiv.2511.19412},
	urldate = {2026-03-18},
	publisher = {arXiv},
	author = {Hekking, Jeroen and Khan, Adeel A. and Rydh, David},
	month = nov,
	year = {2025},
	note = {arXiv:2511.19412 [math]},
}

@article {KhanVirtualCartier,
    AUTHOR = {Khan, Adeel A. and Rydh, David},
     TITLE = {Virtual {C}artier divisors and blow-ups},
   JOURNAL = {Selecta Math. (N.S.)},
  FJOURNAL = {Selecta Mathematica. New Series},
    VOLUME = {31},
      YEAR = {2025},
    NUMBER = {4},
     PAGES = {Paper No. 67, 28},
      ISSN = {1022-1824,1420-9020},
   MRCLASS = {14A30 (14A20)},
  MRNUMBER = {4932694},
       DOI = {10.1007/s00029-025-01060-7},
       URL = {https://doi.org/10.1007/s00029-025-01060-7},
}

@article {AnalyticHKR,
    AUTHOR = {Kelly, Jack and Kremnizer, Kobi and Mukherjee, Devarshi},
     TITLE = {An analytic {H}ochschild-{K}ostant-{R}osenberg theorem},
   JOURNAL = {Adv. Math.},
  FJOURNAL = {Advances in Mathematics},
    VOLUME = {410},
      YEAR = {2022},
     PAGES = {Paper No. 108694, 84},
      ISSN = {0001-8708,1090-2082},
   MRCLASS = {14A30},
  MRNUMBER = {4491247},
MRREVIEWER = {Ralf\ Meyer},
       DOI = {10.1016/j.aim.2022.108694},
       URL = {https://doi.org/10.1016/j.aim.2022.108694},
}

@phdthesis{SoorThesis,
    AUTHOR = {Soor, Arun},
     TITLE = {Topics in derived analytic geometry},
SCHOOL = {Mathematical Institute, University of Oxford},
      YEAR = {2025},
    MONTH = {August},
}

@misc{kelly_localising_2025,
	title = {Localising invariants in derived bornological geometry},
	url = {http://arxiv.org/abs/2505.15750},
	doi = {10.48550/arXiv.2505.15750},
	urldate = {2026-03-14},
	publisher = {arXiv},
	author = {Kelly, Jack and Mukherjee, Devarshi},
	month = oct,
	year = {2025},
	note = {arXiv:2505.15750 [math]},
}

@article{Lutkebohmert1990,
author = {Lütkebohmert, Werner},
journal = {Mathematische Annalen},
number = {1-3},
pages = {341-372},
title = {Formal-algebraic and rigid-analytic geometry.},
url = {http://eudml.org/doc/164642},
volume = {286},
year = {1990},
}

@book{SGA4-1,
     TITLE = {Th\'eorie des topos et cohomologie \'etale des sch\'emas.
              {T}ome 1: {T}h\'eorie des topos},
    SERIES = {Lecture Notes in Mathematics},
    VOLUME = {269},
      NOTE = {S\'eminaire de G\'eom\'etrie Alg\'ebrique du Bois-Marie
              1963--1964 (SGA 4)},
 PUBLISHER = {Springer-Verlag, Berlin-New York},
        editor = {Artin, Michael and Grothendieck, Alexander and Verdier, J. L. and Bourbaki, N. and Deligne, P. and Saint-Donat, Bernard},
      YEAR = {1972},
     PAGES = {xix+525},
     shorthand = {SGA4-1},
   MRCLASS = {14-06},
  MRNUMBER = {354652},
}

@article {BenBassatHekkingBlowUp,
    AUTHOR = {Ben-Bassat, Oren and Hekking, Jeroen},
     TITLE = {Blow-ups and normal bundles in connective and nonconnective
              derived geometries},
   JOURNAL = {Adv. Math.},
  FJOURNAL = {Advances in Mathematics},
    VOLUME = {480},
      YEAR = {2025},
     PAGES = {Paper No. 110530, 59},
      ISSN = {0001-8708,1090-2082},
   MRCLASS = {14A30 (14D23 14F08 14N35)},
  MRNUMBER = {4959048},
       DOI = {10.1016/j.aim.2025.110530},
       URL = {https://doi.org/10.1016/j.aim.2025.110530},
}

@article {ToenChampsAffines,
    AUTHOR = {To{\"e}n, Bertrand},
     TITLE = {Champs affines},
   JOURNAL = {Selecta Math. (N.S.)},
  FJOURNAL = {Selecta Mathematica. New Series},
    VOLUME = {12},
      YEAR = {2006},
    NUMBER = {1},
     PAGES = {39--135},
      ISSN = {1022-1824,1420-9020},
   MRCLASS = {14F35 (14A20 18F10 55U35)},
  MRNUMBER = {2244263},
MRREVIEWER = {Mark\ Hovey},
       DOI = {10.1007/s00029-006-0019-z},
       URL = {https://doi.org/10.1007/s00029-006-0019-z},
}

@misc{DAnG,
	title = {A {Perspective} on the {Foundations} of {Derived} {Analytic} {Geometry}},
	url = {http://arxiv.org/abs/2405.07936},
	doi = {10.48550/arXiv.2405.07936},
	urldate = {2024-07-01},
	publisher = {arXiv},
	author = {Ben-Bassat, Oren and Kelly, Jack and Kremnizer, Kobi},
	month = may,
	year = {2024},
	note = {arXiv:2405.07936 [math]},
}

@article{huber_continuous_1993,
	title = {Continuous valuations},
	volume = {212},
	issn = {0025-5874},
	url = {https://doi.org/10.1007/BF02571668},
	doi = {10.1007/BF02571668},
	number = {3},
	journal = {Mathematische Zeitschrift},
	author = {Huber, Roland},
	year = {1993},
	mrnumber = {1207303},
	pages = {455--477},
}

@Book{JohnstoneStone,
 Author = {Johnstone, Peter T.},
 Title = {Stone spaces},
 FSeries = {Cambridge Studies in Advanced Mathematics},
 Series = {Camb. Stud. Adv. Math.},
 Volume = {3},
 Year = {1986},
 Publisher = {Cambridge University Press, Cambridge},
 Language = {English},
 zbMATH = {3940199},
 Zbl = {0586.54001}
}

@misc{liu_enhanced_2017,
	title = {Enhanced six operations and base change theorem for higher {Artin} stacks},
	url = {http://arxiv.org/abs/1211.5948},
	doi = {10.48550/arXiv.1211.5948},
	urldate = {2024-06-13},
	publisher = {arXiv},
	author = {Liu, Yifeng and Zheng, Weizhe},
	month = sep,
	year = {2017},
	note = {arXiv:1211.5948 [math]},
}

@article{MathewGalois,
    AUTHOR = {Mathew, Akhil},
     TITLE = {The {G}alois group of a stable homotopy theory},
   JOURNAL = {Adv. Math.},
  FJOURNAL = {Advances in Mathematics},
    VOLUME = {291},
      YEAR = {2016},
     PAGES = {403--541},
      ISSN = {0001-8708,1090-2082},
   MRCLASS = {14F35 (11F11 11S20 14F20 18D10 18G55 55P43 55U35)},
  MRNUMBER = {3459022},
MRREVIEWER = {Rui\ Miguel\ Saramago},
       DOI = {10.1016/j.aim.2015.12.017},
       URL = {https://doi.org/10.1016/j.aim.2015.12.017},
}

@misc{heyer_6-functor_2024,
	title = {6-{Functor} {Formalisms} and {Smooth} {Representations}},
	url = {http://arxiv.org/abs/2410.13038},
	doi = {10.48550/arXiv.2410.13038},
	urldate = {2025-01-21},
	publisher = {arXiv},
	author = {Heyer, Claudius and Mann, Lucas},
	month = oct,
	year = {2024},
	note = {arXiv:2410.13038 [math]},
}

@Book{GaitsgoryStudy1,
 Author = {Gaitsgory, Dennis and Rozenblyum, Nick},
 Title = {A study in derived algebraic geometry. {Volume} {I}: {Correspondences} and duality},
 FSeries = {Mathematical Surveys and Monographs},
 Series = {Mathematical Surveys and Monographs},
 ISSN = {0076-5376},
 Volume = {221},
 ISBN = {978-1-4704-3569-1; 978-1-4704-3568-4; 978-1-4704-4085-5},
 Year = {2017},
 Publisher = {Providence, RI: American Mathematical Society},
 Language = {English},
 zbMATH = {6780490},
 Zbl = {1408.14001}
}

@misc{zavyalov_poincare_2023,
      title={Poincar\'e Duality in abstract 6-functor formalisms}, 
      author={Bogdan Zavyalov},
      year={2026},
      eprint={2301.03821},
      archivePrefix={arXiv},
      primaryClass={math.AG},
      url={https://arxiv.org/abs/2301.03821}, 
}

@article{soor_quasicoherent_2023,
	title = {Quasicoherent sheaves for dagger analytic geometry},
	  journal = {Israel J. Math.},
	author = {Soor, Arun},
	month = nov,
	year = {2023},
    note = {to appear},
}

@incollection {ChiarellottoDuality,
    AUTHOR = {Chiarellotto, Bruno},
     TITLE = {Duality in rigid analysis},
 BOOKTITLE = {{$p$}-adic analysis ({T}rento, 1989)},
    SERIES = {Lecture Notes in Math.},
    VOLUME = {1454},
     PAGES = {142--172},
 PUBLISHER = {Springer, Berlin},
      YEAR = {1990},
      ISBN = {3-540-53477-6},
   MRCLASS = {32C37 (32E10 32P05)},
  MRNUMBER = {1094850},
MRREVIEWER = {Peter\ Ullrich},
       DOI = {10.1007/BFb0091137},
       URL = {https://doi.org/10.1007/BFb0091137},
}

@article {BeyerSerre,
    AUTHOR = {Beyer, Peter},
     TITLE = {On {S}erre-duality for coherent sheaves on rigid-analytic
              spaces},
   JOURNAL = {Manuscripta Math.},
  FJOURNAL = {Manuscripta Mathematica},
    VOLUME = {93},
      YEAR = {1997},
    NUMBER = {2},
     PAGES = {219--245},
}

@UNPUBLISHED{ScholzeSixFunctors,
    AUTHOR = "Scholze, Peter",
    TITLE = "{S}ix-{F}unctor {F}ormalisms",
    NOTE = "Available at \url{https://people.mpim-bonn.mpg.de/scholze/SixFunctors.pdf}",
    PAGES = "90",
    URL = "https://people.mpim-bonn.mpg.de/scholze/SixFunctors.pdf",
    YEAR = "2022",
}

@UNPUBLISHED{CondensedComplexGeometry,
    AUTHOR = "Clausen, Dustin and Scholze, Peter",
    TITLE = "Condensed {M}athematics and {C}omplex {G}eometry",
    NOTE = "Available at \url{https://people.mpim-bonn.mpg.de/scholze/Complex.pdf}",
    PAGES = "148",
    URL = "https://people.mpim-bonn.mpg.de/scholze/Complex.pdf",
    YEAR = "2022",
}

@Article{Lutkebohmertmaximum,
 Author = {Bosch, Siegfried and L{\"u}tkebohmert, Werner and Raynaud, Michel},
 Title = {Formal and rigid geometry. {III}: {The} relative maximum principle},
 FJournal = {Mathematische Annalen},
 Journal = {Math. Ann.},
 ISSN = {0025-5831},
 Volume = {302},
 Number = {1},
 Pages = {1--29},
 Year = {1995},
 Language = {English},
 DOI = {10.1007/BF01444485},
 zbMATH = {769467},
 Zbl = {0839.14013}
}

@Article{vanderPutetale,
 Author = {de Jong, Johan and van der Put, Marius},
 Title = {{\'E}tale cohomology of rigid analytic spaces},
 FJournal = {Documenta Mathematica},
 Journal = {Doc. Math.},
 ISSN = {1431-0635},
 Volume = {1},
 Pages = {1--56},
 Year = {1996},
 Language = {English},
 zbMATH = {897856},
 Zbl = {0922.14012}
}

@unpublished{SpectralAlgebraicGeometry,
    author = {Jacob Lurie},
    title = {Spectral Algebraic Geometry},
    note = {Available at the author's webpage \url{https://www.math.ias.edu/~lurie/}},
year= {2018},
}

@book {BGR,
    AUTHOR = {Bosch, Siegfried and G\"{u}ntzer, Ulrich and Remmert, Reinhold},
     TITLE = {Non-{A}rchimedean analysis},
    SERIES = {Grundlehren der mathematischen Wissenschaften [Fundamental
              Principles of Mathematical Sciences]},
    VOLUME = {261},
      NOTE = {A systematic approach to rigid analytic geometry},
 PUBLISHER = {Springer-Verlag, Berlin},
      YEAR = {1984},
     PAGES = {xii+436},
      ISBN = {3-540-12546-9},
   MRCLASS = {32K10 (30G05 46P05)},
  MRNUMBER = {746961},
MRREVIEWER = {W.\ Bartenwerfer},
       DOI = {10.1007/978-3-642-52229-1},
       URL = {https://doi.org/10.1007/978-3-642-52229-1},
}

@book {HuberEtale,
    AUTHOR = {Huber, Roland},
     TITLE = {\'{E}tale cohomology of rigid analytic varieties and adic
              spaces},
    SERIES = {Aspects of Mathematics},
    VOLUME = {E30},
 PUBLISHER = {Friedr. Vieweg \& Sohn, Braunschweig},
      YEAR = {1996},
     PAGES = {x+450},
      ISBN = {3-528-06794-2},
   MRCLASS = {14G22 (14F20)},
  MRNUMBER = {1734903},
MRREVIEWER = {Lorenzo\ Ramero},
       DOI = {10.1007/978-3-663-09991-8},
       URL = {https://doi.org/10.1007/978-3-663-09991-8},
}

@article {vanderPutSerre,
    AUTHOR = {van der Put, Marius},
     TITLE = {Serre duality for rigid analytic spaces},
   JOURNAL = {Indag. Math. (N.S.)},
  FJOURNAL = {Koninklijke Nederlandse Akademie van Wetenschappen.
              Indagationes Mathematicae. New Series},
    VOLUME = {3},
      YEAR = {1992},
    NUMBER = {2},
     PAGES = {219--235},
}

@book{toen_homotopical_2008,
	series = {Mem. {Am}. {Math}. {Soc}.},
	title = {Homotopical algebraic geometry. {II}: {Geometric} stacks and applications},
	volume = {902},
	isbn = {978-0-8218-4099-3 978-1-4704-0508-3},
	language = {English},
	publisher = {Providence, RI: American Mathematical Society (AMS)},
	author = {Toën, Bertrand and Vezzosi, Gabriele},
	year = {2008},
	doi = {10.1090/memo/0902},
	note = {ISSN: 0065-9266},
}

@book {HigherToposTheory,
    AUTHOR = {Lurie, Jacob},
     TITLE = {Higher topos theory},
    SERIES = {Annals of Mathematics Studies},
    VOLUME = {170},
 PUBLISHER = {Princeton University Press, Princeton, NJ},
      YEAR = {2009},
     PAGES = {xviii+925},
      ISBN = {978-0-691-14049-0; 0-691-14049-9},
   MRCLASS = {18-02 (18B25 18E35 18G30 18G55 55U40)},
  MRNUMBER = {2522659},
MRREVIEWER = {Mark\ Hovey},
       DOI = {10.1515/9781400830558},
       URL = {https://doi.org/10.1515/9781400830558},
}

@UNPUBLISHED{CondensedMathematics,
    AUTHOR = "Clausen, Dustin and Scholze, Peter",
    TITLE = "Condensed {M}athematics",
    NOTE = "Available at \url{https://www.math.uni-bonn.de/people/scholze/Condensed.pdf}",
    PAGES = "77",
    URL = "https://www.math.uni-bonn.de/people/scholze/Condensed.pdf",
    YEAR = "2019",
}

@article {BBKNonArch,
    AUTHOR = {Ben-Bassat, Oren and Kremnizer, Kobi},
     TITLE = {Non-archimedean analytic geometry as relative algebraic
              geometry},
   JOURNAL = {Ann. Fac. Sci. Toulouse Math. (6)},
  FJOURNAL = {Annales de la Facult\'{e} des Sciences de Toulouse.
              Math\'{e}matiques. S\'{e}rie 6},
    VOLUME = {26},
      YEAR = {2017},
    NUMBER = {1},
}

@misc{kerodon,
    title        = {Kerodon},
    author       = {Jacob Lurie},
    howpublished = {\url{https://kerodon.net}},
    year         = {2018},
  }

@article{GrosseKlonneRigid,
 Author = {Grosse-Kl{\"o}nne, Elmar},
 Title = {Rigid analytic spaces with overconvergent structure sheaf},
 FJournal = {Journal f{\"u}r die Reine und Angewandte Mathematik},
 Journal = {J. Reine Angew. Math.},
 ISSN = {0075-4102},
 Volume = {519},
 Pages = {73--95},
 Year = {2000},
 Language = {English},
 DOI = {10.1515/crll.2000.018},
 zbMATH = {1414388},
 Zbl = {0945.14013}
}

@article{schneiders_quasi-abelian_1999,
	title = {Quasi-abelian categories and sheaves},
	volume = {76},
	issn = {0249-633X},
	url = {https://eudml.org/doc/94927},
	language = {eng},
	urldate = {2022-12-21},
	journal = {Mémoires de la Société Mathématique de France},
	author = {Schneiders, Jean-Pierre},
	year = {1999},
	note = {Publisher: Société mathématique de France},
	pages = {III1--VI140},
}

@article{emerton_interpolation_2006,
	title = {On the interpolation of systems of eigenvalues attached to automorphic {Hecke} eigenforms},
	volume = {164},
	issn = {1432-1297},
	url = {https://doi.org/10.1007/s00222-005-0448-x},
	doi = {10.1007/s00222-005-0448-x},
	number = {1},
	journal = {Inventiones mathematicae},
	author = {Emerton, Matthew},
	month = apr,
	year = {2006},
	pages = {1--84},
}

@article{pan_locally_2022-1,
	title = {On locally analytic vectors of the completed cohomology of modular curves},
	volume = {10},
	doi = {10.1017/fmp.2022.1},
	journal = {Forum of Mathematics, Pi},
	author = {Pan, Lue},
	year = {2022},
	note = {Publisher: Cambridge University Press},
	pages = {e7},
}

@article {Equivariant ,
    AUTHOR = {Ardakov, Konstantin},
     TITLE = {Equivariant {$\wideparen{\mathcal D}$}-modules on rigid analytic spaces},
   JOURNAL = {Ast\'{e}risque},
  FJOURNAL = {Ast\'{e}risque},
    NUMBER = {423},
      YEAR = {2021},
     PAGES = {161},
      ISSN = {0303-1179},
      ISBN = {978-2-85629-936-4},
   MRCLASS = {14F10 (14G22 32C38)},
  MRNUMBER = {4234536},
       DOI = {10.24033/ast},
       URL = {https://doi.org/10.24033/ast},
}

\Addresses
\end{document}